\documentclass[12pt]{amsart}

\usepackage{ifthen}
\usepackage[utf8]{inputenc}
\usepackage{setspace}
\usepackage{enumerate}
\usepackage{thmtools}
\usepackage[shortlabels]{enumitem}
\usepackage{mathtools}
\usepackage{verbatim}
\usepackage{comment}
\usepackage{multirow}

\usepackage{tikzit}

\tikzstyle{node small}=[fill=none, draw=none, shape=circle, inner sep=0.2]
\tikzstyle{bullet}=[fill=black, draw=black, shape=circle, scale=0.3]
\tikzstyle{bullet small}=[fill=black, draw=black, shape=circle, scale=0.14]
\tikzstyle{bullet mid}=[fill=black, draw=black, shape=circle, scale=0.2]
\tikzstyle{0.8}=[fill=none, draw=none, scale=0.8]
\tikzstyle{0.6}=[fill=none, draw=none, scale=0.6]

\tikzstyle{red}=[draw=red, ->]
\tikzstyle{red 1}=[draw=red, <-]
\tikzstyle{black arrow}=[->]
\tikzstyle{dashed}=[-, draw=black, densely dotted]
\tikzstyle{Red-Thick}=[-, draw=red, thick]
\tikzstyle{spin}=[-, postaction={{decorate}}]
\tikzstyle{Blue-Thick}=[-, fill=none, draw=blue, thick]
\tikzstyle{extend-edge}=[-, dashed]
\tikzstyle{blue-double}=[-, draw=blue, thick, double]

\usepackage{tikz}
\usetikzlibrary{shapes.geometric,intersections,decorations.markings,snakes}
\usetikzlibrary{calc,intersections,through,backgrounds}
\usetikzlibrary{patterns}

\newcommand{\tei}[0]{{Teichm\"uller}}

\DeclareMathOperator{\osp}{OSp}
\DeclareMathOperator{\wt}{wt}

\DeclareMathOperator{\ber}{Ber}
\DeclareMathOperator{\st}{st}
\DeclareMathOperator{\tp}{t}
\DeclareMathOperator{\id}{id}
\DeclareMathOperator{\PSL}{PSL}

\newcommand{\ospmatrix}[9]{
\left(\begin{array}{cc|c}
#1&#2&#3\\#4&#5& #6 
\\
\hline
#7 &#8 &#9
\end{array}
\right)
}

\newcommand{\calA}[0]{\mathcal{A}}

\usetikzlibrary{arrows}
\tikzset {->-/.style={decoration={markings, mark=at position .5 with {\arrow{latex}}}, postaction={decorate}}}
\tikzset {-->-/.style={decoration={markings, mark=at position .5 with {\arrow[scale=2]{latex}}}, postaction={decorate}}}
\newcommand{\midarrow}{\tikz \draw[-triangle 90] (0,0) -- +(.1,0);}

\newcommand{\midrevarrow}{\tikz \draw[-triangle 90] (0,0) -- +(-.1,0);}
\usepackage{comment}
\usepackage{graphicx}
\usepackage{color}
\usepackage{etoolbox}
\usepackage[margin=1in]{geometry}
\usepackage{amsmath,amsthm,amssymb,graphicx,tikz,tikz-cd}
\usepackage{multicol}
\usepackage{hyperref}
\hypersetup{
    colorlinks = true,
    citecolor = orange,%
}
\usepackage[noabbrev]{cleveref}

\makeatletter
\patchcmd{\@settitle}{\uppercasenonmath\@title}{}{}{}
\patchcmd{\@setauthors}{\MakeUppercase}{}{}{}
\patchcmd{\section}{\scshape}{}{}{}
\makeatother
\makeatletter
\patchcmd{\@maketitle}
  {\ifx\@empty\@dedicatory}
  {\ifx\@empty\@date \else {\vskip2ex 
  \centering\footnotesize\@date\par\vskip1ex}\fi
   \ifx\@empty\@dedicatory}
  {}{}
\patchcmd{\@adminfootnotes}
  {\ifx\@empty\@date\else \@footnotetext{\@setdate}\fi}
  {}{}{}
\makeatother

\renewcommand{\l}[2]{\lambda_{#1,#2}}
\newcommand{\tu}[3]{\bigtriangleup^{#1}_{#2,#3}}
\newcommand{\td}[3]{\bigtriangledown^{#1}_{#2,#3}}
\newcommand{\Tu}[0]{\bigtriangleup}
\newcommand{\Td}[0]{\bigtriangledown}

\theoremstyle{plain}
\newtheorem{theorem}{Theorem}[section]

\newtheorem{lemma}[theorem]{Lemma}
\newtheorem{prop}[theorem]{Proposition}

\newtheorem{corollary}[theorem]{Corollary}

\theoremstyle{definition}
\newtheorem{remark}[theorem]{Remark}

\newtheorem{definition}[theorem]{Definition}

\title[]{Matrix Formulae for Decorated Super Teichm\"{u}ller Spaces}

\author{Gregg Musiker}
\author{Nicholas Ovenhouse}
\author{Sylvester W. Zhang}

\thanks{School of Mathematics, University of Minnesota, Minneapolis, MN 55455, USA}
\thanks{\emph{Email}:
\href{mailto:musiker@umn.edu}{musiker@umn.edu},
\href{mailto:ovenh001@umn.edu}{ovenh001@umn.edu},
\href{mailto:swzhang@umn.edu}{swzhang@umn.edu}}
\date{}

\begin{document}

\begin{abstract}
    For an arc on a bordered surface with marked points, we associate a holonomy matrix using a product of elements of the supergroup $\osp(1|2)$,
    which defines a flat $\osp(1|2)$-connection on the surface. We show that our matrix formulas of an arc yields its super $\lambda$-length in Penner-Zeitlin's 
    decorated super Teichm\"uller space. This generalizes the matrix formulas of Fock-Goncharov and Musiker-Williams.
    We also prove that our matrix formulas agree with the combinatorial formulas given in the authors' 
    previous works. As an application, we use our matrix formula in the case of an annulus to obtain new results on super Fibonacci numbers. 
    
\end{abstract}

\maketitle
\setcounter{tocdepth}{1}
\tableofcontents
\setlength{\parindent}{0em}
\setlength{\parskip}{0.5em}

\tableofcontents
\section*{Introduction}

Cluster algebras, first introduced by Fomin and Zelevinsky \cite{fz_02}, are certain commutative algebras posessing additional combinatorial structures. 
Since their discovery, cluster algebras have been connected to many other areas of mathematics and physics such as representation theory, integrable systems, 
Teichm\"{u}ller theory and string theory. In recent years, much progress have been made towards a theory of super-commutative cluster algebras, 
such as \cite{o_15,os_18}, \cite{lmrs_21}, \cite{sv09,shemyakova2022super} and \cite{moz21,moz22}. The current authors, in our previous two papers \cite{moz21, moz22},
began the project of exploring a possible super cluster algebraic interpretation of Penner-Zeitlin's decorated super Teichm\"{u}ller theory,
generalizing the known cluster structure of Penner's $\lambda$-length coordinates. This paper is the third in this series,
and as such we will use several conventions and definitions from our previous works, citing them where appropriate.

For a triangulation $T$ of a marked surface $(S,M)$, we define a graph $\Gamma_T$ embedded in $S$ and for each point of the decorated super Teichm\"{u}ller space,
a flat $\osp(1|2)$-connection on $\Gamma_T$. A similar construction was given in \cite{bouschbacher_13} of a graph connection in terms of shear coordinates
on the (un-decorated) super Teichm\"{u}ller space%
\footnote{In the language of the cluster algebra literature, shear coordinates are $\mathcal{X}$-type cluster variables, 
while $\lambda$-lengths are $\mathcal{A}$-type cluster variables.}.
We then define certain canonical paths on $\Gamma_T$ for each arc $(a,b)$ of $S$,
thereby associating the arc with a holonomy matrix $H_{a,b}$. Our main result (\Cref{thm:12-entry}) is that for a polygon (i.e. a marked disk)
the $(1,2)$-entry of the holonomy matrix is the super $\lambda$-length, up to sign. We also give precise formulas for all entries of
the holonomy matrices in terms of super $\lambda$-lengths and $\mu$-invariants (\Cref{thm:generic}), and
we give combinatorial interpretations of these matrix entries (\Cref{thm:combo}) as generating functions for double dimer covers in the spirit of \cite{mw13}.

The structure of the paper is as follows. 
In \Cref{sec:background}, we review background on the decorated super {\tei} theory of \cite{pz_19}, 
and recall some conventions in our previous papers \cite{moz21,moz22}. 
In \Cref{sec:osp}, we provide necessary information on super-matrices and the ortho-symplectic group $\osp(1|2)$. 
In \Cref{sec:Mpaths}, we define the graph $\Gamma_T$ and a flat $\osp(1|2)$\nobreakdash-connection on it, and state our main theorem. 
A more detailed version of our matrix formulas can be found in \Cref{sec:Hproof} which is devoted to a proof of the main theorem. 
In \Cref{sec:DD} we give a combinatorial interpretation of the result of this paper using double dimer covers, connecting the main results of the current paper and those of \cite{moz22}.  We revisit the \emph{super Fibonacci numbers} studied in \cite{moz22}, and examine the corresponding holonomy matrices in \Cref{sec:Fib}.  Finally, in \Cref{sec:appendix}, we provide a more geometric interpretation of our matrix formula, via a different viewpoint of the decorated super {\tei} theory.

\section{Background on Decorated Super Teichm\"{u}ller Theory} \label{sec:background}

In this section, we briefly recall the basic definitions of the decorated super Teichm\"uller space of a polygon (see \cite{pz_19} and \cite{moz21} for more details).
Consider a polygon $P$ (i.e. a disc with marked points on its boundary), a fixed triangulation $T$, and a choice of orientation of the edges of $T$.

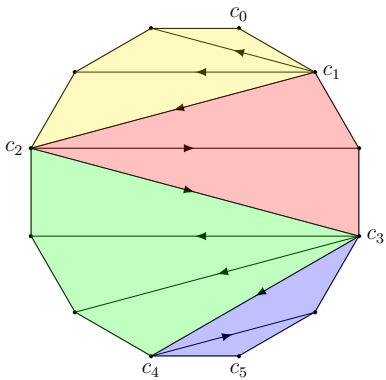
\begin{figure}[h!]
\centering
\scalebox{0.7}{
\begin{tikzpicture}[decoration={
    markings,
    mark=at position 0.5 with {\arrow[scale=1.5]{latex}}},
    scale=0.9] 

\tikzstyle{every path}=[draw] 
		\path
    node[
      regular polygon,
      regular polygon sides=12,
      draw,
      inner sep=2.2cm,
    ] (hexagon) {}
    %
    (hexagon.corner 1) node[above] {$c_0$}
    (hexagon.corner 2) node[above] {}
    (hexagon.corner 3) node[left] {}
    (hexagon.corner 4) node[left] {$c_2$}
    (hexagon.corner 5) node[below] {}
    (hexagon.corner 6) node[right] {}
    (hexagon.corner 7) node[below] {$c_4$}
    (hexagon.corner 8) node[below] {$c_5$}
    (hexagon.corner 9) node[left] {}
    (hexagon.corner 10) node[right] {$c_3$}
    (hexagon.corner 11) node[below] {}
    (hexagon.corner 12) node[right] {$c_1$}
;

\foreach \x in {1,2,...,12}{
\draw (hexagon.corner \x) node [fill,circle,scale=0.2] {};}

\draw[postaction={decorate}] (hexagon.corner 12)--(hexagon.corner 4);
\draw[postaction={decorate}] (hexagon.corner 12)--(hexagon.corner 2);
\draw[postaction={decorate}] (hexagon.corner 12)--(hexagon.corner 3);

\draw[postaction={decorate}](hexagon.corner 4)--(hexagon.corner 10);
\draw[postaction={decorate}](hexagon.corner 4)--(hexagon.corner 11);

\draw[postaction={decorate}](hexagon.corner 10)--(hexagon.corner 7);
\draw[postaction={decorate}](hexagon.corner 10)--(hexagon.corner 5);
\draw[postaction={decorate}](hexagon.corner 10)--(hexagon.corner 6);

\draw[postaction={decorate}](hexagon.corner 7)--(hexagon.corner 9);

\draw[fill=yellow, nearly transparent] (hexagon.corner 12)--(hexagon.corner 4)--(hexagon.corner 3)--(hexagon.corner 2)--(hexagon.corner 1)--cycle;
\draw[fill=red, nearly transparent] (hexagon.corner 12)--(hexagon.corner 4)--(hexagon.corner 10)--(hexagon.corner 11)--cycle;

\draw[fill=green, nearly transparent] (hexagon.corner 10)-- (hexagon.corner 4)--(hexagon.corner 5)--(hexagon.corner 6)--(hexagon.corner 7)--cycle;
\draw[fill=blue, nearly transparent] (hexagon.corner 10)--(hexagon.corner 7)--(hexagon.corner 8)--(hexagon.corner 9)--cycle;

\coordinate (m 1) at  ($0.38*(hexagon.corner 2)+0.38*(hexagon.corner 1)+0.24*(hexagon.corner 12)$);
\coordinate (m 2) at  ($0.28*(hexagon.corner 2)+0.28*(hexagon.corner 3)+0.44*(hexagon.corner 12)$);
\coordinate (m 3) at  ($0.24*(hexagon.corner 4)+0.24*(hexagon.corner 3)+0.52*(hexagon.corner 12)$);
\coordinate (m 4) at  ($0.52*(hexagon.corner 4)+0.24*(hexagon.corner 11)+0.24*(hexagon.corner 12)$);
\coordinate (m 5) at  ($0.56*(hexagon.corner 4)+0.22*(hexagon.corner 11)+0.22*(hexagon.corner 10)$);
\coordinate (m 6) at  ($0.27*(hexagon.corner 4)+0.27*(hexagon.corner 5)+0.46*(hexagon.corner 10)$);
\coordinate (m 7) at  ($0.23*(hexagon.corner 6)+0.23*(hexagon.corner 5)+0.44*(hexagon.corner 10)$);
\coordinate (m 8) at  ($0.26*(hexagon.corner 6)+0.26*(hexagon.corner 7)+0.48*(hexagon.corner 10)$);
\coordinate (m 9) at  ($0.25*(hexagon.corner 9)+0.5*(hexagon.corner 7)+0.25*(hexagon.corner 10)$);
\coordinate (m 10) at  ($0.36*(hexagon.corner 9)+0.24*(hexagon.corner 7)+0.4*(hexagon.corner 8)$);
\end{tikzpicture}
}
\caption{The default orientation of a generic triangulation where each fan segment is colored differently.}
\label{fig:default_spin}
\end{figure}

For simplicity, we will always consider a ``\emph{default orientation}'', as defined in \cite{moz21}, which is pictured in \Cref{fig:default_spin}. 
The maximal groupings of consecutive triangles which share a common vertex (indicated by different colors) are called ``\emph{fan segments}'', and
the common vertex they share is called the ``\emph{fan center}'' (vertices labelled $c_1,c_2, \dots, c_N$ from top to bottom). The default orientation is defined so that the edges connecting
fan centers are oriented $c_1 \to c_2 \to c_3 \to \cdots \to c_N$, and the remaining edges are oriented \emph{away} from the fan centers.

The \emph{decorated super Teichm\"uller space} of $P$ is a super-commutative algebra\footnote{Technically, each choice of spin structure (represented by a choice of orientation
of the triangulation) corresponds to a different connected component of the space.} $\mathcal{A}$ with the following generators: 
for each edge in $T$ with endpoints $i$ and $j$, an even generator $\lambda_{ij}$
(called a ``$\lambda$-length''), 
and for each triangle in $T$ with vertices $i,j,k$, an odd generator $\boxed{ijk}$ (called a ``$\mu$-invariant'')%
\footnote{The algebra $\mathcal{A}$ is technically the tensor product of the field of rational functions in the square roots of the $\lambda$-lengths and
the exterior algebra generated by $\mu$-invariants.}.

When two triangulations are related by a flip, as in \Cref{fig:super_ptolemy}, one can define new elements of the algebra by the following ``\emph{super Ptolemy relations}'':
\begin{align}
    ef      &= ac+bd+\sqrt{abcd} \, \sigma\theta \label{eqn:super_ptolemy_lambda} \\
    \sigma' &= \frac{\sigma\sqrt{bd}-\theta\sqrt{ac}}{\sqrt{ac+bd}} = \frac{\sigma\sqrt{bd}-\theta\sqrt{ac}}{\sqrt{ef}}   \label{eqn:super_ptolemy_mu_right} \\
    \theta' &= \frac{\theta\sqrt{bd}+\sigma \sqrt{ac}}{\sqrt{ac+bd}} = \frac{\theta\sqrt{bd}+\sigma \sqrt{ac}}{\sqrt{ef}} \label{eqn:super_ptolemy_mu_left}
\end{align}

\begin{figure}[h!]
\centering
\begin{tikzpicture}[scale=0.6, baseline, thick]
    \draw (0,0)--(3,0)--(60:3)--cycle;
    \draw (0,0)--(3,0)--(-60:3)--cycle;

    \draw (0,0)--node {\midarrow} (3,0);

    \draw node[above]      at (70:1.5){$a$};
    \draw node[above]      at (30:2.8){$b$};
    \draw node[below]      at (-30:2.8){$c$};
    \draw node[below=-0.1] at (-70:1.5){$d$};
    \draw node[above] at (1,-0.12){$e$};

    \draw node[left] at (0,0) {$i$};
    \draw node[above] at (60:3) {$\ell$};
    \draw node[right] at (3,0) {$k$};
    \draw node[below] at (-60:3) {$j$};

    \draw node at (1.5,1){$\theta$};
    \draw node at (1.5,-1){$\sigma$};
\end{tikzpicture}
\begin{tikzpicture}[baseline]
    \draw[->, thick](0,0)--(1,0);
    \node[above]  at (0.5,0) {};
\end{tikzpicture}
\begin{tikzpicture}[scale=0.6, baseline, thick,every node/.style={sloped,allow upside down}]
    \draw (0,0)--(60:3)--(-60:3)--cycle;
    \draw (3,0)--(60:3)--(-60:3)--cycle;

    \draw node[above]      at (70:1.5)  {$a$};
    \draw node[above]      at (30:2.8)  {$b$};
    \draw node[below]      at (-30:2.8) {$c$};
    \draw node[below=-0.1] at (-70:1.5) {$d$};
    \draw node[left]       at (1.7,1)   {$f$};

    \draw (1.5,-2) --node {\midarrow} (1.5,2);

    \draw node[left] at (0,0) {$i$};
    \draw node[above] at (60:3) {$\ell$};
    \draw node[right] at (3,0) {$k$};
    \draw node[below] at (-60:3) {$j$};

    \draw node at (0.8,0){$\theta'$};
    \draw node at (2.2,0){$\sigma'$};
\end{tikzpicture}
\caption{Super Ptolemy transformation}
\label{fig:super_ptolemy}
\end{figure}
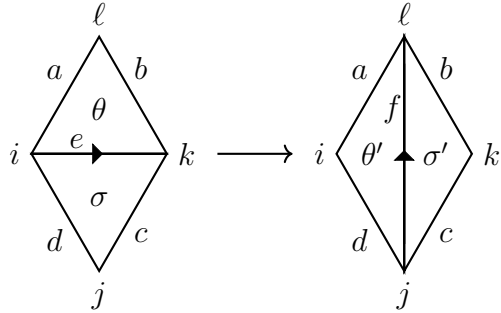
Note that in \Cref{eqn:super_ptolemy_lambda}, the order of multiplying the two odd variables $\sigma$ and $\theta$ is determined by the orientation of the 
edge being flipped (see the arrow in \Cref{fig:super_ptolemy}).

In Figure \ref{fig:super_ptolemy}, the orientations of the four boundary edges 
are omitted, but the super Ptolemy transformation does change the orientation of the edge labeled $b$ (the edges $a,c,d$ keep their same orientation).

\begin {definition} \label{def:h_length}
    For a triangle with vertices $i,j,k$, define the \emph{$h$-length at vertex $i$} to be
    \[ h^i_{jk} = \frac{\lambda_{jk}}{\lambda_{ij}\lambda_{ik}} \]
    Note that this is the same as the definition of ``\emph{$h$-length}'' in the classical (i.e. non-super) case. See e.g. \cite{penner_12}.
\end {definition}

\begin {definition} \label{def:norm_mu}
    For a triangle with vertices $i,j,k$ and $\mu$-invariant $\theta = \boxed{ijk}$, we also define two sets of \emph{normalized $\mu$-invariants}:
    \[ 
        \Tu^i_{jk} := \sqrt{\lambda_{jk}\over \lambda_{ij}\lambda_{ik}}\; \theta 
                    = \sqrt{h^i_{jk}} \; \theta, \;\; 
        \Tu^j_{ik} := \sqrt{\lambda_{ik}\over \lambda_{ij}\lambda_{jk}}\; \theta
                    = \sqrt{h^j_{ik}} \; \theta, \;\; 
        \Tu^k_{ij} := \sqrt{\lambda_{ij}\over \lambda_{ik}\lambda_{jk}}\; \theta
                    = \sqrt{h^k_{ij}} \; \theta,
    \]
    \[ 
        \Td^i_{jk} := \sqrt{\lambda_{ij}\lambda_{ik}\over \lambda_{jk}}\; \theta = \sqrt{1\over h^i_{jk}} \; \theta, \;
        \Td^j_{ik} := \sqrt{\lambda_{ij}\lambda_{jk}\over \lambda_{ik}}\; \theta = \sqrt{1\over h^j_{ik}} \; \theta, \;
        \Td^k_{ij} := \sqrt{\frac{\lambda_{ik}\lambda_{jk}}{\lambda_{ij}}}\theta = \sqrt{\frac{1}{h^k_{ij}}} \; \theta,
    \]
    all of which are associated to a (triangle, vertex) pair, i.e. to an angle.  
\end {definition}

\begin {remark} \label{rmk:normalized_mu_relations}
    The $h$-lengths and normalized $\mu$-invariants within a triangle satisfy the following relations 
    \begin {multicols}{2}
    \begin {itemize}
        \item[$(i)$] $h^j_{ik} = \left( \frac{\lambda_{ik}}{\lambda_{jk}} \right)^2 h^i_{jk}$

        \medskip

        \item[$(ii)$] $\Tu^j_{ik} = \frac{\lambda_{ik}}{\lambda_{jk}} \, \Tu^i_{jk}$

        \columnbreak

        \item[$(iii)$] $\Td^j_{ik} = \frac{\lambda_{jk}}{\lambda_{ik}} \, \Td^i_{jk}$

        \medskip

        \item[$(iv)$] $\Td^j_{ik} = \lambda_{ij} \Tu^i_{jk}$
    \end {itemize}
    \end {multicols}
\end {remark}

\begin {remark} \label{rmk:alternate_ptolemy}
    In terms of the normalized $\mu$-invariants, the super Ptolemy relations
    (\Cref{eqn:super_ptolemy_lambda,eqn:super_ptolemy_mu_right,eqn:super_ptolemy_mu_left}) take a very simple form.     
    Using the labelling of vertices of the quadrilateral in \Cref{fig:super_ptolemy}, we can rewrite these equations as follows.
     \begin {align*}
        \lambda_{j\ell} &= \frac{\lambda_{ij}\lambda_{kl}+\lambda_{i\ell}\lambda_{jk}}{\lambda_{ik}} + \Td^j_{ik}\Td^\ell_{ik} \label{eqn:super_ptolemy_lambda_star} \tag{$1^\star$} \\
        \Tu^k_{j\ell}   &= \Tu^k_{ij} - \Tu^k_{i\ell} \label{eqn:super_ptolemy_mu_right_star}\tag{$2^\star$} \\
        \Tu^i_{j\ell}   &= \Tu^i_{jk} + \Tu^i_{k\ell} \label{eqn:super_ptolemy_mu_left_star}\tag{$3^\star$}
    \end {align*}
\end {remark}

\begin {prop} \label{prop:h-lengths_are_additive}
    Let $ijk$ and $ik\ell$ be two adjacent triangles, with the edge separating the triangles oriented $i \to k$ (as in \Cref{fig:super_ptolemy}). Then
    \begin {itemize}
        \item[$(a)$] $\displaystyle h^i_{j\ell} = h^i_{jk} + h^i_{k\ell} + \Tu^i_{jk} \Tu^i_{k\ell}$
        \item[$(b)$] $\displaystyle h^k_{j\ell} = h^k_{ji} + h^k_{i\ell} + \Tu^k_{ji} \Tu^k_{i\ell}$
    \end {itemize}
\end {prop}
\begin {proof}
    We will just prove part $(a)$. The calculation for $(b)$ is analogous. By definition, $h^i_{j\ell} = \frac{\lambda_{j\ell}}{\lambda_{ij}\lambda_{i\ell}}$.
    Using the super Ptolemy relation (\Cref{eqn:super_ptolemy_lambda}) and substituting for $\lambda_{j\ell}$, we get
    \[
        h^i_{j\ell} = \frac{\lambda_{ij}\lambda_{k\ell} + \lambda_{jk}\lambda_{i\ell}}{\lambda_{ij}\lambda_{i\ell}\lambda_{ik}} + \frac{\Td^j_{ik} \Td^\ell_{ik}}{\lambda_{ij}\lambda_{i\ell}} 
                    = \frac{\lambda_{k\ell}}{\lambda_{i\ell}\lambda_{ik}} + \frac{\lambda_{jk}}{\lambda_{ij}\lambda_{ik}} + \frac{\Td^j_{ik}}{\lambda_{ij}} \cdot \frac{\Td^\ell_{ik}}{\lambda_{i\ell}}. 
    \]
    By definition, the first two terms are $h^i_{k\ell}$ and $h^i_{jk}$. By \Cref{rmk:normalized_mu_relations}($iv$), the last term is equal to $\Tu^i_{jk} \Tu^i_{k\ell}$.
\end {proof}

\section{Super-Matrices and $\osp(1|2)$} \label{sec:osp}

An $m|n \times m|n$ (even) super-matrix $M$ over a super-algebra can be written as a block matrix of the form
\[
    M = \left( \begin{array}{c|c}
            A    & \Xi \\ \hline
            \Psi & B
        \end{array} \right),
\]
where $A,B$ are $m\times m$, $n\times n$ matrices with even entries, and $\Psi,\Xi$ are $m\times n,n\times m$ matrices with odd entries. 
We follow the convention that Greek letters denote odd variables.
The super-symmetric analogue of the determinant of a matrix, called \emph{Berezinian}, is defined as follows.
\[\ber(M):=\det(B)^{-1}\det(A+ \Xi B^{-1}\Psi)\]
when $B$ is invertible.
Let $A^{\tp}$ denote the transpose of a matrix, the super-transpose of a super-matrix is defined as:
\[
    M^{\st} := \left( \begin{array}{c|c}
                   A^{\tp} & \Psi^{\tp} \\ \hline
                   -\Xi^{\tp} & B^{\tp}
               \end{array} \right).
\]

Consider the set of $2|1\times 2|1$ super matrices over $\calA$. 
\[ M = \ospmatrix {a}{b}{\gamma} {c}{d}{\delta} {\alpha}{\beta}{e} \]

Its Berezinian is given by
$\ber(M)={1\over e}(ad-bc)+{\alpha\over e^2}(d\gamma - b\delta)+{\beta\over e^2}(c\gamma-a\delta)-{2\alpha\beta\gamma\delta\over e^3}.$
Let $J$ denote the following matrix
\[
    J = \left( \begin{array}{cc|c}
        0 & 1 & 0 \\
        -1 & 0 & 0 \\ \hline
        0 & 0 & 1
    \end{array} \right)
\]
The group $\osp(1|2)$ is defined as the set of $2|1\times 2|1$ super-matrices $g$ with $\ber(g)=1$, and satisfying $g^{\st} J g = J$.
These constraints can be written down explicitly in the following system of equations.
\begin{align}
    e      &= 1 + \alpha\beta   \label{eq:osp1} \\
    e^{-1} &= ad - bc           \label{eq:osp2} \\
    \alpha &= c\gamma - a\delta \label{eq:osp3} \\
    \beta  &= d\gamma - b\delta \label{eq:osp4} \\
    \gamma &= a\beta  - b\alpha \label{eq:osp5} \\
    \delta &= c\beta  - d\alpha \label{eq:osp6}
\end{align}
Notice that combining \cref{eq:osp1,eq:osp2} gives us that
\begin{equation}
    ad - bc = 1 - \alpha\beta.
\end{equation}
Cross multiplying \cref{eq:osp3,eq:osp5} or \cref{eq:osp4,eq:osp6} gives us that
\begin{equation} \label{eq:osp-cm}
    \alpha\beta = \gamma\delta.
\end{equation}

\begin {remark} \label{remark:inverse}
    Re-arranging the equation $g^{\mathrm{st}} J g = J$ gives $g^{-1} = J^{-1} g^{\mathrm{st}} J$. Thus if $\ber(g) = 1$,
    then $g \in \osp(1|2)$ if and only if the inverse is given by
    \[ 
        g^{-1} = \ospmatrix {a}{b}{\gamma} {c}{d}{\delta} {\alpha}{\beta}{e}^{-1}
               = \ospmatrix {d}{-b}{-\beta} {-c}{a}{\alpha} {\delta}{-\gamma}{e}
    \]
\end {remark}

Now we define special elements of $\osp(1|2)$ which will be the main ingredients in our matrix formulas in \Cref{sec:Mpaths}.

\begin {definition} \label{def:matrices}
    Let $x$ and $h$ be even variables (with $\sqrt{h}$ well-defined), and $\theta$ an odd variable. Then we define the following matrices:
    \[
            E(x) = \ospmatrix {0}{-x}{0} {1/x}{0}{0} {0}{0}{1}  \quad \quad 
            A(h|\theta) = \ospmatrix {1}{0}{0} {h}{1}{-\sqrt{h}\theta} {\sqrt{h}\theta}{0}{1} \quad \quad
            \rho = \ospmatrix {-1}{0}{0} {0}{-1}{0} {0}{0}{1}
    \]
    Their inverses are given by $\rho^{-1} = \rho$, $E(x)^{-1} = \rho E(x) = E(-x)$, and
    \[ A(h|\theta)^{-1} = \ospmatrix {1}{0}{0} {-h}{1}{\sqrt{h}\theta} {-\sqrt{h}\theta}{0}{1} \]
    If $i,j,k$ are three marked points (vertices of a polyon), then we will almost always use the following shorthand notations:
    \[
        E_{ij} := E(\lambda_{ij}) = \ospmatrix {0}{-\lambda_{ij}}{0} {\lambda_{ij}^{-1}}{0}{0} {0}{0}{1}  \quad \quad
        A^i_{jk} := A\left(h^{i}_{jk}\middle|\boxed{ijk}\right) = \ospmatrix {1}{0}{0} {h^i_{jk}}{1}{-\Tu^i_{jk}} {\Tu^i_{jk}}{0}{1}
    \]
\end{definition}

\begin {remark} \label{rmk:fermionic_reflection}
    The matrix $\rho$ was called ``\emph{fermionic reflection}'' in \cite{pz_19}. 
    Note that we have $\rho A(h|\theta)\rho=A(h|-\theta)$ (i.e. conjugation of $A$ by $\rho$ negates the fermionic variable $\theta$).
    This is easy to see, since left-multiplication by $\rho$ scales the first
    two rows by $-1$, and right-multiplication by $\rho$ scales the first two columns by $-1$. 
    \end {remark}

\begin{remark} \label{rem:osp}
    Observe that $\ber E(x) = \ber A(h|\theta) = \ber \rho = 1$, and that $A(h|\theta)^{-1}$, $E(x)^{-1}$, and $\rho^{-1}$ have the form of \Cref{remark:inverse},
    and so these matrices are in $\osp(1|2)$. 
\end{remark}

\section{A Flat $\osp(1|2)$-Connection} \label{sec:Mpaths}

Following \cite{fg_06, mw13}, from a triangulation $T$ of a marked surface with boundary, we will define a planar graph $\Gamma_T$
and associate certain matrices to the (oriented) edges of the graph, giving a flat $\osp(1|2)$-connection.

\begin {remark}
    Although our main results (\Cref{thm:12-entry} and \Cref{thm:generic}) are stated only for polygons,
    the constructions given below for $\Gamma_T$ and the connection make sense for any triangulated surface.
    For a surface with non-trivial topology, the monodromy of this connection should coincide (up to conjugation)
    with the representation $\pi_1(S) \to \osp(1|2)$ described in section 6 of \cite{pz_19}. The benefit of
    our approach is that we are able to get nontrivial information even in the case of a polygon (where the fundamental group is trivial).
\end {remark}

\begin {definition} \label{def:GammaT}
    Inside each triangle of $T$, there is a hexagonal face of $\Gamma_T$ with three sides parallel to the sides of the triangle.
    When two triangles share a side, the two vertices of $\Gamma_T$ on
    opposite sides of this edge are connected (see \Cref{fig:Gamma_T}).
\end {definition}

\begin {figure}[h]
\centering
\begin {tikzpicture}[scale=0.8]
    \draw[dashed] (-3,0) -- (0,-3) -- (3,0) -- (0,3) -- cycle;
    \draw[dashed] (0,-3) -- (0,3);

    \foreach \x in {-0.3, 0.3} {\foreach \y in {-1.5, 1.5} {
        \draw[fill=black] (\x, \y) circle (0.06);
    }}

    \foreach \x in {-0.7, 0.7} {\foreach \y in {-1.9, 1.9} {
        \draw[fill=black] (\x,\y) circle (0.06);
    }}

    \foreach \x in {-2.2, 2.2} {\foreach \y in {-0.4, 0.4} {
        \draw[fill=black] (\x,\y) circle (0.06);
    }}

    \draw (0.3,-1.5) -- (0.3,1.5) -- (-0.3,1.5) -- (-0.3,-1.5) -- cycle;
    \draw (-0.3,-1.5) -- (-0.7,-1.9) -- (-2.2,-0.4) -- (-2.2,0.4) -- (-0.7,1.9) -- (-0.3,1.5);
    \draw (0.3,-1.5) -- (0.7,-1.9) -- (2.2,-0.4) -- (2.2,0.4) -- (0.7,1.9) -- (0.3,1.5);
\end {tikzpicture}
\caption {The graph $\Gamma_T$, with $T$ in dashed lines.}
\label {fig:Gamma_T}
\end {figure}
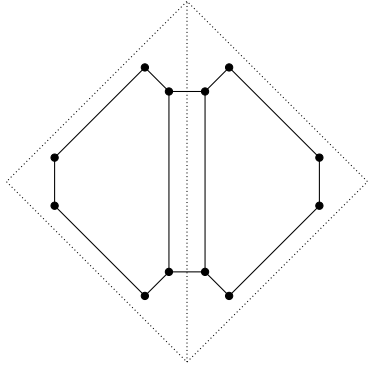

\begin {remark} \label{rmk:Gamma_T_edges_faces}
    The graph $\Gamma_T$ has 3 kinds of edges and 2 kinds of faces. The three types of edges of $\Gamma_T$ are:
    \begin {itemize}
        \item The edges parallel to arcs $\tau$ of the triangulation $T$. (If $\tau \in T$ is a boundary edge,
              then there is only one such edge of $\Gamma_T$, and if $\tau$ is an internal diagonal, then there are two such parallel edges in $\Gamma_T$.)
        \item The edges within a triangle that are \emph{not} prallel to arcs $\tau$ of $T$. (These naturally correspond to the angles
              of the triangles.)
        \item The edges which cross the arcs $\tau$ of $T$.
    \end {itemize}
  
    The two types of faces are as follows:
    \begin {itemize}
        \item Within each triangle of $T$, there is a hexagonal face of $\Gamma_T$.
        \item Surrounding each internal diagonal of $T$, there is a quadrilateral face of $\Gamma_T$.
    \end {itemize}
\end {remark}

\begin {definition}
    For a graph embedded on a surface, a \emph{graph connection} is an assignment of a matrix to each oriented edge,
    such that opposite orientations of the same edge are assigned inverse matrices. For a path in the graph,
    the \emph{holonomy} is the corresponding composition/product of matrices along the path. If the path is a loop,
    then the holonomy is also called \emph{monodromy}. A connection is called \emph{flat} if the monodromy around each contractible face
    is the identity matrix.
\end {definition}

We will now define a flat $\osp(1|2)$-connection on the graph $\Gamma_T$.  

\begin{definition} \label{def:holonomy_matrices}
    Given $(S,M)$ and $T$ with a given orientation, we define the following holonomy matrices for the edges described in \Cref{rmk:Gamma_T_edges_faces}.
    They are pictured in \Cref{fig:holonomy_matrices}.
    \begin {enumerate}
        \item Inside triangle $ijk$, the clockwise orientation of the edge at angle $i$ is assigned the matrix $A(h^i_{jk}|\theta)$.
        \item Inside triangle $ijk$, the clockwise orientation of the edge $ij$ is assigned the matrix $E(\lambda_{ij})$.
        \item For each internal diagonal $ij$, there are two edges of $\Gamma_T$ which cross $ij$. Supposing that the spin structure 
              has orientation $i \to j$, the edge closer to $i$ is assigned the identity matrix, and the edge closer to $j$ is assigned $\rho$ (the fermionic reflection).
    \end {enumerate}
\end{definition}

\begin{figure}[h]
\centering
\begin{tabular}{|c|cc|cc|}
    \hline &&&&\\[-1.2em]
    Type $(i)$ & \begin{tabular}{c} \begin{tikzpicture}
	\begin{pgfonlayer}{nodelayer}
		\node [style=bullet] (0) at (-3.25, 0) {};
		\node [style=bullet] (1) at (-1, 4) {};
		\node [style=bullet] (2) at (1.5, 0) {};
		\node [style=bullet small] (3) at (-2.25, 1.25) {};
		\node [style=bullet small] (4) at (-1.75, 0.25) {};
		\node [style=none] (5) at (-3.75, -0.5) {$i$};
		\node [style=none] (6) at (-1, 4.75) {$k$};
		\node [style=none] (7) at (2, -0.5) {$j$};
		\node [style=none] (8) at (-1, 1.5) {$\theta$};
	\end{pgfonlayer}
	\begin{pgfonlayer}{edgelayer}
		\draw (0) to (1);
		\draw (1) to (2);
		\draw (2) to (0);
		\draw [style=red, bend left] (3) to (4);
	\end{pgfonlayer}
\end{tikzpicture} \end{tabular} & ${A^i_{jk}}^{-1}$ & \begin{tabular}{c} \begin{tikzpicture}
	\begin{pgfonlayer}{nodelayer}
		\node [style=bullet] (0) at (-3.25, 0) {};
		\node [style=bullet] (1) at (-1, 4) {};
		\node [style=bullet] (2) at (1.5, 0) {};
		\node [style=bullet small] (3) at (-2.25, 1.25) {};
		\node [style=bullet small] (4) at (-1.75, 0.25) {};
		\node [style=none] (5) at (-3.75, -0.5) {$i$};
		\node [style=none] (6) at (-1, 4.75) {$k$};
		\node [style=none] (7) at (2, -0.5) {$j$};
		\node [style=none] (8) at (-1, 1.5) {$\theta$};
	\end{pgfonlayer}
	\begin{pgfonlayer}{edgelayer}
		\draw (0) to (1);
		\draw (1) to (2);
		\draw (2) to (0);
		\draw [style=red 1, bend left] (3) to (4);
	\end{pgfonlayer}
\end{tikzpicture} \end{tabular} & $A^{i}_{jk}$ \\
    & & & & \\[-1.2em] \hline
    & & & & \\[-1.2em]
    Type $(ii)$ & \begin{tabular}{c} \begin{tikzpicture}
	\begin{pgfonlayer}{nodelayer}
		\node [style=bullet] (0) at (-3.25, 0) {};
		\node [style=bullet] (1) at (-1, 3.75) {};
		\node [style=bullet] (2) at (1.25, 0) {};
		\node [style=bullet small] (3) at (-2.5, 0.25) {};
		\node [style=bullet small] (4) at (0.5, 0.25) {};
		\node [style=none] (5) at (-3.75, -0.5) {$i$};
		\node [style=none] (6) at (-1, 4.5) {$k$};
		\node [style=none] (7) at (1.75, -0.5) {$j$};
		\node [style=none] (8) at (-1, 1.5) {$\theta$};
	\end{pgfonlayer}
	\begin{pgfonlayer}{edgelayer}
		\draw (0) to (1);
		\draw (1) to (2);
		\draw (2) to (0);
		\draw [style=red] (3) to (4);
	\end{pgfonlayer}
\end{tikzpicture} \end{tabular} & $E_{ij}^{-1}$ & \begin{tabular}{c} \begin{tikzpicture}
	\begin{pgfonlayer}{nodelayer}
		\node [style=bullet] (0) at (-3.25, 0) {};
		\node [style=bullet] (1) at (-1, 3.75) {};
		\node [style=bullet] (2) at (1.25, 0) {};
		\node [style=bullet small] (3) at (-2.5, 0.25) {};
		\node [style=bullet small] (4) at (0.5, 0.25) {};
		\node [style=none] (5) at (-3.75, -0.5) {$i$};
		\node [style=none] (6) at (-1, 4.5) {$k$};
		\node [style=none] (7) at (1.75, -0.5) {$j$};
		\node [style=none] (8) at (-1, 1.5) {$\theta$};
	\end{pgfonlayer}
	\begin{pgfonlayer}{edgelayer}
		\draw (0) to (1);
		\draw (1) to (2);
		\draw (2) to (0);
		\draw [style=red 1] (3) to (4);
	\end{pgfonlayer}
\end{tikzpicture} \end{tabular} & $E_{ij}$ \\
    & & & & \\[-1.2em] \hline
    & & & & \\[-1.2em]
    Type $(iii)$ & \begin{tabular}{c} \begin{tikzpicture}
	\begin{pgfonlayer}{nodelayer}
		\node [style=bullet] (0) at (-3.25, 0) {};
		\node [style=bullet] (2) at (1.25, 0) {};
		\node [style=bullet small] (3) at (0.25, 0.5) {};
		\node [style=bullet small] (4) at (0.25, -0.5) {};
		\node [style=none] (5) at (-3.75, -0.5) {$i$};
		\node [style=none] (7) at (1.75, -0.5) {$j$};
		\node [style=none] (8) at (-1, 0) {};
	\end{pgfonlayer}
	\begin{pgfonlayer}{edgelayer}
		\draw [style=black arrow] (0) to (8.center);
		\draw (2) to (8.center);
		\draw [style=red, bend left=45] (3) to (4);
	\end{pgfonlayer}
\end{tikzpicture} \end{tabular} & $\rho$ & \begin{tabular}{c} \begin{tikzpicture}
	\begin{pgfonlayer}{nodelayer}
		\node [style=bullet] (0) at (-3.25, 0) {};
		\node [style=bullet] (2) at (1.25, 0) {};
		\node [style=bullet small] (3) at (-2.5, 0.5) {};
		\node [style=bullet small] (4) at (-2.5, -0.5) {};
		\node [style=none] (5) at (-3.75, -0.5) {$i$};
		\node [style=none] (7) at (1.75, -0.5) {$j$};
		\node [style=none] (8) at (-1, 0) {};
	\end{pgfonlayer}
	\begin{pgfonlayer}{edgelayer}
		\draw [style=red, bend left=45] (3) to (4);
		\draw [style=black arrow] (0) to (8.center);
		\draw (8.center) to (2);
	\end{pgfonlayer}
\end{tikzpicture} \end{tabular} & $\id$ \\[1em] \hline
\end{tabular}
\caption {The three types of holonomy matrices.}
\label{fig:holonomy_matrices}
\end{figure}

\begin {prop}
    The holonomy matrices from \Cref{def:holonomy_matrices} define a flat $\osp(1|2)$-connection on $\Gamma_T$.
\end {prop}

\begin {proof}
    As was mentioned in \Cref{rmk:Gamma_T_edges_faces}, there are only two types of faces in $\Gamma_T$. So we only need to check that these two
    types of monodromy give the identity matrix. Also note that changing the starting point of a cycle changes the monodromy only by conjugation.
    So if we verify that a particular monodromy around a face is the identity, then the same follows for any starting point.

    First let us consider a rectangular face corresponding to a diagonal $ij$ of the triangulation. The (counter-clockwise) monodromy around this face is
    \[ \mathrm{id} \cdot E_{ij} \cdot \rho \cdot E_{ij} \]
    But since $\rho \cdot E_{ij} = E_{ij}^{-1}$, this gives the result.

    Second, we must consider a hexagonal face inside a triangle $ijk$. The (clockwise) monodromy around this face, starting near vertex $i$, is given by
    \[ A^i_{jk} \, E_{ji} \, A^j_{ik} \, E_{kj} \, A^k_{ij} \, E_{ik} \]
    It is straightforward to check that this product is the identity matrix.
\end {proof}

\begin {remark}
    Since the connection is flat, the holonomy between two vertices of $\Gamma_T$ does not depend on the choice of path, since the graph is planar
    and any two paths are homotopic (thought of as paths on the ambient surface).
\end {remark}

\begin {remark}
    Note that the additional data of the spin structure on the triangulation $T$ allows an additional
    elementary step corresponding to $\rho$ that was not present in \cite{fg_06} nor \cite{mw13}.
    However, its inclusion ensures that all monodromies yield the identity matrix (rather than the identity
    matrix up to sign).
\end {remark}

\begin {definition} \label{def:near}
    If vertex $i$ of a polygon is incident to $m$ triangles in $T$, then there are $2m$ vertices of $\Gamma_T$ corresponding to the angles
    of these triangles at $m$. We will say that any of these $2m$ vertices of $\Gamma_T$ are ``\emph{near}'' the vertex $m$.
\end {definition}

\begin {theorem} \label{thm:12-entry}
    Suppose we have a triangulation $T$ of a polygon endowed with an orientation.  Let $i$ and $j$ be two vertices of the polygon, 
    and $i'$ and $j'$ any vertices of $\Gamma_T$ that are near $i$ and $j$, and let $H$
    be the holonomy from $i'$ to $j'$. Then the $(1,2)$-entry of $H$ is equal to $\pm \lambda_{ij}$.
\end {theorem}

We will prove this theorem in the next section. The first step in partially proving this theorem is the following.

\begin {lemma} \label{lem:ij_independence}
    The result of \Cref{thm:12-entry} does not depend on the particular choices of $i'$ and $j'$.
\end {lemma}
\begin {proof}
    Choosing different $i'$ or $j'$ near the same $i$ and $j$ corresponds to multiplying $H$ (on the right for $i$ and the left for $j$) by 
    a product of matrices of the following types: $A^i_{jk}$, ${A^i_{jk}}^{-1}$, $\rho$, or $\id$.
    Note that we do not need a separate case for $\rho^{-1}$ since $\rho^{-1} = \rho$.
    See \Cref{fig:change_near_vertex} for an illustration of the different cases. In the figure, adding the red edge to the beginning of the blue path
    corresponds to prepending (i.e. right-multiplying) the holonomy by the indicated matrix.

    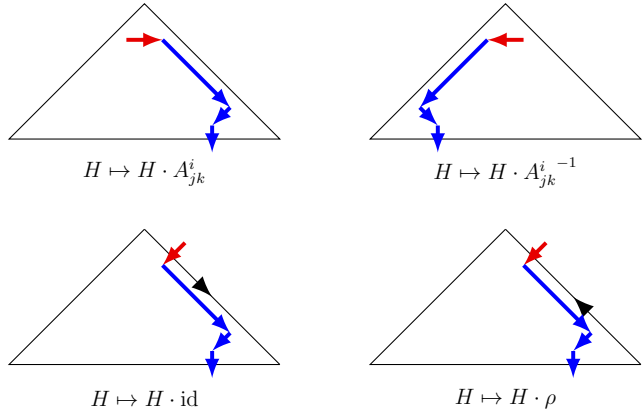
\begin {figure}[h]
    \centering
    \begin {tikzpicture}[scale=0.6, every node/.style={scale = 0.7}]

        \draw (-3,0) -- (3,0) -- (0,3) -- cycle;

        \draw[-latex, red!90!black, line width = 1.5] (-0.4,2.2) -- (0.4,2.2);
        \draw[-latex, blue, line width = 1.5] (0.4,2.2) -- (1.9,0.7);
        \draw[-latex, blue, line width = 1.5] (1.9,0.7) -- (1.5,0.3);
        \draw[-latex, blue, line width = 1.5] (1.5,0.3) -- (1.5,-0.3);

        \draw (0,-0.75) node {$H \mapsto H \cdot A^i_{jk}$};

        \begin {scope}[shift={(8,0)}]
            \draw (-3,0) -- (3,0) -- (0,3) -- cycle;

            \draw[-latex, red!90!black, line width = 1.5] (0.4,2.2) -- (-0.4,2.2);
            \draw[-latex, blue, line width = 1.5] (-0.4,2.2) -- (-1.9,0.7);
            \draw[-latex, blue, line width = 1.5] (-1.9,0.7) -- (-1.5,0.3);
            \draw[-latex, blue, line width = 1.5] (-1.5,0.3) -- (-1.5,-0.3);

            \draw (0,-0.75) node {$H \mapsto H \cdot {A^i_{jk}}^{-1}$};
        \end {scope}

        \begin {scope}[shift={(0,-5)}]
            \draw (0,3) -- (-3,0) -- (3,0);
            \draw[-->-] (0,3) -- (3,0);

            \draw[-latex, red!90!black, line width = 1.5] (0.9,2.7) -- (0.4,2.2);
            \draw[-latex, blue, line width = 1.5] (0.4,2.2) -- (1.9,0.7);
            \draw[-latex, blue, line width = 1.5] (1.9,0.7) -- (1.5,0.3);
            \draw[-latex, blue, line width = 1.5] (1.5,0.3) -- (1.5,-0.3);

            \draw (0,-0.75) node {$H \mapsto H \cdot \id$};
        \end {scope}

        \begin {scope}[shift={(8,-5)}]
            \draw (0,3) -- (-3,0) -- (3,0);
            \draw[-->-] (3,0) -- (0,3);

            \draw[-latex, red!90!black, line width = 1.5] (0.9,2.7) -- (0.4,2.2);
            \draw[-latex, blue, line width = 1.5] (0.4,2.2) -- (1.9,0.7);
            \draw[-latex, blue, line width = 1.5] (1.9,0.7) -- (1.5,0.3);
            \draw[-latex, blue, line width = 1.5] (1.5,0.3) -- (1.5,-0.3);

            \draw (0,-0.75) node {$H \mapsto H \cdot \rho$};
        \end {scope}
    \end {tikzpicture}
    \caption {The different ways to move to another ``near'' vertex.}
    \label {fig:change_near_vertex}
    \end {figure}

    \medskip

    In the third case, multiplying by $\rho$ (on either the left or right) will negate the $(1,2)$-entry.
    For the first and second case, it is easy to see (simply by matrix multiplication) that multiplying on the left or right by $A^i_{jk}$ or ${A^i_{jk}}^{-1}$
    will not change the $(1,2)$-entry.
\end {proof}

\section{Proof of \Cref{thm:12-entry}} \label{sec:Hproof}

This section is devoted to a proof of our main theorem, \Cref{thm:12-entry}. We first state a more detailed version of the main theorem.
Let $T$ be a generic triangulation with default orientation, and with fan centers labelled as $c_i$ for $1\leq i\leq N$. 
Let $(a,b)$ be the longest diagonal in $T$ and denote $a=c_0$ and $b=c_{N+1}$.

\begin{definition} \label{def:H_ab}
    Let $H_{a,b}$ denote the holonomy following a path from a vertex near $a=c_0$ (on the side closer to $c_1$) to a vertex near $b=c_{N+1}$ 
    (on the side closer to $c_N$). We say that the holonomy is of type $\epsilon_a \epsilon_b$ where 
    \begin{align*}
        \epsilon_a &= \begin{cases}
                          0 & \text{ if } (c_0,c_1,c_2) \text{ are oriented clockwise,} \\
                          1 & \text{ otherwise.}
                      \end{cases} \\
        \epsilon_b &= \begin{cases}
                          0 & \text{ if } (c_{N-1},c_N,c_{N+1}) \text{ are oriented clockwise,} \\
                          1 & \text{ otherwise.}
        \end{cases}
    \end{align*}
\end{definition}

\begin{remark}
    Note that given $\epsilon_a$, $\epsilon_b$ is determined by the number of fans $N$ via the relation $\epsilon_a+\epsilon_b= N+1\mod 2$.
\end{remark}

\begin{theorem} \label{thm:generic}
    Let $T$ be a generic triangulation endowed with an arbitrary orientation (based on its spin structure), 
    and with fan centers labelled as $c_i$ for $1\leq i\leq N$ and $a=c_0,b=c_{N+1}$. 
    The holonomy matrix $H_{a,b}$ of type $\epsilon_a \epsilon_b$ is given by 
    \[ 
        H_{a,b} = \ospmatrix
        {-{\lambda_{c_1,c_{N+1}}\over \lambda_{c_0,c_1}}}
        { (-1)^{\epsilon_a} \lambda_{c_0,c_{N+1}} }
        {\Td^{c_{N+1}}_{c_0,c_1} }
        {(-1)^{\epsilon_b}{\l{c_1}{c_N}\over \lambda_{c_0,c_1}\l {c_N}{c_{N+1}}}}
        {(-1)^{\epsilon_a+\epsilon_b -1}{\l{c_0}{c_N}\over\l{c_N}{c_{N+1}}}}
        {(-1)^{\epsilon_b-1}{1\over \l{c_N}{c_{N+1}}}\Td^{c_N}_{c_0,c_1}}
        {{1\over\lambda_{c_0,c_1}}\td {c_1} {c_N}{c_{N+1}}}
        {(-1)^{\epsilon_a-1}\td {c_0}{c_N}{c_{N+1}}} 
        {1+\star}
    \]

    Here the formula for the $(3,3)$-entry (i.e. $1+\star$) can be given two equivalent ways, which (due to \Cref{rem:osp}) 
    follows from applications of both \Cref{eq:osp1} and \Cref{eq:osp-cm}:
    \[ 
        1 + \star = 1 + (-1)^{\epsilon_a-1}{1\over\lambda_{c_0,c_1}}\td {c_1}{c_N}{c_{N+1}}\td {c_0} {c_N}{c_{N+1}}
                  = 1 + (-1)^{\epsilon_b-1}{1\over \l{c_N}{c_{N+1}}}\Td^{c_{N+1}}_{c_0,c_1} \Td^{c_{N}}_{c_0,c_1}.
    \]
\end{theorem}

We begin the proof of Theorem \ref{thm:generic} by considering the special case of a fan triangulation with default orientation.  
Without loss of generality, we will assume the fan has vertices labeled by $1, 2, \dots, n$ in cyclic order.
In particular, there is one non-trivial fan center, but including the endpoints of the longest arc, we have $a=c_0=2$, $c_1=1$, and $b=c_2=c_{N+1}=n$.  
We recover that the holonomy $H_{2n}$ can only be type $00$ or $11$, see \Cref{fig:single-fan}.

\subsection{Fan Triangulation}

The next two results (\Cref{lem:product_of_A_matrices} and \Cref{cor:product_of_A_matrices}) compute
the holonomy of a path which stays near a fan center, and traverses over all the angles in a fan segement.

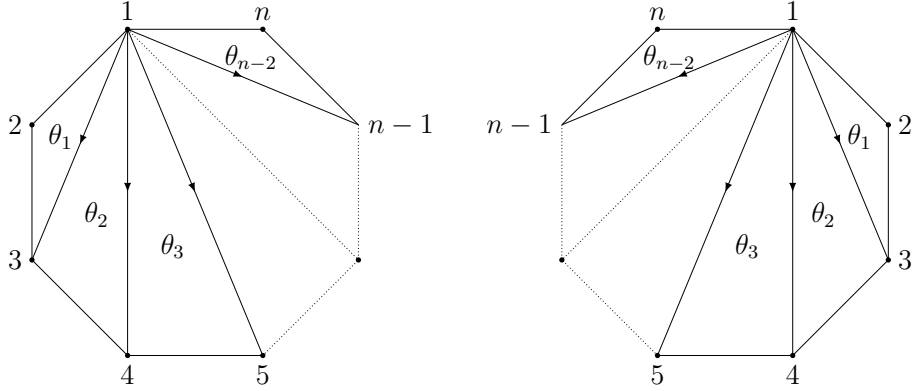
\begin{figure}[h]
    \scalebox{0.9}{\begin{tikzpicture}[decoration={
    markings,
    mark=at position 0.5 with {\arrow{latex}}}] 

\tikzstyle{every path}=[draw] 
		\path
    node[
      regular polygon,
      regular polygon sides=8,
      draw=none,
      inner sep=1.7cm,
    ] (T) {}
    %
    (T.corner 1) node[above] {$n$}
    (T.corner 2) node[above] {$1$}
    (T.corner 3) node[left]  {$2$}
    (T.corner 4) node[left]  {$3$}
    (T.corner 5) node[below] {$4$}
    (T.corner 6) node[below] {$5$}
    (T.corner 7) node[right] {}
    (T.corner 8) node[right] {$n-1$}
;

\draw [-] (T.corner 6) to (T.corner 5) to (T.corner 4) to (T.corner 3) to (T.corner 2) to (T.corner 1) to (T.corner 8);
\draw [dashed] (T.corner 8) to (T.corner 7) to (T.corner 6);
\foreach \i in {4,5,6,8}{
    \draw [postaction={decorate}]  (T.corner 2) to (T.corner \i);
}
\draw [dashed] (T.corner 2) to (T.corner 7);

\foreach \x in {1,2,...,7}{
\draw (T.corner \x) node [fill,circle,scale=0.2] {};}

\coordinate (m1) at ($0.3*(T.corner 2)+0.3*(T.corner 4)+0.4*(T.corner 3)$);
\coordinate (m2) at ($0.33*(T.corner 5)+0.33*(T.corner 4)+0.33*(T.corner 2)$);
\coordinate (m3) at ($0.33*(T.corner 5)+0.33*(T.corner 2)+0.33*(T.corner 6)$);
\coordinate (m4) at ($0.33*(T.corner 7)+0.33*(T.corner 2)+0.33*(T.corner 6)$);
\coordinate (m5) at ($0.33*(T.corner 7)+0.33*(T.corner 2)+0.33*(T.corner 8)$);
\coordinate (m6) at ($0.4*(T.corner 1)+0.3*(T.corner 2)+0.3*(T.corner 8)$);

\node at (m1) {$\theta_1$};
\node at (m2) {$\theta_2$};
\node at (m3) {$\theta_3$};
\node at (m6) {$\theta_{n-2}$};

\end{tikzpicture}
\;\;
\begin{tikzpicture}[decoration={
    markings,
    mark=at position 0.5 with {\arrow{latex}}}
    ] 

\tikzstyle{every path}=[draw] 
		\path
    node[
      regular polygon,
      regular polygon sides=8,
      draw=none,
      inner sep=1.7cm,
    ] (T) {}
    %
    (T.corner 1) node[above] {$1$}
    (T.corner 2) node[above] {$n$}
    (T.corner 3) node[left]  {$n-1$}
    (T.corner 4) node[left]  {}
    (T.corner 5) node[below] {$5$}
    (T.corner 6) node[below] {$4$}
    (T.corner 7) node[right] {$3$}
    (T.corner 8) node[right] {$2$}
;

\coordinate (p1) at (T.corner 2);
\coordinate (p2) at (T.corner 1);
\coordinate (p3) at (T.corner 8);
\coordinate (p4) at (T.corner 7);
\coordinate (p5) at (T.corner 6);
\coordinate (p6) at (T.corner 5);
\coordinate (p7) at (T.corner 4);
\coordinate (p8) at (T.corner 3);

\draw [-] (p6) to (p5) to (p4) to (p3) to (p2) to (p1) to (p8);
\draw [dashed] (p8) to (p7) to (p6);
\foreach \i in {4,5,6,8}{
    \draw [postaction={decorate}]  (p2) to (p\i);
}
\draw [dashed] (p2) to (p7);

\foreach \x in {1,2,...,7}{
\draw (p\x) node [fill,circle,scale=0.2] {};}

\coordinate (m1) at ($0.3*(p2)+0.3*(p4)+0.4*(p3)$);
\coordinate (m2) at ($0.33*(p5)+0.33*(p4)+0.33*(p2)$);
\coordinate (m3) at ($0.33*(p5)+0.33*(p2)+0.33*(p6)$);
\coordinate (m4) at ($0.33*(p7)+0.33*(p2)+0.33*(p6)$);
\coordinate (m5) at ($0.33*(p7)+0.33*(p2)+0.33*(p8)$);
\coordinate (m6) at ($0.4*(p1)+0.3*(p2)+0.3*(p8)$);

\node at (m1) {$\theta_1$};
\node at (m2) {$\theta_2$};
\node at (m3) {$\theta_3$};
\node at (m6) {$\theta_{n-2}$};

\end{tikzpicture}\quad\quad}
    \caption{Single fan triangulation. Left: type 00. Right: type 11.}
    \label{fig:single-fan}
\end{figure}

\begin {lemma} \label{lem:product_of_A_matrices}
    Suppose $i,j,k,\ell$ are vertices of a quadrilateral in counter-clockwise order, 
    and the (oriented) triangulation contains the edge $i \to k$ (as in \Cref{fig:super_ptolemy}).
    Then the product of $A$-matrices is 
    \begin {itemize}
        \item[$(a)$] $A^i_{k\ell} \, A^i_{jk} = A^i_{j\ell}$
        \item[$(b)$] $A^k_{ij} \, \rho \, A^k_{i\ell} = A^k_{j\ell} \, \rho$
    \end {itemize}
\end {lemma}
\begin {proof}
    $(a)$ Let $\theta = \boxed{ik\ell}$ and $\sigma = \boxed{ijk}$. 
    The matrix product gives
    \[ 
        \ospmatrix 
        {1}{0}{0} 
        {h^i_{k\ell}+h^i_{jk}-\Tu^i_{k\ell} \Tu^i_{jk}}{1}{- \Tu^i_{jk} - \Tu^i_{k\ell} } 
        {\Tu^i_{k\ell} + \Tu^i_{jk} }{0}{1} 
        =
        \ospmatrix 
        {1}{0}{0} 
        {h^i_{jk}+h^i_{k\ell} + \Tu^i_{jk} \Tu^i_{k\ell}}{1}{-\left( \Tu^i_{jk} + \Tu^i_{k\ell} \right)} 
        {\Tu^i_{jk} + \Tu^i_{k\ell}}{0}{1}.
    \]
    By \Cref{prop:h-lengths_are_additive}(a), the $(2,1)$-entry is equal to $h^i_{j\ell}$. By 
    \Cref{eqn:super_ptolemy_mu_left_star},
    the $(2,3)$- and $(3,1)$-entries are $-\Tu^i_{j\ell}$ and $\Tu^i_{j\ell}$, respectively.

    $(b)$ Recall from \Cref{rmk:fermionic_reflection} that $\rho A(h|\theta) \rho = A(h|-\theta)$. 
    By right-multiplying the equation in part $(b)$ by $\rho$,
    the claim is equivalent to 
    \[ A^k_{ij} \, A(h^k_{i\ell} | -\theta) = A^k_{j\ell} .\]
    This matrix product is equal to
    \[ 
        \ospmatrix 
        {1}{0}{0} 
        {h^k_{ij}+h^k_{i\ell} + \Tu^k_{ij} \Tu^k_{i\ell}}{1}{\Tu^k_{i\ell} - \Tu^k_{ij}} 
        {\Tu^k_{ij} - \Tu^k_{i\ell}}{0}{1}.
    \]
    By \Cref{prop:h-lengths_are_additive}(b) and \Cref{eqn:super_ptolemy_mu_right_star},
    the $(2,1)$-entry is $h^k_{j\ell}$ and the $(2,3)$- and $(3,1)$\nobreakdash-entries
    are $- \Tu^k_{j\ell}$ and $\Tu^k_{j\ell}$, respectively.
\end {proof}

\begin {corollary} \label{cor:product_of_A_matrices}
    Consider a single fan triangulation with default orientation, as depicted in \Cref{fig:single-fan}. 
    The ordered product of all $A$-matrices is
    \begin{equation} \label{eq:fan-angle-formula} 
        A^1_{n-1,n} \cdots A^1_{34} \, A^1_{23} = A^1_{2n}
    \end{equation}
    if the holonomy is type 00, and
    \begin{equation} \label{eq:fan-angle-formula} 
        {A^1_{n-1,n}}^{-1} \cdots {A^1_{34}}^{-1} \, {A^1_{23}}^{-1} = {A^1_{2n}}^{-1}
    \end{equation}
    if the holonomy is type 11.
\end {corollary}
\begin {proof}
    This follows from \Cref{lem:product_of_A_matrices} by induction. The base case of two triangles is simply \Cref{lem:product_of_A_matrices}.
    In general, if we first multiply the two right-most factors, they combine to give $A^1_{24}$. After performing
    the associated flip on the arc $(1,3)$, we now have a smaller polygon (on the vertices $1,2,4,5,\dots, n-1,n$ in counter-clockwise order), 
    again with a fan triangulation and the default orientation. 
    For $k\geq 4$, after $(k-3)$ such steps, we have multiplied together the $(k-2)$ right-most factors into $A^1_{2k}$ 
    and have flipped the arcs $(1,3)$, $(1,4)$, $\dots$, $(1,k-1)$ in order, resulting again in a smaller polygon, this time on the vertices $1,2,k,k+1,\dots, n-1,n$, with a 
    fan triangulation and the default orientation.  
    So the result follows by induction.
\end{proof}

The special case of \Cref{thm:generic} (and hence of \Cref{thm:12-entry}) for a fan triangulation is the following.  

\begin {theorem} \label{thm:single-fan}
    Consider a fan triangulation with default orientation (as in \Cref{fig:single-fan}). 
    The holonomy $H_{2n}$, of type $\epsilon \epsilon$, is given by
    \[ H_{2n} = \ospmatrix {-\frac{\lambda_{1n}}{\lambda_{12}}}{(-1)^{\epsilon}\lambda_{2n}}{\Td^n_{12}} {0}{-\frac{\lambda_{12}}{\lambda_{1n}}}{0} {0}{-\Td^2_{1n}}{1}\]
    In particular, the $(1,2)$-entry is equal to $\pm \lambda_{2n}$.
\end {theorem}
\begin{proof}
    By \Cref{cor:product_of_A_matrices}, the holonomy is simply the product of three matrices:
    \[ E_{1n} A^1_{2n} E_{12} \quad \text{ or } \quad E_{1n}^{-1} {A^1_{2n}}^{-1} E_{12}^{-1} \]
    After multiplying the matrices, use \Cref{rmk:normalized_mu_relations} to simplify (in particular 
    $\lambda_{1n} \Tu^1_{2n} = \Td^n_{12}$, and $-\lambda_{12} \Tu^1_{2n} = - \Td^2_{1n}$). 
\end{proof}

\begin {remark}
    By \Cref{lem:ij_independence}, if we choose different starting and ending vertices near $2$ and $n$, the resulting holonomy matrix
    will still have $(1,2)$-entry equal to $\pm \lambda_{2n}$.
\end {remark}

\subsection{Canonical paths} 

Given a generic triangulation $T$, we identify its fan centers as in Section \ref{sec:background}, letting $a=c_0$ and $b=c_{N+1}$ so that $(a,b)$ is the longest arc of $T$.  
We define two canonical paths along the corresponding graph $\Gamma_T$
in order to compute the holonomy $H_{a,b}$ as follows.

The first one, called the \emph{early-crossing canonical path} (or \emph{early path} for short), is defined as follows:

\begin {enumerate}
    \item For each $0\leq k\leq N-1$, follow the $E$-edge parallel to $(c_k,c_{k+1})$ and then continue along a series of $A$-edges 
          until reaching a point near $(c_{k+1},c_{k+2})$. Immediately cross the diagonal $(c_{k+1},c_{k+2})$.

    \item Continue step (1) $N-1$ times until reaching the last fan segment. 
          After crossing the diagonal $(c_{N-1},c_{N})$ from a point near $c_{N-1}$, we follow by an $E$-edge parallel to $(c_{N-1},c_{N})$ 
          and continue along a series of $A$-edges until reaching a point near $c_N$ as well as the arc $(c_N,c_{N+1})$. We then end with the $E$-edge parallel to $(c_N,c_{N+1})$.
\end {enumerate}

We also define the \emph{late-crossing canonical path}  (or \emph{late path} for short). 

\begin {enumerate}
    \item Follow the $E$-edge parallel to $(c_0,c_1)$ and then, as long as $N\geq 1$, continue along a series of $A$-edges until reaching 
          a point near $c_1$ as well as the arc $(c_1,c_2)$.  Immediately follow the $E$-edge parallel to $(c_1,c_2)$.

    \item For $1 \leq k \leq N-1$, cross the arc $(c_k,c_{k+1})$, followed by $A$-edges until reaching a point near $c_{k+1}$ 
          as well as the arc $(c_{k+1},c_{k+2})$.  Immediately follow the $E$-edge parallel to $(c_{k+1},c_{k+2})$.

    \item After traversing along $N-1$ such subpaths, we have arrived at a point near $b=c_{N+1}$.
\end {enumerate}

Note that flipping the triangulation upside down turns an early path into a late path, and vice-versa.

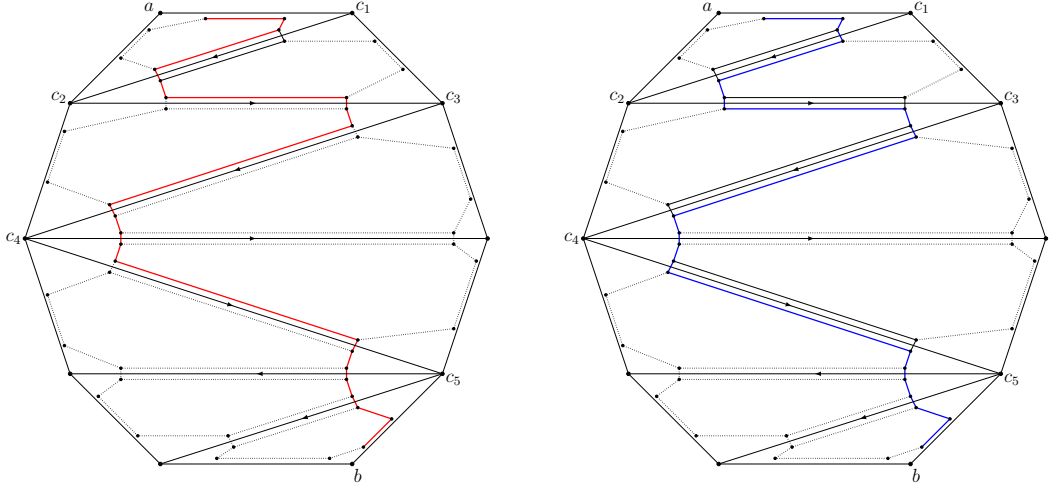
\begin{figure}
\centering
\scalebox{0.6}{\begin{tikzpicture}[decoration={
    markings,
    mark=at position 0.5 with {\arrow{latex}}}]
	\begin{pgfonlayer}{nodelayer}
		\node [style=bullet mid] (0) at (-8.25, 6) {};
		\node [style=bullet mid] (1) at (-10.25, 0) {};
		\node [style=bullet mid] (2) at (-8.25, -6) {};
		\node [style=bullet mid] (3) at (-4.25, -10) {};
		\node [style=bullet mid] (4) at (4.25, -10) {};
		\node [style=bullet mid] (5) at (8.25, -6) {};
		\node [style=bullet mid] (6) at (10.25, 0) {};
		\node [style=bullet mid] (7) at (8.25, 6) {};
		\node [style=bullet mid] (8) at (4.25, 10) {};
		\node [style=bullet mid] (9) at (-4.25, 10) {};
		\node [style=bullet small] (10) at (-4.75, 9.25) {};
		\node [style=bullet small] (11) at (-6, 8) {};
		\node [style=bullet small] (12) at (-4.5, 7.5) {};
		\node [style=bullet small] (13) at (-2.25, 9.75) {};
		\node [style=bullet small] (14) at (1.25, 9.75) {};
		\node [style=bullet small] (15) at (1, 9.25) {};
		\node [style=bullet small] (16) at (1.25, 8.75) {};
		\node [style=bullet small] (17) at (5.25, 8.75) {};
		\node [style=bullet small] (18) at (6.5, 7.5) {};
		\node [style=bullet small] (19) at (4, 6.25) {};
		\node [style=bullet small] (20) at (-4, 6.25) {};
		\node [style=bullet small] (21) at (-4.25, 7) {};
		\node [style=bullet small] (22) at (4, 5.75) {};
		\node [style=bullet small] (23) at (-4, 5.75) {};
		\node [style=bullet small] (24) at (-8.5, 4.75) {};
		\node [style=bullet small] (25) at (-9.25, 2.5) {};
		\node [style=bullet small] (26) at (-6.5, 1.5) {};
		\node [style=bullet small] (27) at (4.25, 5) {};
		\node [style=bullet small] (28) at (-6.25, 1) {};
		\node [style=bullet small] (29) at (-6, 0.25) {};
		\node [style=bullet small] (30) at (8.75, 0.25) {};
		\node [style=bullet small] (31) at (9.5, 1.75) {};
		\node [style=bullet small] (32) at (8.75, 4) {};
		\node [style=bullet small] (33) at (4.5, 4.5) {};
		\node [style=bullet small] (34) at (-6, -0.25) {};
		\node [style=bullet small] (35) at (-6.25, -1) {};
		\node [style=bullet small] (36) at (4.5, -4.5) {};
		\node [style=bullet small] (37) at (8.75, -4) {};
		\node [style=bullet small] (38) at (9.75, -1) {};
		\node [style=bullet small] (39) at (8.75, -0.25) {};
		\node [style=bullet small] (40) at (-6.5, -1.5) {};
		\node [style=bullet small] (41) at (-9.25, -2.5) {};
		\node [style=bullet small] (42) at (-8.5, -4.75) {};
		\node [style=bullet small] (43) at (-6, -5.75) {};
		\node [style=bullet small] (44) at (4, -5.75) {};
		\node [style=bullet small] (45) at (4.25, -5) {};
		\node [style=bullet small] (46) at (4, -6.25) {};
		\node [style=bullet small] (47) at (4.25, -7) {};
		\node [style=bullet small] (48) at (-1.25, -8.75) {};
		\node [style=bullet small] (49) at (-5.25, -8.75) {};
		\node [style=bullet small] (50) at (-7, -7) {};
		\node [style=bullet small] (51) at (-6, -6.25) {};
		\node [style=bullet small] (52) at (-1, -9.25) {};
		\node [style=bullet small] (53) at (-1.75, -9.75) {};
		\node [style=bullet small] (54) at (3.25, -9.75) {};
		\node [style=bullet small] (55) at (4.75, -9.25) {};
		\node [style=bullet small] (56) at (6, -8) {};
		\node [style=bullet small] (57) at (4.5, -7.5) {};
		\node [style=none] (58) at (-4.75, 10.25) {$a$};
		\node [style=none] (59) at (-8.75, 6.25) {$c_2$};
		\node [style=none] (60) at (8.75, 6.25) {$c_3$};
		\node [style=none] (61) at (-10.75, 0) {$c_4$};
		\node [style=none] (62) at (8.75, -6.25) {$c_5$};
		\node [style=none] (63) at (4.5, -10.5) {$b$};
		\node [style=none] (64) at (4.75, 10.25) {$c_1$};
	\end{pgfonlayer}
	\begin{pgfonlayer}{edgelayer}
		\draw (9) to (8);
		\draw (8) to (7);
		\draw (7) to (6);
		\draw (6) to (5);
		\draw (5) to (4);
		\draw (4) to (3);
		\draw (3) to (2);
		\draw (2) to (1);
		\draw (1) to (0);
		\draw (0) to (9);
		\draw [style=dashed] (10) to (13);
		\draw [style=Red-Thick] (13) to (14);
		\draw [style=Red-Thick] (14) to (15);
		\draw [style=Red-Thick] (15) to (12);
		\draw [style=dashed] (12) to (11);
		\draw [style=dashed] (11) to (10);
		\draw [style=Red-Thick] (12) to (21);
		\draw [style=Red-Thick] (21) to (20);
		\draw (21) to (16);
		\draw [style=dashed] (16) to (17);
		\draw [style=dashed] (17) to (18);
		\draw [style=dashed] (18) to (19);
		\draw [style=Red-Thick] (19) to (20);
		\draw (15) to (16);
		\draw [style=dashed] (20) to (23);
		\draw [style=dashed] (23) to (24);
		\draw [style=dashed] (24) to (25);
		\draw [style=dashed] (25) to (26);
		\draw [style=Red-Thick] (26) to (27);
		\draw [style=Red-Thick] (27) to (22);
		\draw [style=dashed] (22) to (23);
		\draw [style=Red-Thick] (22) to (19);
		\draw [style=Red-Thick] (26) to (28);
		\draw [style=dashed] (28) to (33);
		\draw [style=dashed] (33) to (27);
		\draw [style=dashed] (33) to (32);
		\draw [style=dashed] (32) to (31);
		\draw [style=dashed] (31) to (30);
		\draw [style=dashed] (30) to (29);
		\draw [style=Red-Thick] (29) to (28);
		\draw [style=Red-Thick] (29) to (34);
		\draw [style=dashed] (34) to (39);
		\draw [style=dashed] (39) to (30);
		\draw [style=Red-Thick] (34) to (35);
		\draw [style=Red-Thick] (35) to (36);
		\draw [style=dashed] (36) to (37);
		\draw [style=dashed] (37) to (38);
		\draw [style=dashed] (38) to (39);
		\draw [style=dashed] (35) to (40);
		\draw [style=dashed] (40) to (45);
		\draw [style=Red-Thick] (45) to (36);
		\draw [style=Red-Thick] (45) to (44);
		\draw [style=dashed] (44) to (43);
		\draw [style=dashed] (43) to (42);
		\draw [style=dashed] (42) to (41);
		\draw [style=dashed] (41) to (40);
		\draw [style=dashed] (43) to (51);
		\draw [style=dashed] (51) to (46);
		\draw [style=Red-Thick] (46) to (44);
		\draw [style=Red-Thick] (46) to (47);
		\draw [style=dashed] (47) to (48);
		\draw [style=dashed] (48) to (49);
		\draw [style=dashed] (50) to (49);
		\draw [style=dashed] (50) to (51);
		\draw [style=dashed] (48) to (52);
		\draw [style=dashed] (52) to (57);
		\draw [style=Red-Thick] (57) to (47);
		\draw [style=Red-Thick] (57) to (56);
		\draw [style=Red-Thick] (56) to (55);
		\draw [style=dashed] (55) to (54);
		\draw [style=dashed] (54) to (53);
		\draw [style=dashed] (53) to (52);
		\draw [style=spin] (8) to (0);
		\draw [style=spin] (0) to (7);
		\draw [style=spin] (7) to (1);
		\draw [style=spin] (1) to (5);
		\draw [style=spin] (5) to (3);
		\draw [style=spin] (5) to (2);
		\draw [style=spin] (1) to (6);
	\end{pgfonlayer}
\end{tikzpicture}}\;\;\;\;\;\;
\scalebox{0.6}{\begin{tikzpicture}[decoration={
    markings,
    mark=at position 0.5 with {\arrow{latex}}}]
	\begin{pgfonlayer}{nodelayer}
		\node [style=bullet mid] (0) at (-8.25, 6) {};
		\node [style=bullet mid] (1) at (-10.25, 0) {};
		\node [style=bullet mid] (2) at (-8.25, -6) {};
		\node [style=bullet mid] (3) at (-4.25, -10) {};
		\node [style=bullet mid] (4) at (4.25, -10) {};
		\node [style=bullet mid] (5) at (8.25, -6) {};
		\node [style=bullet mid] (6) at (10.25, 0) {};
		\node [style=bullet mid] (7) at (8.25, 6) {};
		\node [style=bullet mid] (8) at (4.25, 10) {};
		\node [style=bullet mid] (9) at (-4.25, 10) {};
		\node [style=bullet small] (10) at (-4.75, 9.25) {};
		\node [style=bullet small] (11) at (-6, 8) {};
		\node [style=bullet small] (12) at (-4.5, 7.5) {};
		\node [style=bullet small] (13) at (-2.25, 9.75) {};
		\node [style=bullet small] (14) at (1.25, 9.75) {};
		\node [style=bullet small] (15) at (1, 9.25) {};
		\node [style=bullet small] (16) at (1.25, 8.75) {};
		\node [style=bullet small] (17) at (5.25, 8.75) {};
		\node [style=bullet small] (18) at (6.5, 7.5) {};
		\node [style=bullet small] (19) at (4, 6.25) {};
		\node [style=bullet small] (20) at (-4, 6.25) {};
		\node [style=bullet small] (21) at (-4.25, 7) {};
		\node [style=bullet small] (22) at (4, 5.75) {};
		\node [style=bullet small] (23) at (-4, 5.75) {};
		\node [style=bullet small] (24) at (-8.5, 4.75) {};
		\node [style=bullet small] (25) at (-9.25, 2.5) {};
		\node [style=bullet small] (26) at (-6.5, 1.5) {};
		\node [style=bullet small] (27) at (4.25, 5) {};
		\node [style=bullet small] (28) at (-6.25, 1) {};
		\node [style=bullet small] (29) at (-6, 0.25) {};
		\node [style=bullet small] (30) at (8.75, 0.25) {};
		\node [style=bullet small] (31) at (9.5, 1.75) {};
		\node [style=bullet small] (32) at (8.75, 4) {};
		\node [style=bullet small] (33) at (4.5, 4.5) {};
		\node [style=bullet small] (34) at (-6, -0.25) {};
		\node [style=bullet small] (35) at (-6.25, -1) {};
		\node [style=bullet small] (36) at (4.5, -4.5) {};
		\node [style=bullet small] (37) at (8.75, -4) {};
		\node [style=bullet small] (38) at (9.75, -1) {};
		\node [style=bullet small] (39) at (8.75, -0.25) {};
		\node [style=bullet small] (40) at (-6.5, -1.5) {};
		\node [style=bullet small] (41) at (-9.25, -2.5) {};
		\node [style=bullet small] (42) at (-8.5, -4.75) {};
		\node [style=bullet small] (43) at (-6, -5.75) {};
		\node [style=bullet small] (44) at (4, -5.75) {};
		\node [style=bullet small] (45) at (4.25, -5) {};
		\node [style=bullet small] (46) at (4, -6.25) {};
		\node [style=bullet small] (47) at (4.25, -7) {};
		\node [style=bullet small] (48) at (-1.25, -8.75) {};
		\node [style=bullet small] (49) at (-5.25, -8.75) {};
		\node [style=bullet small] (50) at (-7, -7) {};
		\node [style=bullet small] (51) at (-6, -6.25) {};
		\node [style=bullet small] (52) at (-1, -9.25) {};
		\node [style=bullet small] (53) at (-1.75, -9.75) {};
		\node [style=bullet small] (54) at (3.25, -9.75) {};
		\node [style=bullet small] (55) at (4.75, -9.25) {};
		\node [style=bullet small] (56) at (6, -8) {};
		\node [style=bullet small] (57) at (4.5, -7.5) {};
		\node [style=none] (58) at (-4.75, 10.25) {$a$};
		\node [style=none] (59) at (-8.75, 6.25) {$c_2$};
		\node [style=none] (60) at (8.75, 6.25) {$c_3$};
		\node [style=none] (61) at (-10.75, 0) {$c_4$};
		\node [style=none] (62) at (8.75, -6.25) {$c_5$};
		\node [style=none] (63) at (4.5, -10.5) {$b$};
		\node [style=none] (64) at (4.75, 10.25) {$c_1$};
	\end{pgfonlayer}
	\begin{pgfonlayer}{edgelayer}
		\draw (9) to (8);
		\draw (8) to (7);
		\draw (7) to (6);
		\draw (6) to (5);
		\draw (5) to (4);
		\draw (4) to (3);
		\draw (3) to (2);
		\draw (2) to (1);
		\draw (1) to (0);
		\draw (0) to (9);
		\draw [style=dashed] (10) to (13);
		\draw [style=Blue-Thick] (13) to (14);
		\draw [style=Blue-Thick] (14) to (15);
		\draw (15) to (12);
		\draw [style=dashed] (12) to (11);
		\draw [style=dashed] (11) to (10);
		\draw (12) to (21);
		\draw [style=Blue-Thick] (21) to (20);
		\draw [style=Blue-Thick] (21) to (16);
		\draw [style=dashed] (16) to (17);
		\draw [style=dashed] (17) to (18);
		\draw [style=dashed] (18) to (19);
		\draw (19) to (20);
		\draw [style=Blue-Thick] (15) to (16);
		\draw [style=Blue-Thick] (20) to (23);
		\draw [style=dashed] (23) to (24);
		\draw [style=dashed] (24) to (25);
		\draw [style=dashed] (25) to (26);
		\draw (26) to (27);
		\draw [style=Blue-Thick] (27) to (22);
		\draw [style=Blue-Thick] (22) to (23);
		\draw (22) to (19);
		\draw (26) to (28);
		\draw [style=Blue-Thick] (28) to (33);
		\draw [style=Blue-Thick] (33) to (27);
		\draw [style=dashed] (33) to (32);
		\draw [style=dashed] (32) to (31);
		\draw [style=dashed] (31) to (30);
		\draw [style=dashed] (30) to (29);
		\draw [style=Blue-Thick] (29) to (28);
		\draw [style=Blue-Thick] (29) to (34);
		\draw [style=dashed] (34) to (39);
		\draw [style=dashed] (39) to (30);
		\draw [style=Blue-Thick] (34) to (35);
		\draw (35) to (36);
		\draw [style=dashed] (36) to (37);
		\draw [style=dashed] (37) to (38);
		\draw [style=dashed] (38) to (39);
		\draw [style=Blue-Thick] (35) to (40);
		\draw [style=Blue-Thick] (40) to (45);
		\draw (45) to (36);
		\draw [style=Blue-Thick] (45) to (44);
		\draw [style=dashed] (44) to (43);
		\draw [style=dashed] (43) to (42);
		\draw [style=dashed] (42) to (41);
		\draw [style=dashed] (41) to (40);
		\draw [style=dashed] (43) to (51);
		\draw [style=dashed] (51) to (46);
		\draw [style=Blue-Thick] (46) to (44);
		\draw [style=Blue-Thick] (46) to (47);
		\draw [style=dashed] (47) to (48);
		\draw [style=dashed] (48) to (49);
		\draw [style=dashed] (50) to (49);
		\draw [style=dashed] (50) to (51);
		\draw [style=dashed] (48) to (52);
		\draw [style=dashed] (52) to (57);
		\draw [style=Blue-Thick] (57) to (47);
		\draw [style=Blue-Thick] (57) to (56);
		\draw [style=Blue-Thick] (56) to (55);
		\draw [style=dashed] (55) to (54);
		\draw [style=dashed] (54) to (53);
		\draw [style=dashed] (53) to (52);
		\draw [style=spin] (8) to (0);
		\draw [style=spin] (0) to (7);
		\draw [style=spin] (7) to (1);
		\draw [style=spin] (1) to (5);
		\draw [style=spin] (5) to (3);
		\draw [style=spin] (5) to (2);
		\draw [style=spin] (1) to (6);
	\end{pgfonlayer}
\end{tikzpicture}}
\caption{Left: Example of a late-crossing canonical path. Right: Example of an early-crossing canonical path (with default orientation illustrated).}
\label{fig:M-path-example}
\end{figure}

\subsection{Zig-Zag Triangulation} \label{sec:zig-zag}

Let $T$ be a zig-zag triangulation with default orientation, and with fan centers labelled as $c_i=i$, 
as depicted in \Cref{fig:arrow-reverse}. Let $H_{a,b}$ denote the holonomy following one of the canonical 
paths from a vertex near $a=c_0$ (on the side closer to $c_1$) to a vertex near $b=c_{N+1}$ (on the side closer to $c_N$), as in \Cref{fig:M-path-example}.

\begin {remark} \label{rmk:early-late-path-products}
    The holonomy matrix obtained from the late path will be a product of $E$, $A$, and $\rho$ matrices.
    In particular, if we define
    $X_i := E_{i,i+1} A^i_{i-1,i+1} \rho$ and $Y_i := E_{i,i+1}^{-1} {A^i_{i-1,i+1}}^{-1} \rho$, then we will have the following forms for $H_{ab}$
    depending on the type\footnote{
        If $\epsilon_a=0$, the late path actualy starts $\cdots Y_2 E_{12}A^1_{02}E_{01}$. But since $\rho^2 = \mathrm{id}$ and $\rho E_{01} = E_{01}^{-1}$,
        this is equal to the more concise expression given in the table. Similarly when $\epsilon_a=1$, the path starts with
        $\cdots X_2 E_{12}^{-1} {A^1_{02}}^{-1} E_{01}^{-1}$, but this is equal to the product shown.
    }:
    \[
        \begin {array}{c|c|c} 
            H_{ab}         & \epsilon_b = 0                      & \epsilon_b = 1                      \\ \hline
            \epsilon_a = 0 & X_N \cdots Y_4X_3Y_2X_1 E_{01}^{-1} & Y_N \cdots Y_4X_3Y_2X_1 E_{01}^{-1} \\
            \epsilon_a = 1 & X_N \cdots X_4Y_3X_2Y_1 E_{01}      & Y_N \cdots X_4Y_3X_2Y_1 E_{01}
        \end {array}
    \]
    To describe the early path, define $\mathcal{X}_i := A^i_{i-1,i+1} E_{i-1,i} $ and $\mathcal{Y}_i := {A^i_{i-1,i+1}}^{-1} E_{i-1,i}^{-1} $. Then
    the early path will have the form
    \[
        \begin {array}{c|c|c} 
            H_{ab}         & \epsilon_b = 0                      & \epsilon_b = 1                      \\ \hline
            \epsilon_a = 0 & E_{N,N+1} \mathcal{X}_N \cdots \mathcal{Y}_4\mathcal{X}_3\mathcal{Y}_2\mathcal{X}_1 & E_{N,N+1}^{-1} \mathcal{Y}_N \cdots \mathcal{Y}_4\mathcal{X}_3\mathcal{Y}_2\mathcal{X}_1 \\
            \epsilon_a = 1 & E_{N,N+1} \mathcal{X}_N \cdots \mathcal{X}_4\mathcal{Y}_3\mathcal{X}_2\mathcal{Y}_1 & E_{N,N+1}^{-1} \mathcal{Y}_N \cdots \mathcal{X}_4\mathcal{Y}_3\mathcal{X}_2\mathcal{Y}_1
        \end {array}
    \]
\end {remark}

Our proof has two parts:
the first is an induction via left matrix multiplication, proving the first two columns of the holonomy matrix formula, 
and the second is an induction via right matrix multiplication, proving the formula for the first two rows.
As mentioned earlier, the expression for the $(3,3)$-entry follows immediately from \Cref{eq:osp1}.

\subsubsection{Proof for the first two columns} \label{sec:1st-2-columns}
We first induct by left-multiplication, which corresponds to flipping certain diagonals from top to bottom. 
Recall that performing a quadrilateral flip will alter the orientation of another edge, so it is important that we 
keep track of the arrows as we perform a sequence of flips. For a zig-zag triangulation with default orientation $1\to 2\to\cdots\to N$, 
the natural flip sequence is from bottom to top\footnote{This is called the \emph{default flip sequence} in \cite{moz21}.}, 
because every flip in this sequence will not alter the arrows of other un-flipped edges. 
On the other hand, if we flip from top to bottom, certain steps in this sequence will change the orientation of other edges 
that are not yet flipped, which makes it difficult to keep track of the orientations. Therefore, we need to ``manually'' 
reverse all the arrows so that the top-to-bottom flip has the desired property.

It is explained in \cite{pz_19} that reversing the arrows of a triangle and negating the $\mu$-invariants is an equivalence of spin structure, 
so we start by applying this equivalence move on every even numbered triangles, i.e. $c_{i-1},c_i,c_{i+1}$ for even $i$. 
This will negate all the $\theta_i$'s for even $i$ and turn the orientation in to the reversed default orientation as desired. See \Cref{fig:arrow-reverse} for illustration.

\begin{figure}[h]
\centering
\scalebox{0.85}{
	\begin{tikzpicture}[decoration={
    markings,
    mark=at position 0.5 with {\arrow{latex}}}]
	\begin{pgfonlayer}{nodelayer}
		\node [style=bullet mid] (0) at (-2.25, -6) {};
		\node [style=bullet mid] (1) at (2.25, -6) {};
		\node [style=bullet mid] (2) at (-5.25, -3.5) {};
		\node [style=bullet mid] (3) at (5.25, -3.5) {};
		\node [style=bullet mid] (4) at (-6.25, 0) {};
		\node [style=bullet mid] (5) at (6.25, 0) {};
		\node [style=bullet mid] (6) at (-5.25, 3.5) {};
		\node [style=bullet mid] (7) at (5.25, 3.5) {};
		\node [style=bullet mid] (8) at (-2.25, 6) {};
		\node [style=bullet mid] (9) at (2.25, 6) {};
		\node [style=none] (10) at (-2.5, 6.5) {$0$};
		\node [style=none] (11) at (2.5, 6.5) {$1$};
		\node [style=none] (12) at (-5.75, 3.75) {$2$};
		\node [style=none] (13) at (5.75, 3.75) {$3$};
		\node [style=none] (14) at (-6.75, 0) {$4$};
		\node [style=none] (15) at (6.75, 0) {$5$};
		\node [style=none] (16) at (-5.75, -3.75) {$6$};
		\node [style=none] (17) at (6, -4) {$N-1$};
		\node [style=none] (18) at (-2.5, -6.5) {$N$};
		\node [style=none] (19) at (2.5, -6.5) {$N+1$};
		\node [style=none] (20) at (-1.75, 5.25) {$\theta_1$};
		\node [style=none] (21) at (1, 4.5) {$\theta_2$};
		\node [style=none] (22) at (-2.75, 2.25) {$\theta_3$};
		\node [style=none] (23) at (2.75, 1.25) {$\theta_4$};
		\node [style=none] (24) at (-2.5, -1.5) {$\theta_5$};
		\node [style=none] (25) at (1.5, -5.5) {$\theta_{N}$};
	\end{pgfonlayer}
	\begin{pgfonlayer}{edgelayer}
		\draw (6) to (8);
		\draw (8) to (9);
		\draw (9) to (7);
		\draw (7) to (5);
		\draw (6) to (4);
		\draw (4) to (2);
		\draw (0) to (1);
		\draw (1) to (3);
		\draw [style=dashed] (5) to (3);
		\draw [style=dashed] (0) to (2);
		\draw [style=spin] (9) to (6);
		\draw [style=spin] (6) to (7);
		\draw [style=spin] (7) to (4);
		\draw [style=spin] (4) to (5);
		\draw [style=spin] (5) to (2);
		\draw [style=spin] (3) to (0);
		\draw [style=dashed] (2) to (3);
	\end{pgfonlayer}
\end{tikzpicture}
	\begin{tikzpicture}[baseline,scale=1]
	    \draw[->, thick](0,0)--(1,0);
	    \node[above]  at (0.5,0) {};
	\end{tikzpicture}
	\begin{tikzpicture}[decoration={
    markings,
    mark=at position 0.5 with {\arrow{latex}}}]
	\begin{pgfonlayer}{nodelayer}
		\node [style=bullet mid] (0) at (-2.25, -6) {};
		\node [style=bullet mid] (1) at (2.25, -6) {};
		\node [style=bullet mid] (2) at (-5.25, -3.5) {};
		\node [style=bullet mid] (3) at (5.25, -3.5) {};
		\node [style=bullet mid] (4) at (-6.25, 0) {};
		\node [style=bullet mid] (5) at (6.25, 0) {};
		\node [style=bullet mid] (6) at (-5.25, 3.5) {};
		\node [style=bullet mid] (7) at (5.25, 3.5) {};
		\node [style=bullet mid] (8) at (-2.25, 6) {};
		\node [style=bullet mid] (9) at (2.25, 6) {};
		\node [style=none] (10) at (-2.5, 6.5) {$0$};
		\node [style=none] (11) at (2.5, 6.5) {$1$};
		\node [style=none] (12) at (-5.75, 3.75) {$2$};
		\node [style=none] (13) at (5.75, 3.75) {$3$};
		\node [style=none] (14) at (-6.75, 0) {$4$};
		\node [style=none] (15) at (6.75, 0) {$5$};
		\node [style=none] (16) at (-5.75, -3.75) {$6$};
		\node [style=none] (17) at (6, -4) {$N-1$};
		\node [style=none] (18) at (-2.5, -6.5) {$N$};
		\node [style=none] (19) at (2.5, -6.5) {$N+1$};
		\node [style=none] (20) at (-1.75, 5.25) {$\theta_1$};
		\node [style=none] (21) at (1, 4.5) {$-\theta_2$};
		\node [style=none] (22) at (-2.75, 2.25) {$\theta_3$};
		\node [style=none] (23) at (2.75, 1.25) {$-\theta_4$};
		\node [style=none] (24) at (-2.5, -1.5) {$\theta_5$};
		\node [style=none] (25) at (1.5, -5.5) {$\pm\theta_{N}$};
	\end{pgfonlayer}
	\begin{pgfonlayer}{edgelayer}
		\draw (6) to (8);
		\draw (8) to (9);
		\draw (9) to (7);
		\draw (7) to (5);
		\draw (6) to (4);
		\draw (4) to (2);
		\draw (0) to (1);
		\draw (1) to (3);
		\draw [style=dashed] (5) to (3);
		\draw [style=dashed] (0) to (2);
		\draw [style=dashed] (2) to (3);
		\draw [style=spin] (0) to (3);
		\draw [style=spin] (6) to (9);
		\draw [style=spin] (7) to (6);
		\draw [style=spin] (4) to (7);
		\draw [style=spin] (5) to (4);
		\draw [style=spin] (2) to (5);
	\end{pgfonlayer}
\end{tikzpicture}
}
\caption{Reversing the default orientation of a zig-zag triangulation.}
\label{fig:arrow-reverse}
\end{figure}
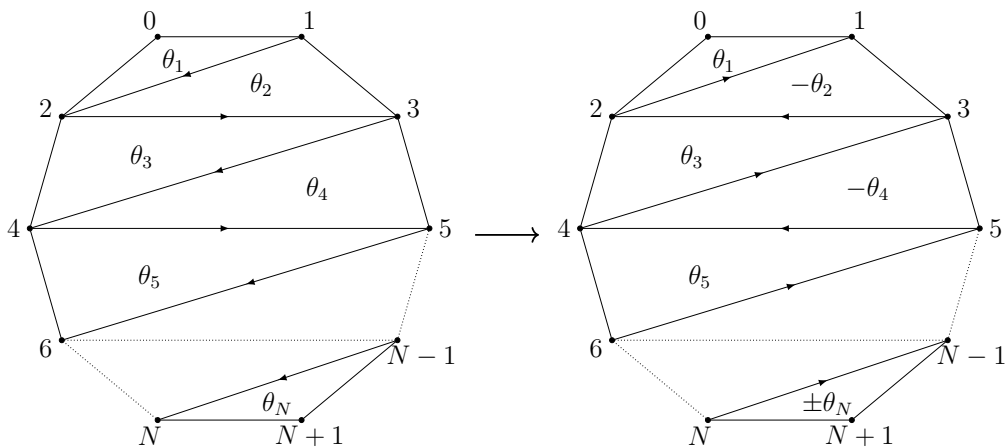

\begin{proof}
We show the proof of \Cref{thm:generic} only in the case when
$c_0,c_1,c_2$ are oriented clockwise (i.e. $\epsilon_a = 0$),
noting that the argument can be checked in a similar manner when $\epsilon_a = 1$.
We will induct on $N$, the number of triangles.

\paragraph{\textbf{Base Case.}} For the base case, we have a single triangle. In the notation from earlier in \Cref{sec:zig-zag}, we are
computing $H_{02}$. Using \Cref{thm:single-fan} with the specialization $N=1$
and using the labelling of \Cref{fig:single-fan} (so that  $c_0=2$, $c_1=1$, and $c_2=3$) yields the desired matrix $H_{a,b}$ after the proper substitutions, 
noting e.g. that $\lambda_{ij}$ and $\boxed{ijk}$ equals zero when $i=j$.

Next, when $k > 1$, we assume that the formula holds for $H_{0,k}$ and prove it for $H_{0,k+1}$. The induction step will be slightly different when $k$ is even or odd.
Note that if $k$ is even (resp. odd), then the holonomy $H_{0,k}$ is of type 00 (resp. type 01) and $ H_{0,k+1}$ is of type 01 (resp. type 00).

\textbf{Inductive step for $k$ even.}
By the induction hypothesis, we have
\[H_{0,k}  =\ospmatrix {-{\lambda_{1  k}\over \lambda_{0 1}}} {\lambda_{0k}} {\star}
	{{\l{1}{ {k-1}}\over \lambda_{ 0 1}\l { {k-1}}{k}}} {-{\l 0 {k-1}\over\l{k-1}{k}}} {\star}
	{{1\over\lambda_{01}}\td {1} { {k-1}}{ k}} {-\td { 0}{ {k-1}}{ {k}}} {\star}
\]
Next we compute $H_{0,k+1} = Y_k H_{0,k} = E_{k,k+1}^{-1} {A^k_{k-1,k+1}}^{-1} \rho H_{0,k}$:
\begin{equation}\label{eqn:two_col}
    \ospmatrix
    { 
        - \left( \frac{\l{1}{k}\l{k-1}{k+1}+\l{1}{k-1}\l{k}{k+1}} {\l{0}{1}\l{k-1}{k}}
        + \frac{\td{1}{k-1}{k} \td{k+1}{k-1}{k}} {\l{0}{1}} \right)
    }
    { 
        \frac{\l{0}{k}\l{k-1}{k+1} + \l{0}{k-1}\l{k}{k+1}} {\l{k-1}{k}}
        - \td{k+1}{k-1}{k}\td{0}{k-1}{k}
    }
    {\star} 
    { 
        -{\l{ 1}{ k} \over \lambda_{{01}} \l{{k}}{{k+1}} } 
    } 
    { 
        \lambda_{{0,k}}\over \lambda_{{k,k+1} } 
    }
    {\star} 
    { 
        {\lambda_{1k}\over\lambda_{01}} \left(\tu{ k }{ 1 } { {k-1} }-\tu{ {k} }{ {k-1} }{ {k+1} }  \right) 
    } 
    { 
        \l{0}{k} \left( \tu{ {k} }{ {k-1} }{ {k+1} } - \tu{ {k} }{ {0} }{ {k-1} } \right) 
    }
    {\star}
\end{equation}

Note that the expressions in the first column are the Ptolemy relations (c.f. \Cref{rmk:alternate_ptolemy}) 
on the quadrilateral $({1},{k},{k-1},{k+1})$. So matrix multiplication in the first column corresponds to flipping the edges 
$({2},{3}),({3},{4}), \cdots, (k-1,k)$. Recall that the flips in this sequence will not alter
the orientation of other (un-flipped) edges. So for every even $k$, 
the quadrilateral flip is depicted as follows, where the $\mu$-invariants associated to the triangle ${k-1},{k},{k+1}$ are negated.
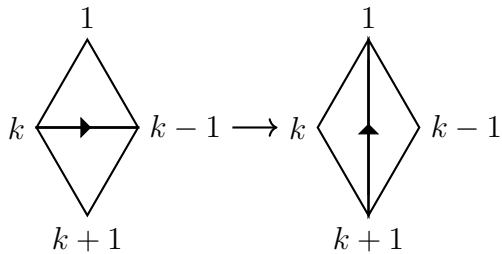
\[\begin{tikzpicture}[scale=0.45, baseline, thick]
    \draw (0,0)--(3,0)--(60:3)--cycle;
    \draw (0,0)--(3,0)--(-60:3)--cycle;
    \draw (0,0)--node {\midarrow} (3,0);
    \draw node[left] at (0,0) {${k}$};
    \draw node[above] at (60:3) {${1}$};
    \draw node[right] at (3,0) {${k-1}$};
    \draw node[below] at (-60:3) {${k+1}$};
    \end{tikzpicture}
    \begin{tikzpicture}[baseline,scale=0.6]
    \draw[->, thick](0,0)--(1,0);
    \node[above]  at (0.5,0) {};
    \end{tikzpicture}
    \begin{tikzpicture}[scale=0.45, baseline, thick,every node/.style={sloped,allow upside down}]
    \draw (0,0)--(60:3)--(-60:3)--cycle;
    \draw (3,0)--(60:3)--(-60:3)--cycle;
    \draw (1.5,-2) --node {\midarrow} (1.5,2);
    \draw node[left] at (0,0) {${k}$};
    \draw node[above] at (60:3) {${1}$};
    \draw node[right] at (3,0) {${k-1}$};
    \draw node[below] at (-60:3) {${k+1}$};
\end{tikzpicture}
\]
Thus by \Cref{eqn:super_ptolemy_lambda_star,eqn:super_ptolemy_mu_right_star}, we have
\begin{align*}
    \l{1}{k+1}     &= \frac{\l{1}{k} \l{k-1}{k+1} + \l{1}{k-1} \l{k}{k+1}} {\l{k-1}{k}} + (-\td{k+1}{k-1}{k}) \td{1}{k-1}{k} \\
    \tu{k}{1}{k+1} &= (-\tu{k}{k-1}{k+1}) + \tu{k}{1}{k-1}
\end{align*}
Now for the second column, the matrix multiplication corresponds to flipping the edges $(1,2),(2,3),\cdots,(k-1,k)$. 
Similar to the previous case, the the quadrilateral flip is depicted as follows
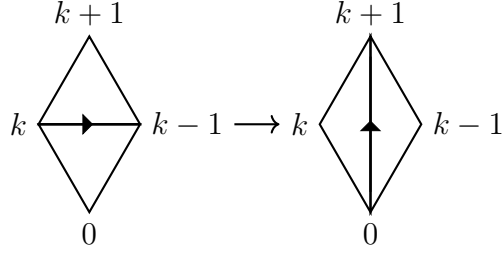
\[
\begin{tikzpicture}[scale=0.45, baseline, thick]
    \draw (0,0)--(3,0)--(60:3)--cycle;
    \draw (0,0)--(3,0)--(-60:3)--cycle;
    \draw (0,0)--node {\midarrow} (3,0);
    \draw node[left] at (0,0) {${k}$};
    \draw node[above] at (60:3) {${k+1}$};
    \draw node[right] at (3,0) {${k-1}$};
    \draw node[below] at (-60:3) {${0}$};
    \end{tikzpicture}
    \begin{tikzpicture}[baseline,scale=0.6]
    \draw[->, thick](0,0)--(1,0);
    \node[above]  at (0.5,0) {};
    \end{tikzpicture}
    \begin{tikzpicture}[scale=0.45, baseline, thick,every node/.style={sloped,allow upside down}]
    \draw (0,0)--(60:3)--(-60:3)--cycle;
    \draw (3,0)--(60:3)--(-60:3)--cycle;
    \draw (1.5,-2) --node {\midarrow} (1.5,2);
    \draw node[left] at (0,0) {${k}$};
    \draw node[above] at (60:3) {${k+1}$};
    \draw node[right] at (3,0) {${k-1}$};
    \draw node[below] at (-60:3) {${0}$};
\end{tikzpicture}
\]
and the Ptolemy relations are
\begin{align*}
    \l{0}{k+1}      &= {\l{0}{k}\l{k-1}{k+1} +\l{0}{k-1}\l{k}{k+1} \over \l{k-1}{k}} + (-\td{k}{k-1}{k+1})\td{0}{k-1}{k} \\
    \tu{k}{ 0}{k+1} &= (-\tu{k}{k-1}{k+1})+\tu{k}{0}{k-1}
\end{align*}

Now plugging these back into Matrix (\ref{eqn:two_col}), and using \Cref{rmk:normalized_mu_relations} (iv) twice, we get
\[
    H_{ 0, {k+1}} =
    \ospmatrix {-{\lambda_{1, { k+1}}\over \lambda_{ 0  1}}} {\lambda_{ 0,{k+1}}} {\star}
            {-{\l{ 1}{ {k}}\over \lambda_{ 0 1}\l { {k}}{{k+1}}}} {{\l{ 0}{{k}}\over\l{{k}}{ {k+1}}}} {\star}
            {{\lambda_{1k}\over\lambda_{ 0 1}}\tu {k} { {1}}{ k+1}} {-\lambda_{0k}\tu {k}{0}{ {k+1}}} {\star}
     = \ospmatrix {-{\lambda_{1, { k+1}}\over \lambda_{ 0  1}}} {\lambda_{ 0,{k+1}}} {\star}
            {-{\l{ 1}{ {k}}\over \lambda_{ 0 1}\l { {k}}{{k+1}}}} {{\l{ 0}{{k}}\over\l{{k}}{ {k+1}}}} {\star}
            {{1\over\lambda_{ 0 1}}\td { 1} { {k}}{ k+1}} {-\td { 0}{ {k}}{ {k+1}}} {\star}.
\]
This agrees with the formula of type 01 holonomy matrix.

\textbf{Induction step for $k$ odd.}
By the induction hypothesis, we have
\[
    H_{0,k} = \ospmatrix 
              {-{\lambda_{1, { k}}\over \lambda_{ 0  1}}} 
              {\lambda_{0,k}} 
              {\star}
	      {-{\l{ 1}{ {k-1}}\over \lambda_{ 0 1}\l { {k-1}}{{k}}}} 
              {{\l{ 0}{{k-1}}\over\l{{k-1}}{ {k}}}} 
              {\star}
	      {{1\over\lambda_{ 0 1}}\td { 1} { {k-1}}{ k}} 
              {-\td { 0}{ {k-1}}{ {k}}} 
              {\star}
\]
Then we compute $H_{0,k+1} = X_k H_{0k} = E_{k,k+1} A^k_{k-1,k+1} \rho H_{0,k}$:
\begin{equation}\label{eqn:two_col_1}
 \ospmatrix
 {-\left({\l{{1}}{{k}}\l{{k-1}}{{k+1}}+\l{{1}}{{k-1}}\l{{k}}{{k+1}}
  \over \lambda_{{01}}\lambda_{{k-1,k}}} + {\td{ 1}{{k-1}}{{k}}
 \td{{k}}{ k-1}{{k+1}}\over \lambda_{01}}\right) }
 {
 {\l{{0}}{{k}}\l{{k-1}}{{k+1}}+\l{{0}}{{k-1}}\l{{k}}{{k+1}}\over \lambda_{k-1,k}} +\td{ 0}{{k-1}}{{k}}
 \td{{k}}{ k-1}{{k+1}}
 }
 {\star}
 {{\l{ 1}{ k} \over \lambda_{{01}} \l{{k}}{{k+1}} } } {-\lambda_{{0,k}}\over \lambda_{{k,k+1} } }{\star}
 { { \lambda_{1k}\over\lambda_{01}} \left(\tu{ k }{ 1 } { {k-1} }+\tu{ {k} }{ {k-1} }{ {k+1} }  \right) } {
 {-{\lambda_{0k}}} \left( \tu{ {k} }{ {k-1} }{ {k+1} } + \tu{ {k} }{ {0} }{ {k-1} } \right) 
 }{\star}\end{equation} 
Similar to the $k$ even case, the first column are Ptolemy relations corresponding to flipping the edges $(2,3),(3,4),\cdots,(k-1,k)$. 
The last flip is depicted as follows. Note that in this case, the $\mu$-invariant associated to the triangle $(k-1,k,k+1)$ is not negated.
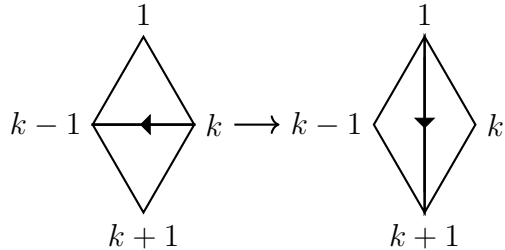
\[
		\begin{tikzpicture}[scale=0.45, baseline, thick]
		\draw (0,0)--(3,0)--(60:3)--cycle;
		\draw (0,0)--(3,0)--(-60:3)--cycle;
		\draw (0,0)--node {\midrevarrow} (3,0);
		\draw node[left] at (0,0) {${k-1}$};
		\draw node[above] at (60:3) {${1}$};
		\draw node[right] at (3,0) {${k}$};
		\draw node[below] at (-60:3) {${k+1}$};
		\end{tikzpicture}
		\begin{tikzpicture}[baseline,scale=0.6]
		\draw[->, thick](0,0)--(1,0);
		\node[above]  at (0.5,0) {};
		\end{tikzpicture}
		\begin{tikzpicture}[scale=0.45, baseline, thick,every node/.style={sloped,allow upside down}]
		\draw (0,0)--(60:3)--(-60:3)--cycle;
		\draw (3,0)--(60:3)--(-60:3)--cycle;
		\draw (1.5,-2) --node {\midrevarrow} (1.5,2);
		\draw node[left] at (0,0) {${k-1}$};
		\draw node[above] at (60:3) {${1}$};
		\draw node[right] at (3,0) {${k}$};
		\draw node[below] at (-60:3) {${k+1}$};
		\end{tikzpicture}
\]
Thus by \Cref{eqn:super_ptolemy_lambda_star,eqn:super_ptolemy_mu_left_star}, we have
\begin{align*}
    \l{1}{k+1}     &= \frac{\l{1}{k}\l{k-1}{k+1} + \l{1}{k-1}\l{k}{k+1}}{\l{k-1}{k}} + \td{1}{k-1}{k}\td{k+1}{k-1}{k} \\
    \tu{k}{1}{k+1} &= \tu{k}{k-1}{k+1} + \tu{k}{1}{k-1}
\end{align*}
The matrix multiplication for the second column corresponds to flipping the edges, in order, $(1,2),(2,3),\cdots,(k-1,k)$, where the the quadrilateral flip is depicted as follows
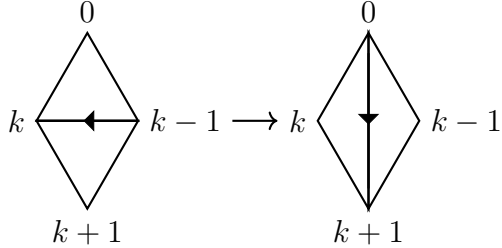
\[
		\begin{tikzpicture}[scale=0.45, baseline, thick]
		\draw (0,0)--(3,0)--(60:3)--cycle;
		\draw (0,0)--(3,0)--(-60:3)--cycle;
		\draw (0,0)--node {\midrevarrow} (3,0);
		\draw node[left] at (0,0) {${k}$};
		\draw node[above] at (60:3) {${0}$};
		\draw node[right] at (3,0) {${k-1}$};
		\draw node[below] at (-60:3) {${k+1}$};
		\end{tikzpicture}
		\begin{tikzpicture}[baseline,scale=0.6]
		\draw[->, thick](0,0)--(1,0);
		\node[above]  at (0.5,0) {};
		\end{tikzpicture}
		\begin{tikzpicture}[scale=0.45, baseline, thick,every node/.style={sloped,allow upside down}]
		\draw (0,0)--(60:3)--(-60:3)--cycle;
		\draw (3,0)--(60:3)--(-60:3)--cycle;
		\draw (1.5,-2) --node {\midrevarrow} (1.5,2);
		\draw node[left] at (0,0) {${k}$};
		\draw node[above] at (60:3) {${0}$};
		\draw node[right] at (3,0) {${k-1}$};
		\draw node[below] at (-60:3) {${k+1}$};
		\end{tikzpicture}
\]
and the Ptolemy relations are
\begin{align*}
    \l{0}{k+1}     &= \frac{\l{0}{k}\l{k-1}{k+1} + \l{0}{k-1}\l{k}{k+1}}{\l{k-1}{k}} + \td{0}{k-1}{k}\td{k+1}{k-1}{k} \\
    \tu{k}{0}{k+1} &= \tu{k}{k-1}{k+1} + \tu{k}{0}{k-1}
\end{align*}
Now plugging these Ptolemy relations into Matrix (\ref{eqn:two_col_1}) we get
\[
    H_{ 0, {k+1}} =
    \ospmatrix {-{\lambda_{1, { k+1}}\over \lambda_{ 0  1}}} {\lambda_{ 0,{k+1}}} {\star}
            {{\l{ 1}{ {k}}\over \lambda_{ 0 1}\l { {k}}{{k+1}}}} {-{\l{ 0}{{k}}\over\l{{k}}{ {k+1}}}} {\star}
            {{\lambda_{1k}\over\lambda_{ 0 1}}\tu { k} { 1}{ k+1}} {-\lambda_{0k}\tu { k}{0}{ {k+1}}} {\star}
    = \ospmatrix {-{\lambda_{1, { k+1}}\over \lambda_{ 0  1}}} {\lambda_{ 0,{k+1}}} {\star}
            {{\l{ 1}{ {k}}\over \lambda_{ 0 1}\l { {k}}{{k+1}}}} {-{\l{ 0}{{k}}\over\l{{k}}{ {k+1}}}} {\star}
            {{1\over\lambda_{ 0 1}}\td { 1} { {k}}{ k+1}} {-\td { 0}{ {k}}{ {k+1}}} {\star}
\]
This agrees with the formula of type 00 holonomy matrix.
\end{proof}

\subsubsection{Proof for the first two rows.} \label{sec:1st-2-rows}

Next we turn to the induction for the first two rows via right multiplication. In this case we will use the early path. 

It turns out that induction by right multiplication corresponds to flipping the diagonals from bottom to top,  
as opposed to the previous case. This already has the property that each flip does not alter the orientation of other unflipped edges. 
Therefore here we do not need the extra step of reversing all the arrows.

\begin{proof}
We illustrate the proof in the case that
$c_{N-1}, c_N, c_{N+1}$ are oriented counterclockwise, i.e. the path is of type $01$ or $11$. In this case the holonomy matrix looks like
\[
    H_{0,N+1} = E_{N,N+1}^{-1} \mathcal{Y}_N \cdots \mathcal{Y}_2 \mathcal{X}_1 \quad \text{or} \quad 
                E_{N,N+1}^{-1} \mathcal{Y}_N \cdots \mathcal{X}_2 \mathcal{Y}_1,
\]
the former for type $01$ and the latter for type $11$
(recall that we are using the early path).

The base case is a triangle $H_{N-1,N+1} = E_{N,N+1}^{-1}A^N_{N-1,N+1}E_{N-1,N}$, which can be verified using \Cref{thm:single-fan} with the specialization $N=1$.

We now assume by induction that the formula holds for $H_{k,N+1}$ for some $k<N$. 

\textbf{Induction step for $N-k$ even.}
If $N-k$ is even, i.e. there are even number of triangles 
in the sub-triangulation spanned by the diagonal $(k,N+1)$, the holonomy $H_{k,N+1}$ is of type $01$. Thus by induction hypothesis we have
\[
    H_{k,N+1} = \ospmatrix
    {-\frac{\l{k+1}{N+1}}{\l{k}{k+1}}}
    {\l{k}{N+1}}
    {\td{N+1}{k}{k+1}}
    {-{\l{k+1}{N} \over \l{k}{k+1}\l{N}{N+1}}}
    {{\l{k}{N} \over \l{N}{N+1}}}
    {{1 \over \l{N}{N+1}}\td{N}{k}{k+1}}
    {\star}
    {\star}
    {\star}.
\]
Then we compute $H_{k-1,N+1} = H_{k,N+1} \mathcal{Y}_k = H_{k,N+1} {A^k_{k-1,k+1}}^{-1} E_{k-1,k}^{-1}$:
\begin{equation}\label{eqn:two-row}
    \ospmatrix
    {-{\l{k}{N+1} \over \l{k-1}{k}} }
    {-\left({ \l{k-1}{k} \l{k+1}{N+1} + \l{k-1}{k+1}\l{k}{N+1} \over \l{k}{k+1} } + \td{N+1}{k}{k+1}\td{k-1}{k}{k+1} \right) }
    {\l{k}{N+1}\left(\tu{k}{k+1}{N+1}+\tu{k}{k-1}{k+1}\right)}%
    { -{\l{k}{N} \over \l{k-1}{k}\l{N}{N+1}}  }
    { -\left({ \l{k-1}{k} \l{k+1}{N} + \l{k-1}{k+1}\l{k}{N} \over \l{k}{k+1} \l{N}{N+1} } + {\td{N}{k}{k+1}\td{k-1}{k}{k+1} \over \l{N}{N+1} } \right)  }
    { {\l{k}{N} \over \l{N}{N+1} } \left(\tu{k}{k+1}{N} + \tu{k}{k-1}{k+1}\right) }
    {\star}{\star}{\star}
\end{equation}

The expression on the first row are Ptolemy relations from flipping, starting from the bottom, the diagonals $(N,N+1),(N-1,N),\cdots,(k,k+1)$. 
The final flip in the sequence is depicted as follows:
\begin{center}
\begin{tikzpicture}[scale=0.45, baseline, thick]
    \draw (0,0)--(3,0)--(60:3)--cycle;
    \draw (0,0)--(3,0)--(-60:3)--cycle;
    \draw (0,0)--node {\midarrow} (3,0);
    \draw node[left] at (0,0) {${k}$};
    \draw node[above] at (60:3) {${k-1}$};
    \draw node[right] at (3,0) {${k+1}$};
    \draw node[below] at (-60:3) {${N+1}$};
    \end{tikzpicture}
    \begin{tikzpicture}[baseline,scale=0.6]
    \draw[->, thick](0,0)--(1,0);
    \node[above]  at (0.5,0) {};
    \end{tikzpicture}
    \begin{tikzpicture}[scale=0.45, baseline, thick,every node/.style={sloped,allow upside down}]
    \draw (0,0)--(60:3)--(-60:3)--cycle;
    \draw (3,0)--(60:3)--(-60:3)--cycle;
    \draw (1.5,-2) --node {\midarrow} (1.5,2);
    \draw node[left] at (0,0) {${k}$};
    \draw node[above] at (60:3) {${k-1}$};
    \draw node[right] at (3,0) {${k+1}$};
    \draw node[below] at (-60:3) {${N+1}$};
\end{tikzpicture}
\end{center}
\Cref{eqn:super_ptolemy_lambda_star,eqn:super_ptolemy_mu_left_star} gives us
\begin{align*}
    \l{k-1}{N+1}     &= { \l{k-1}{k} \l{k+1}{N+1} + \l{k-1}{k+1}\l{k}{N+1} \over \l{k}{k+1} } + \td{N+1}{k}{k+1}\td{k-1}{k}{k+1} \\
    \tu{k}{k-1}{N+1} &= \tu{k}{k-1}{k+1} + \tu{k}{k+1}{N+1}
\end{align*}

The expressions on the second row come analogously from flipping the diagonals \\ $(N-1,N),(N-2,N-1),\cdots, (k,k+1)$, where the final flip is depicted as follows.
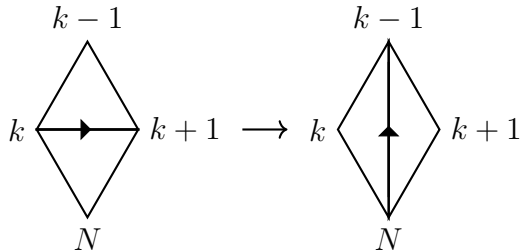
\begin{center}
\begin{tikzpicture}[scale=0.45, baseline, thick]
    \draw (0,0)--(3,0)--(60:3)--cycle;
    \draw (0,0)--(3,0)--(-60:3)--cycle;
    \draw (0,0)--node {\midarrow} (3,0);
    \draw node[left] at (0,0) {${k}$};
    \draw node[above] at (60:3) {${k-1}$};
    \draw node[right] at (3,0) {${k+1}$};
    \draw node[below] at (-60:3) {${N}$};
    \end{tikzpicture}
    \begin{tikzpicture}[baseline,scale=0.6]
    \draw[->, thick](0,0)--(1,0);
    \node[above]  at (0.5,0) {};
    \end{tikzpicture}
    \begin{tikzpicture}[scale=0.45, baseline, thick,every node/.style={sloped,allow upside down}]
    \draw (0,0)--(60:3)--(-60:3)--cycle;
    \draw (3,0)--(60:3)--(-60:3)--cycle;
    \draw (1.5,-2) --node {\midarrow} (1.5,2);
    \draw node[left] at (0,0) {${k}$};
    \draw node[above] at (60:3) {${k-1}$};
    \draw node[right] at (3,0) {${k+1}$};
    \draw node[below] at (-60:3) {${N}$};
\end{tikzpicture}
\end{center}
Here the Ptolemy relations are
\begin{align*}
    \l{k-1}{N}     &= { \l{k-1}{k} \l{k+1}{N} + \l{k-1}{k+1}\l{k}{N} \over \l{k}{k+1} } + \td{N}{k}{k+1}\td{k-1}{k}{k+1} \\
    \tu{k}{k-1}{N} &= \tu{k}{k-1}{k+1} + \tu{k}{k+1}{N}
\end{align*}
Now Plugging these relations into \Cref{eqn:two-row}, we get
\[
    H_{k-1,N+1}= \ospmatrix
    {-{\l k {N+1}\over\l {k-1}k} }%
    {-{\l{k-1}{N+1}}}%
    { \td{N+1}{k-1}{k} }%
    { -{\lambda_{k,N}\over\l {k-1}k\l{N}{N+1} }  }%
    { -{\l{k-1}{N} \over\l{N}{N+1} }  }%
    { {1\over \l {N}{N+1} }\td{N}{k-1}{k} }%
    {\star}{\star}{\star}
\] 
which matches the formula for type 11 holonomy matrix.

\textbf{Induction step for $N-k$ odd.}
Next we turn to the case when $N-k$ is odd, where $H_{k,N+1}$ is of type $11$ and $H_{k-1,N+1}$ is of type $01$.
By the induction hypothesis we have
\[
    H_{k,N+1} =  \ospmatrix
    {-{\l {k+1} {N+1}\over\l {k}{k+1} } }%
    {-{\l{k}{N+1}}}%
    { \td{N+1}{k}{k+1} }%
    { -{\lambda_{k+1,N}\over\l {k}{k+1}\l{N}{N+1} }  }%
    { -{\l{k}{N} \over\l{N}{N+1} }  }%
    { {1\over \l {N}{N+1} }\td{N}{k}{k+1} }%
    {\star}{\star}{\star}
\]
Then we calculate $H_{k-1,N+1} = H_{k,N+1} X_k = H_{k,N+1} A^k_{k-1,k+1} E_{k-1,k}$:
\begin{equation} \label{eqn:two-row-1}
    \ospmatrix
    {-{\l k {N+1}\over\l {k-1}k} }%
    {{ \l{k-1}k \l{k+1}{N+1}+\l{k-1}{k+1}\l k{N+1} \over \l k{k+1} }+\td{k-1}{k}{k+1}\td{N+1}{k}{k+1}  }%
    {\l{k}{N+1}\left(\tu{k}{k+1}{N+1}+\tu{k}{k-1}{k+1}\right)}%
    { -{\lambda_{k,N}\over\l {k-1}k\l{N}{N+1} }  }
    { { \l{k-1}k \l{k+1}{N}+\l{k-1}{k+1}\l k{N} \over \l k{k+1}\l N{N+1} }+{ \td{k-1}{k}{k+1}\td{N}{k}{k+1}\over\l{N}{N+1} }  }
    { {\l{k}{N}\over \l N{N+1} }\left(\tu{k}{k+1}{N}+\tu{k}{k-1}{k+1}\right) }
    {\star}{\star}{\star}
\end{equation}
Similar to the previous case, the first row corresponds to the flip sequence given by the following $(N,N+1),(N-1,N),\cdots,(k,k+1)$ 
and the second row corresponds to the flip sequence $(N-1,N),(N-2,N-1),\cdots,(k,k+1) $, which are given by the following two diagrams respectively.
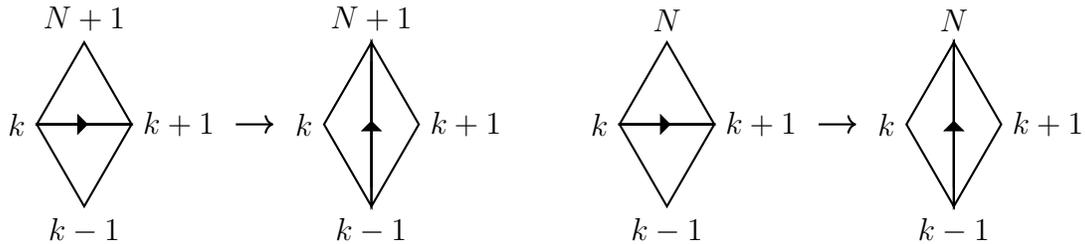
\begin{center}
\begin{tikzpicture}[scale=0.42, baseline, thick]
    \draw (0,0)--(3,0)--(60:3)--cycle;
    \draw (0,0)--(3,0)--(-60:3)--cycle;
    \draw (0,0)--node {\midarrow} (3,0);
    \draw node[left] at (0,0) {$k$};
    \draw node[above] at (60:3) {$N+1$};
    \draw node[right] at (3,0) {${k+1}$};
    \draw node[below] at (-60:3) {$k-1$};
\end{tikzpicture}
\begin{tikzpicture}[baseline,scale=0.5]
    \draw[->, thick](0,0)--(1,0);
    \node[above]  at (0.5,0) {};
\end{tikzpicture}
\begin{tikzpicture}[scale=0.42, baseline, thick,every node/.style={sloped,allow upside down}]
    \draw (0,0)--(60:3)--(-60:3)--cycle;
    \draw (3,0)--(60:3)--(-60:3)--cycle;
    \draw (1.5,-2) --node {\midarrow} (1.5,2);
    \draw node[left] at (0,0) {$k$};
    \draw node[above] at (60:3) {$N+1$};
    \draw node[right] at (3,0) {${k+1}$};
    \draw node[below] at (-60:3) {$k-1$};
\end{tikzpicture}
\quad\;\;
\begin{tikzpicture}[scale=0.42, baseline, thick]
    \draw (0,0)--(3,0)--(60:3)--cycle;
    \draw (0,0)--(3,0)--(-60:3)--cycle;
    \draw (0,0)--node {\midarrow} (3,0);
    \draw node[left] at (0,0) {$k$};
    \draw node[above] at (60:3) {$N$};
    \draw node[right] at (3,0) {${k+1}$};
    \draw node[below] at (-60:3) {$k-1$};
\end{tikzpicture}
\begin{tikzpicture}[baseline,scale=0.5]
    \draw[->, thick](0,0)--(1,0);
    \node[above]  at (0.5,0) {};
\end{tikzpicture}
\begin{tikzpicture}[scale=0.42, baseline, thick,every node/.style={sloped,allow upside down}]
    \draw (0,0)--(60:3)--(-60:3)--cycle;
    \draw (3,0)--(60:3)--(-60:3)--cycle;
    \draw (1.5,-2) --node {\midarrow} (1.5,2);
    \draw node[left] at (0,0) {$k$};
    \draw node[above] at (60:3) {$N$};
    \draw node[right] at (3,0) {${k+1}$};
    \draw node[below] at (-60:3) {$k-1$};
\end{tikzpicture}
\end{center}
The Ptolemy relations are
\begin{align*}
    \lambda_{k-1,N+1} &= { \l{k-1}{k} \l{k+1}{N+1}+\l{k-1}{k+1}\l k{N+1} \over \l k{k+1} }+\td{k-1}{k}{k+1}\td{N+1}{k}{k+1}	\\
    \tu k{k-1}{N+1} & = \tu k{k-1}{k+1}+\tu k{k+1}{N+1}\\
    \lambda_{k-1,N} &= { \l{k-1}k \l{k+1}{N}+\l{k-1}{k+1}\l k{N} \over \l k{k+1} }+\td{k-1}{k}{k+1}\td{N}{k}{k+1}	\\
    \tu k{k-1}N & = \tu k{k-1}{k+1}+\tu k{k+1}{N}
\end{align*}
Plugging into \Cref{eqn:two-row-1} we get
\[
    H_{k-1,N+1} =
    \ospmatrix
    {-{\l k {N+1}\over\l {k-1}k} }%
    {{\l{k-1}{N+1}}}%
    { \td{N+1}{k-1}{k} }%
    { -{\lambda_{k,n}\over\l {k-1}k\l{N}{N+1} }  }%
    { {\l{k-1}N \over\l{N}{N+1} }  }%
    { {1\over \l N{N+1} }\td{N}{k-1}{k} }%
    {\star}{\star}{\star}
\] 
which matches the formula for type 01 holonomy matrix.

We omit the proof of the other case when $c_{N-1},c_N,c_{N+1}$ are oriented clockwise.
\end{proof}

\begin {remark} \label{rmk:local_signs}
    In \Cref{sec:1st-2-rows}, it was noted that the matrix product computes the odd entries in the third column ($\td{N+1}{0}{1}$ and $\frac{1}{\lambda_{N,N+1}}\td{N}{0}{1}$)
    using a particular sequence of Ptolemy relations, which always uses  
    \Cref{eqn:super_ptolemy_mu_left_star}, rather than \Cref{eqn:super_ptolemy_mu_right_star}, from \Cref{rmk:alternate_ptolemy}.
    Therefore by Theorem 6.2(b) from \cite{moz22}, these odd elements, 
    when expressed as polynomials in the variables from the original triangulation, have all positive terms.
    Similarly, in \Cref{sec:1st-2-columns}, the matrix product computes the odd elements 
    from the third row ($\frac{1}{\lambda_{01}}\td{1}{N}{N+1}$ and $\td{0}{N}{N+1}$) using Ptolemy relation 
    \Cref{eqn:super_ptolemy_mu_right_star}, which can be 
    seen as an instance of \Cref{eqn:super_ptolemy_mu_left_star} after negating half of the odd variables, and reversing the orientations on all the diagonals.
    So although the polynomial expressions of these odd elements have some signs, 
    Theorem 6.2(b) from \cite{moz22} says that these expressions have all positive terms when expressed instead in the new variables $\theta_i' = (-1)^{i+1} \theta_i$.
\end {remark}

We have now completed the proof of \Cref{thm:generic} (and hence of \Cref{thm:12-entry}) for zig-zag triangulations with default orientation.

\subsection{Generic Triangulation}

The theorem for generic triangulations (with default orientation) is a direct consequence of the zig-zag case and \Cref{cor:product_of_A_matrices}. 
Write the holonomy as a product of matrices following one of the canonical paths. As we traverse trough the $i$-th fan segment, 
we can use \Cref{cor:product_of_A_matrices} to write the product of $A$-matrices in a fan as the single matrix ${A^{c_i}_{c_{i-1},c_{i+1}}}^{\pm 1}$. 
This is the same as flipping the diagonals inside each fan segment, which turns a generic triangulation into a zig-zag triangulation whose vertices 
are the original fan centers. See \cite[Figure 15]{moz21}.

What remains is to consider the case of an orientation $\tau$ of $T$ that is \emph{not} the default one.
In this case, it is possible to define the holonomy matrix $H_{a,b}$ as a product of matrices 
just as we did in \Cref{def:holonomy_matrices}, but relative to orientation $\tau$.  
The only difference will be that some instances of the matrix $\rho$ will instead be an identity matrix, and vice-versa.
The effect in either case is that the holonomy matrices crossing the edges whose orientation has changed are multiplied by $\rho$.

It is explained in \cite{pz_19} (and again in \cite{moz21, moz22} using our notations and conventions) that reversing the
orientations around all three edges of a triangle corresponds to negating the associated odd variable. Also, it is possible to go from
any orientation to the default one (the boundary edges may differ, but the interior diagonals can be made to agree with the default orientation)
by a sequence of such orientation-reversals around triangles. So we may reduce the general case to examining what happens when we
do this orientation-reversal in a single triangle.

When reversing the orientation of all three edges around a triangle,
all six vertices of the hexagonal face of $\Gamma_T$ in this triangle will be incident to an edge whose holonomy has been multiplied by $\rho$.
If we perform a gauge transformation by $\rho$ at each of these six vertices
\footnote{By a ``\emph{gauge transformation by $\rho$ at a vertex}'', we mean left-multiplying all outgoing edge holonomies by $\rho$
and right-multiplying all incoming edge holonomies by $\rho$.}, 
we can restore those edges to their previous weights (before we changed
their orientations). Since each edge has two endpoints which are gauged, the effect on the three $A$ matrices and the three $E$ matrices will be that they are all conjugated by $\rho$.
Since $E$ commutes with $\rho$, and $\rho^2 = \mathrm{id}$, this leaves the $E$-matrices unaffected. But as was pointed out in
\Cref{rmk:fermionic_reflection}, we have $\rho A(h|\theta) \rho = A(h| -\theta)$. So, in agreement with the remark in the preceding paragraph, 
the effect that this orientation-reversal has on the connection is simply to change $\theta \mapsto -\theta$ in the $A$-matrices.

It is clear that if a path passes through a vertex $v$ of $\Gamma_T$, then a gauge transformation at $v$ will not affect the holonomy
along this path (the contributions of the incoming and outgoing edges will cancel). The conclusion here is that the holonomy formula 
from \Cref{thm:generic} still holds for arbitrary orientation, provided we negate the corresponding odd variables every time we do such an orientation-reversal
around a triangle. 

However, if the vertex $v$ where we perform a gauge transformation is either the beginning or ending point of the path, then the holonomy \emph{will}
change. Specifically, if we reverse the orientations around the first or last triangle (or both), the effect on the holonomy is $H_{ab} \mapsto H_{ab} \rho$,
or $H_{ab} \mapsto \rho H_{ab}$, or $H_{ab} \mapsto \rho H_{ab} \rho$.

\section{Double Dimer Interpretation of Matrix Formulae}
\label{sec:DD}

Motivated by the methods of Sections 4 and 5 of \cite{mw13}, we now provide a combinatorial interpretation of the holonomy matrices, 
which were defined in Section \ref{sec:Mpaths} and described explicitly for generic triangulations with the default orientation in \Cref{thm:generic}.   
In the case considered in \cite{mw13}, the construction involved matrices in $\PSL_2(\mathbb{C})$ whose entries were given interpretations in terms of 
perfect matchings of snake graphs.  In the present work, we instead consider $2|1$-by-$2|1$ matrices in the group $\osp(1|2)$, and 
obtain combinatorial interpretations of the entries in terms of double dimer covers of snake graphs, using results from \cite{moz22}.

Let $T$ be an arbitrary acyclic triangulation of a polygon such that the arc $(a,b)$ is the longest arc in $T$, i.e. it cuts through all internal arcs of $T$.  
Assume further that $T$ is equipped with the default orientation with fan centers labeled as $c_i$ for $1 \leq i \leq N$.  
Like in \Cref{sec:Hproof}, we let $a=c_0$, $b= c_{N+1}$,  and let $H_{ab}$ be the holonomy as defined in \Cref{def:H_ab}.
The main result of this section is to reinterpret the entries of $H_{a,b}$ as combinatorial generating functions as follows.

First, let $\widetilde{T}$ denote the triangulation that extends triangulation $T$ by defining two new marked points, $\widetilde{a}$ and $\widetilde{b}$, 
and adjoining the triangles $(\widetilde{a},c_0,c_1)$ and $(c_N, c_{N+1},\widetilde{b})$, respectively about the edges $(c_0,c_1)$ and $(c_N, c_{N+1})$.  
We will use $\theta_{\widetilde{a}}$ and $\theta_{\widetilde{b}}$ to denote the $\mu$-invariants associated to these two new triangles, repsectively.  See \Cref{fig:extT}.
We let $\widetilde{G} = G_{\widetilde{T}}$ be the snake graph corresponding to the longest arc $(\widetilde{a}, \widetilde{b})$ in $\widetilde{T}$, 
as defined initially in \cite{ms10} and extended to the case of decorated super Teichm\"uller space in \cite[Sec. 3]{moz22}.

\begin{figure}[h]
    \scalebox{0.7}{
        \begin{tikzpicture}
	\begin{pgfonlayer}{nodelayer}
		\node [style=bullet mid] (0) at (-9, 4) {};
		\node [style=bullet mid] (1) at (-8, -3.5) {};
		\node [style=bullet mid] (2) at (-6.25, -8.75) {};
		\node [style=bullet mid] (5) at (0, -8) {};
		\node [style=bullet mid] (6) at (8, -3.5) {};
		\node [style=bullet mid] (7) at (9, 4) {};
		\node [style=bullet mid] (8) at (4.5, 10) {};
		\node [style=bullet mid] (9) at (-4.5, 10) {};
		\node [style=bullet small] (10) at (-5.25, 8.5) {};
		\node [style=bullet small] (11) at (-7.5, 5.5) {};
		\node [style=bullet small] (12) at (-6.75, 5.25) {};
		\node [style=bullet small] (13) at (-2.25, 9.75) {};
		\node [style=bullet small] (14) at (0.75, 9.75) {};
		\node [style=bullet small] (15) at (1, 8.75) {};
		\node [style=bullet small] (16) at (1.25, 8.25) {};
		\node [style=bullet small] (17) at (5.5, 8) {};
		\node [style=bullet small] (18) at (7.75, 5) {};
		\node [style=bullet small] (21) at (-6.5, 4.75) {};
		\node [style=bullet small] (24) at (-8.5, 2.75) {};
		\node [style=bullet small] (25) at (-8, -1.25) {};
		\node [style=bullet small] (26) at (-5.25, -2) {};
		\node [style=bullet small] (27) at (5.5, 2.75) {};
		\node [style=bullet small] (28) at (-5, -2.5) {};
		\node [style=bullet small] (29) at (-4.75, -3.25) {};
		\node [style=bullet small] (30) at (5.75, -3.25) {};
		\node [style=bullet small] (31) at (7.75, -2.25) {};
		\node [style=bullet small] (32) at (8.25, 1.75) {};
		\node [style=bullet small] (33) at (5.75, 2.25) {};
		\node [style=bullet small] (34) at (-4.75, -3.75) {};
		\node [style=bullet small] (35) at (-4.75, -5) {};
		\node [style=bullet small] (36) at (-2, -6.5) {};
		\node [style=bullet small] (37) at (1, -7) {};
		\node [style=bullet small] (38) at (6, -4.25) {};
		\node [style=bullet small] (39) at (5.75, -3.75) {};
		\node [style=bullet small] (40) at (-5, -5.5) {};
		\node [style=bullet small] (41) at (-6.75, -6.25) {};
		\node [style=bullet small] (42) at (-6.25, -7.75) {};
		\node [style=bullet small] (43) at (-4.5, -8.25) {};
		\node [style=bullet small] (44) at (-2.25, -8) {};
		\node [style=bullet small] (45) at (-2.25, -7) {};
		\node [style=none] (58) at (-5.75, 10) {$a=c_0$};
		\node [style=none] (59) at (-9.75, 4) {$v$};
		\node [style=none] (61) at (-8.75, -3.5) {$c_2$};
		\node [style=none] (62) at (8.75, -3.5) {$w$};		
		\node [style=none] (63) at (0.5, -8.5) {$b=c_3$};
		\node [style=none] (64) at (5.25, 10) {$c_1$};
		\node [style=bullet small] (65) at (0, 12.75) {};
		\node [style=bullet small] (66) at (-2.25, 10.25) {};
		\node [style=bullet small] (67) at (0.75, 10.25) {};
		\node [style=bullet small] (68) at (2.5, 11) {};
		\node [style=bullet small] (69) at (0.75, 12) {};
		\node [style=bullet small] (70) at (-0.75, 12) {};
		\node [style=bullet small] (71) at (-2.5, 11) {};
		\node [style=none] (72) at (-6.25, -9.25) {$\tilde b$};
		\node [style=none] (73) at (0, 13.25) {$\tilde a$};
		\node [style=none] (74) at (-2.25, 9.25) {$a'$};
		\node [style=none] (75) at (-1.75, -6) {$b'$};
		\node [style=none] (76) at (-4.75, -7) {$\theta_{\tilde b}$};
		\node [style=none] (77) at (-0.25, 11) {$\theta_{\tilde a}$};
		\node [style=bullet small] (78) at (2, 7.75) {};
		\node [style=bullet small] (79) at (2.5, 7.5) {};
		\node [style=bullet small] (80) at (-5.75, -1.5) {};
		\node [style=bullet small] (81) at (-6.25, -1.25) {};
	\end{pgfonlayer}
	\begin{pgfonlayer}{edgelayer}
		\draw (9) to (8);
		\draw (8) to (7);
		\draw (7) to (6);
		\draw (6) to (5);
		\draw [style=extend-edge] (2) to (1);
		\draw (1) to (0);
		\draw (0) to (9);
		\draw [style=dashed] (10) to (13);
		\draw [style=dashed] (13) to (14);
		\draw [style=dashed] (14) to (15);
		\draw [style=dashed] (15) to (12);
		\draw [style=dashed] (12) to (11);
		\draw [style=dashed] (11) to (10);
		\draw [style=dashed] (12) to (21);
		\draw [style=dashed] (21) to (16);
		\draw [style=dashed] (17) to (18);
		\draw [style=dashed] (15) to (16);
		\draw [style=dashed] (24) to (25);
		\draw [style=dashed] (26) to (27);
		\draw [style=dashed] (26) to (28);
		\draw [style=dashed] (28) to (33);
		\draw [style=dashed] (33) to (27);
		\draw [style=dashed] (33) to (32);
		\draw [style=dashed] (32) to (31);
		\draw [style=dashed] (31) to (30);
		\draw [style=dashed, in=360, out=180] (30) to (29);
		\draw [style=dashed] (29) to (28);
		\draw [style=dashed] (29) to (34);
		\draw [style=dashed] (34) to (39);
		\draw [style=dashed] (39) to (30);
		\draw [style=dashed] (34) to (35);
		\draw [style=dashed] (35) to (36);
		\draw [style=dashed] (36) to (37);
		\draw [style=dashed] (37) to (38);
		\draw [style=dashed] (38) to (39);
		\draw [style=dashed] (35) to (40);
		\draw [style=dashed] (40) to (45);
		\draw [style=dashed] (45) to (36);
		\draw [style=dashed] (45) to (44);
		\draw [style=dashed] (44) to (43);
		\draw [style=dashed] (43) to (42);
		\draw [style=dashed] (42) to (41);
		\draw [style=dashed] (41) to (40);
		\draw (8) to (0);
		\draw (7) to (1);
		\draw (1) to (5);
		\draw [style=extend-edge] (5) to (2);
		\draw (1) to (6);
		\draw [style=extend-edge] (65) to (9);
		\draw [style=extend-edge] (65) to (8);
		\draw [style=dashed] (71) to (70);
		\draw [style=dashed] (70) to (69);
		\draw [style=dashed] (69) to (68);
		\draw [style=dashed] (68) to (67);
		\draw [style=dashed] (67) to (66);
		\draw [style=dashed] (66) to (71);
		\draw [style=dashed] (13) to (66);
		\draw [style=dashed] (67) to (14);
		\draw (1) to (8);
		\draw [style=dashed] (24) to (21);
		\draw [style=dashed] (25) to (81);
		\draw [style=dashed] (81) to (78);
		\draw [style=dashed] (78) to (16);
		\draw [style=dashed] (78) to (79);
		\draw [style=dashed] (79) to (17);
		\draw [style=dashed] (18) to (27);
		\draw [style=dashed] (80) to (79);
		\draw [style=dashed] (81) to (80);
		\draw [style=dashed] (80) to (26);
	\end{pgfonlayer}
\end{tikzpicture}
    }
\caption{Extended triangulation $\tilde T$.  We let $v$ (resp. $w$) denote the fourth corner of the quadrilateral 
defined by the marked points $\widetilde{a}$, $a=c_0$, and $c_1$ (resp. $c_N$, $b=c_{N+1}$, and $\widetilde{b}$).}
\label{fig:extT}
\end{figure}
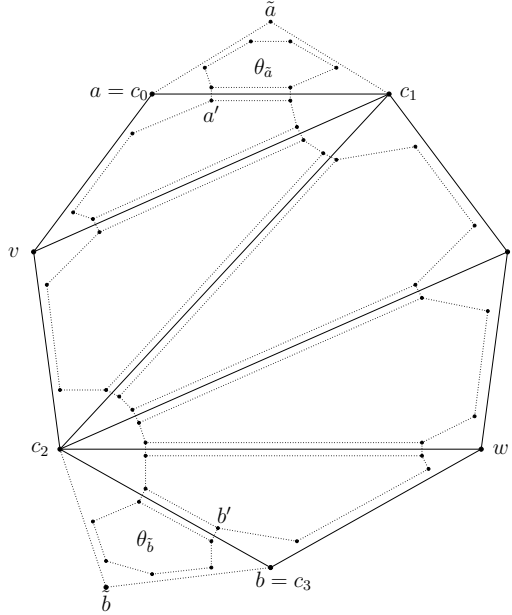

In $\widetilde{T}$, we let $i_1,i_2,\dots, i_d$ denote the internal arcs crossed in order by the longest arc $(\widetilde{a},\widetilde{b})$.  
In particular, $i_1 = (c_0,c_1)$ and $i_d = (c_N, c_{N+1})$.  

Given our earlier definitions of $\epsilon_a$ and $\epsilon_b$, and noting that in the quadrilateral on $\widetilde{a},c_0,c_1,c_2$, 
the triangles $(c_0,c_1,c_2)$ and $(\widetilde{a}, c_0, c_1)$ are of opposite orientations (and we have an analogous statement for the quadrilateral on 
$c_{N-1}, c_N, c_{N+1}, \widetilde{b}$ we get the following equivalent usage of the values $\epsilon_a$ and $\epsilon_b$:
\begin{align*}
    \epsilon_a &= \begin{cases}
                      0&\text{ if }(\widetilde{a},c_0,c_1)\text{ are oriented {\bf counter-}clockwise,}\\
                      1&\text{ otherwise.}
                  \end{cases}\\
    \epsilon_b &= \begin{cases}
                      0&\text{ if }(c_N,c_{N+1},\widetilde{b})\text{ are oriented {\bf counter-}clockwise,}\\
                      1&\text{ otherwise.}
                  \end{cases}
\end{align*}

When building the snake graph $G_{\widetilde{T}}$, we note that as we progress from the bottom-left to the top right, 
the second tile is to the east (resp. north) of the first tile if $\epsilon_a = 0$ (resp. $1$).  
We use the following notation as shorthand for the weights that appear on the bottom and left edges of the first tile as well as the top and right edges of the last tile, 
in some order: $e_0 = \lambda_{\widetilde{a},c_0}$, $e_1 = \lambda_{\widetilde{a},c_1}$, $e_N = \lambda_{c_N,\widetilde{b}}$, and $e_{N+1} = \lambda_{c_{N+1},\widetilde{b}}$.
For example, $e_0$ is the weight of the  bottom edge (resp. left edge) of the first tile if $\epsilon_a = 0$ (resp. $\epsilon_a = 1$), 
and $e_1$ is the weight of left edge (or bottom edge) respectively.  See \Cref{fig:different_snake}. 
We will sometimes abuse notation and let $e_0$, $e_1$, $e_N$, and $e_{N+1}$ denote the corresponding arcs themselves.

\begin{figure}[h]
\begin{tabular}{c|c|c}
                     & $|T| - N = 0 \mod 2$               & $|T|-N = 1 \mod 2$                 \\ \hline
                     &                                    &                                    \\[-1em]
    $\epsilon_a = 0$ & \scalebox{0.8}{\begin{tikzpicture}[scale=0.5]
	\begin{pgfonlayer}{nodelayer}
		\node [style=none] (0) at (-5, 0) {};
		\node [style=none] (1) at (-5, -4) {};
		\node [style=none] (2) at (-1, -4) {};
		\node [style=none] (3) at (-1, 0) {};
		\node [style=none] (4) at (3, 0) {};
		\node [style=none] (5) at (3, -4) {};
		\node [style=none] (6) at (8, 3) {};
		\node [style=none] (7) at (12, 3) {};
		\node [style=none] (8) at (16, 3) {};
		\node [style=none] (9) at (16, 7) {};
		\node [style=none] (10) at (12, 7) {};
		\node [style=none] (11) at (8, 7) {};
		\node [style=none] (12) at (5, 1.25) {$\cdot$};
		\node [style=none] (13) at (5.5, 1.5) {$\cdot$};
		\node [style=none] (14) at (6, 1.75) {$\cdot$};
		\node [style=0.8] (15) at (-5.5, -2) {$e_1$};
		\node [style=0.8] (16) at (-3, 0.75) {$\lambda_{c_1,v}$};
		\node [style=0.8] (17) at (-3, -4.5) {$e_0$};
		\node [style=0.8] (18) at (-3, -2) {$x_{i_1}$};
		\node [scale=0.7] (19) at (-1, -1.75) {$\lambda_{c_0,v}$};
		\node [style=0.8] (20) at (1, -2) {$x_{i_2}$};
		\node [style=0.8] (21) at (9.75, 5) {$x_{i_d-1}$};
		\node [style=0.8] (22) at (14, 5) {$x_{i_d}$};
		\node [style=0.8] (23) at (14, 7.5) {$e_{N+1}$};
		\node [style=0.8] (24) at (17, 5) {$e_N$};
		\node [style=none] (25) at (-1, -1.5) {};
		\node [style=none] (26) at (-1, -2.5) {};
		\node [style=none] (27) at (-3.5, -1.5) {};
		\node [style=none] (28) at (-2.5, -2.5) {};
		\node [style=none] (29) at (0.5, -1.5) {};
		\node [style=none] (30) at (1.5, -2.5) {};
		\node [style=none] (31) at (9.5, 5.5) {};
		\node [style=none] (32) at (10.5, 4.5) {};
		\node [style=none] (33) at (13.5, 5.5) {};
		\node [style=none] (34) at (14.5, 4.5) {};
	\end{pgfonlayer}
	\begin{pgfonlayer}{edgelayer}
		\draw (11.center) to (6.center);
		\draw (6.center) to (7.center);
		\draw (7.center) to (8.center);
		\draw (8.center) to (9.center);
		\draw (9.center) to (10.center);
		\draw (11.center) to (10.center);
		\draw (10.center) to (7.center);
		\draw (0.center) to (3.center);
		\draw (0.center) to (1.center);
		\draw (1.center) to (2.center);
		\draw (2.center) to (5.center);
		\draw (5.center) to (4.center);
		\draw (4.center) to (3.center);
		\draw (3.center) to (25.center);
		\draw (26.center) to (2.center);
		\draw [style=extend-edge] (0.center) to (27.center);
		\draw [style=extend-edge] (28.center) to (2.center);
		\draw [style=extend-edge] (3.center) to (29.center);
		\draw [style=extend-edge] (30.center) to (5.center);
		\draw [style=extend-edge] (11.center) to (31.center);
		\draw [style=extend-edge] (32.center) to (7.center);
		\draw [style=extend-edge] (10.center) to (33.center);
		\draw [style=extend-edge] (34.center) to (8.center);
	\end{pgfonlayer}
\end{tikzpicture}} & \scalebox{0.8}{\begin{tikzpicture}[scale=0.5]
	\begin{pgfonlayer}{nodelayer}
		\node [style=none] (0) at (-5, 0) {};
		\node [style=none] (1) at (-5, -4) {};
		\node [style=none] (2) at (-1, -4) {};
		\node [style=none] (3) at (-1, 0) {};
		\node [style=none] (4) at (3, 0) {};
		\node [style=none] (5) at (3, -4) {};
		\node [style=none] (6) at (8, 3) {};
		\node [style=none] (7) at (8, 7) {};
		\node [style=none] (8) at (8, 11) {};
		\node [style=none] (9) at (12, 11) {};
		\node [style=none] (10) at (12, 7) {};
		\node [style=none] (11) at (12, 3) {};
		\node [style=none] (12) at (5, 1.25) {$\cdot$};
		\node [style=none] (13) at (5.5, 1.5) {$\cdot$};
		\node [style=none] (14) at (6, 1.75) {$\cdot$};
		\node [style=0.8] (15) at (-5.5, -2) {$e_1$};
		\node [style=0.8] (16) at (-3, 0.75) {$\lambda_{c_1,v}$};
		\node [style=0.8] (17) at (-3, -4.5) {$e_0$};
		\node [style=0.8] (18) at (-3, -2) {$x_{i_1}$};
		\node [scale=0.7] (19) at (-1, -1.75) {$\lambda_{c_0,v}$};
		\node [style=0.8] (20) at (1, -2) {$x_{i_2}$};
		\node [style=0.8] (21) at (10, 5) {$x_{i_d-1}$};
		\node [style=0.8] (22) at (10, 9) {$x_{i_d}$};
		\node [style=0.8] (23) at (13.5, 9) {$e_{N+1}$};
		\node [style=0.8] (24) at (10, 11.5) {$e_N$};
		\node [style=none] (25) at (-1, -1.5) {};
		\node [style=none] (26) at (-1, -2.5) {};
		\node [style=none] (27) at (-3.5, -1.5) {};
		\node [style=none] (28) at (-2.5, -2.5) {};
		\node [style=none] (29) at (0.5, -1.5) {};
		\node [style=none] (30) at (1.5, -2.5) {};
		\node [style=none] (31) at (9.5, 5.5) {};
		\node [style=none] (32) at (10.5, 4.5) {};
		\node [style=none] (33) at (9.5, 9.5) {};
		\node [style=none] (34) at (10.5, 8.5) {};
	\end{pgfonlayer}
	\begin{pgfonlayer}{edgelayer}
		\draw (11.center) to (6.center);
		\draw (6.center) to (7.center);
		\draw (7.center) to (8.center);
		\draw (8.center) to (9.center);
		\draw (9.center) to (10.center);
		\draw (11.center) to (10.center);
		\draw (10.center) to (7.center);
		\draw (0.center) to (3.center);
		\draw (0.center) to (1.center);
		\draw (1.center) to (2.center);
		\draw (2.center) to (5.center);
		\draw (5.center) to (4.center);
		\draw (4.center) to (3.center);
		\draw (3.center) to (25.center);
		\draw (26.center) to (2.center);
		\draw [style=extend-edge] (8.center) to (33.center);
		\draw [style=extend-edge] (34.center) to (10.center);
		\draw [style=extend-edge] (7.center) to (31.center);
		\draw [style=extend-edge] (32.center) to (11.center);
		\draw [style=extend-edge] (3.center) to (29.center);
		\draw [style=extend-edge] (30.center) to (5.center);
		\draw [style=extend-edge] (28.center) to (2.center);
		\draw [style=extend-edge] (0.center) to (27.center);
	\end{pgfonlayer}
\end{tikzpicture}} \\
                     &                                    &                                    \\[-1em] \hline
                     &                                    &                                    \\[-1em]
    $\epsilon_a = 1$ & \scalebox{0.8}{\begin{tikzpicture}[scale=0.5]
	\begin{pgfonlayer}{nodelayer}
		\node [style=none] (0) at (3, -8) {};
		\node [style=none] (1) at (-1, -8) {};
		\node [style=none] (2) at (-1, -4) {};
		\node [style=none] (3) at (3, -4) {};
		\node [style=none] (4) at (3, 0) {};
		\node [style=none] (5) at (-1, 0) {};
		\node [style=none] (6) at (8, 3) {};
		\node [style=none] (7) at (8, 7) {};
		\node [style=none] (8) at (8, 11) {};
		\node [style=none] (9) at (12, 11) {};
		\node [style=none] (10) at (12, 7) {};
		\node [style=none] (11) at (12, 3) {};
		\node [style=none] (12) at (5, 1.25) {$\cdot$};
		\node [style=none] (13) at (5.5, 1.5) {$\cdot$};
		\node [style=none] (14) at (6, 1.75) {$\cdot$};
		\node [style=0.8] (15) at (1, -8.5) {$e_1$};
		\node [style=0.8] (16) at (4.25, -6) {$\lambda_{c_1,v}$};
		\node [style=0.8] (17) at (-1.5, -6) {$e_0$};
		\node [style=0.8] (18) at (1, -6) {$x_{i_1}$};
		\node [style=0.8] (20) at (1, -2) {$x_{i_2}$};
		\node [style=0.8] (21) at (9.75, 5) {$x_{i_d-1}$};
		\node [style=0.8] (22) at (10.25, 9) {$x_{i_d}$};
		\node [scale=0.75] (23) at (13.25, 9) {$e_{N+1}$};
		\node [style=0.8] (24) at (10, 11.5) {$e_N$};
		\node [scale=0.7] (25) at (0.5, -3.25) {$\lambda_{c_0,v}$};
		\node [style=none] (26) at (9.5, 5.5) {};
		\node [style=none] (27) at (10.25, 4.75) {};
		\node [style=none] (28) at (9.5, 9.5) {};
		\node [style=none] (29) at (10.5, 8.5) {};
		\node [style=none] (30) at (0.5, -1.5) {};
		\node [style=none] (31) at (1.5, -2.5) {};
		\node [style=none] (32) at (0.5, -5.5) {};
		\node [style=none] (33) at (1.5, -6.5) {};
	\end{pgfonlayer}
	\begin{pgfonlayer}{edgelayer}
		\draw (11.center) to (6.center);
		\draw (6.center) to (7.center);
		\draw (7.center) to (8.center);
		\draw (8.center) to (9.center);
		\draw (9.center) to (10.center);
		\draw (11.center) to (10.center);
		\draw (10.center) to (7.center);
		\draw (0.center) to (3.center);
		\draw (0.center) to (1.center);
		\draw (1.center) to (2.center);
		\draw (2.center) to (5.center);
		\draw (5.center) to (4.center);
		\draw (4.center) to (3.center);
		\draw (2.center) to (3.center);
		\draw [style=extend-edge] (7.center) to (26.center);
		\draw [style=extend-edge] (27.center) to (11.center);
		\draw [style=extend-edge] (8.center) to (28.center);
		\draw [style=extend-edge] (29.center) to (10.center);
		\draw [style=extend-edge] (2.center) to (32.center);
		\draw [style=extend-edge] (33.center) to (0.center);
		\draw [style=extend-edge] (5.center) to (30.center);
		\draw [style=extend-edge] (31.center) to (3.center);
	\end{pgfonlayer}
\end{tikzpicture}} & \scalebox{0.8}{\begin{tikzpicture}[scale=0.5]
	\begin{pgfonlayer}{nodelayer}
		\node [style=none] (0) at (3, -8) {};
		\node [style=none] (1) at (-1, -8) {};
		\node [style=none] (2) at (-1, -4) {};
		\node [style=none] (3) at (3, -4) {};
		\node [style=none] (4) at (3, 0) {};
		\node [style=none] (5) at (-1, 0) {};
		\node [style=none] (6) at (8, 3) {};
		\node [style=none] (7) at (12, 3) {};
		\node [style=none] (8) at (16, 3) {};
		\node [style=none] (9) at (16, 7) {};
		\node [style=none] (10) at (12, 7) {};
		\node [style=none] (11) at (8, 7) {};
		\node [style=none] (12) at (5, 1.25) {$\cdot$};
		\node [style=none] (13) at (5.5, 1.5) {$\cdot$};
		\node [style=none] (14) at (6, 1.75) {$\cdot$};
		\node [style=0.8] (15) at (1, -8.5) {$e_1$};
		\node [style=0.8] (16) at (4.25, -6.25) {$\lambda_{c_1,v}$};
		\node [style=0.8] (17) at (-1.75, -6) {$e_0$};
		\node [style=0.8] (18) at (1.25, -6) {$x_{i_1}$};
		\node [style=0.8] (20) at (1.25, -2) {$x_{i_2}$};
		\node [style=0.8] (21) at (10, 5) {$x_{i_d-1}$};
		\node [style=0.8] (22) at (14, 5) {$x_{i_d}$};
		\node [style=0.8] (23) at (14, 7.5) {$e_{N+1}$};
		\node [style=0.8] (24) at (17, 5) {$e_N$};
		\node [scale=0.7] (25) at (0.5, -3.25) {$\lambda_{c_0,v}$};
		\node [style=none] (26) at (13.5, 5.5) {};
		\node [style=none] (27) at (14.5, 4.5) {};
		\node [style=none] (28) at (9.5, 5.5) {};
		\node [style=none] (29) at (10.5, 4.5) {};
		\node [style=none] (30) at (0.5, -1.5) {};
		\node [style=none] (31) at (1.75, -2.75) {};
		\node [style=none] (32) at (0.5, -5.5) {};
		\node [style=none] (33) at (1.5, -6.5) {};
	\end{pgfonlayer}
	\begin{pgfonlayer}{edgelayer}
		\draw (11.center) to (6.center);
		\draw (6.center) to (7.center);
		\draw (7.center) to (8.center);
		\draw (8.center) to (9.center);
		\draw (9.center) to (10.center);
		\draw (11.center) to (10.center);
		\draw (10.center) to (7.center);
		\draw (0.center) to (3.center);
		\draw (0.center) to (1.center);
		\draw (1.center) to (2.center);
		\draw (2.center) to (5.center);
		\draw (5.center) to (4.center);
		\draw (4.center) to (3.center);
		\draw (2.center) to (3.center);
		\draw [style=extend-edge] (2.center) to (32.center);
		\draw [style=extend-edge] (33.center) to (0.center);
		\draw [style=extend-edge] (5.center) to (30.center);
		\draw [style=extend-edge] (31.center) to (3.center);
		\draw [style=extend-edge] (11.center) to (28.center);
		\draw [style=extend-edge] (29.center) to (7.center);
		\draw [style=extend-edge] (10.center) to (26.center);
		\draw [style=extend-edge] (27.center) to (8.center);
	\end{pgfonlayer}
\end{tikzpicture}}
\end{tabular}	
\caption{Different Snake Graphs as $\epsilon_a$ and $(|T| - N)$ vary, where $|T|$ is the number of triangles in $T$.}
\label{fig:different_snake}
\end{figure}
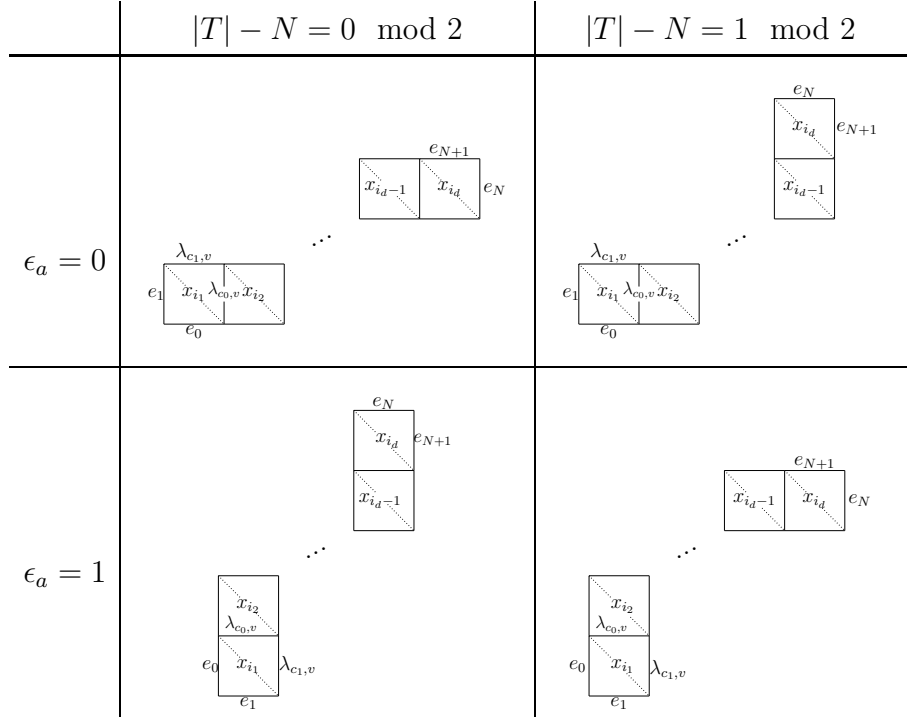

\begin {definition}
    For a snake graph $G$, let $D(G)$ denote the set of double dimer covers of $G$. 
    Also, let $D_{ab}(\widetilde{G})$, $D^{cd}(\widetilde{G})$, and $D_{ab}^{cd}(\widetilde{G})$
    denote the subsets of double dimer covers which includes (as sub-multisets) $\{e_a,e_b\}$, $\{e_c,e_d\}$, or $\{e_a,e_b,e_c,e_d\}$ respectively. 
    Here, $a$ and $b$ will always be $0$ or $1$ (the bottom/left edges of the first tile), and $c,d$ will always $N$ or $N+1$ (the top/right edges
    of the last tile).
\end {definition}

\begin{theorem} \label{thm:combo}
    The entries of $H_{a,b}$ each have combinatorial interpretations as weighted generating functions of double dimer covers of $G_{\widetilde{T}}$, 
    where each is subject to a restriction on the bottom-left and top-right tiles of $G_{\widetilde{T}}$.
    More precisely $H_{a,b}$ is given by the following matrix:
    \[ 
        \frac{1}{x_{i_2} \cdots x_{i_{d-1}}}
        \ospmatrix {\frac{1}{e_N}}{0}{0} {0}{\frac{(-1)^{\epsilon_b-1}}{x_{i_d}e_{N+1}}}{0} {0}{0}{\frac{1}{\sqrt{x_{i_d}e_Ne_{N+1}}}}
        \ospmatrix {-\widetilde{A}}     {-\widetilde{B}}     {\widetilde{\gamma}}
                   {-\widetilde{C}}     {-\widetilde{D}}     {\widetilde{\delta}}
                   {\widetilde{\alpha}} {\widetilde{\beta}} {\widetilde{E}}
        \ospmatrix {\frac{1}{e_0x_{i_1}}}{0}{0} {0}{\frac{(-1)^{\epsilon_a-1}}{e_1}}{0} {0}{0}{\frac{1}{\sqrt{e_0e_1x_{i_1}}}}
    \]
    such that
    \[
    \begin{array}{ccc}\displaystyle
    	\widetilde{A} = 
        \sum_{M \in D_{00}^{NN}(\widetilde{G})} \wt(M),\;\; & \displaystyle
        \widetilde{B} = 
        \sum_{M \in D_{11}^{NN}(\widetilde{G})} \wt(M),\;\; & \displaystyle
        \widetilde{\gamma} = 
        \sum_{M \in D_{01}^{NN}(\widetilde{G})} \wt(M)^*
    \end {array}
    \]
    \[
    \begin{array}{ccc}
        \displaystyle\widetilde{C} = 
        \sum_{M \in D_{00}^{N+1,N+1}(\widetilde{G})} \wt(M),\;\; &
        \displaystyle\widetilde{D} = 
        \sum_{M \in D_{11}^{N+1,N+1}(\widetilde{G})} \wt(M),\;\; &
        \displaystyle\widetilde{\delta} = 
        \sum_{M \in D_{01}^{N+1,N+1}(\widetilde{G})} \wt(M)^*
    \end {array}
    \]
    \[
    \begin{array}{ccc}
       \displaystyle  \widetilde{\alpha} = 
        \sum_{M \in D_{00}^{N,N+1}(\widetilde{G})} \pm \wt(M)^\dagger,
        & \displaystyle        
        \widetilde{\beta} = 
        \sum_{M \in D_{11}^{N,N+1}(\widetilde{G})} \pm \wt(M)^\dagger, &
        \displaystyle \widetilde{E} = \sum_{M \in D_{01}^{N,N+1}(\widetilde{G})} \pm  (\wt(M)^*)^\dagger 
    \end{array}
    \]

    where $\wt(M)$ denotes the weight of the double dimer cover (see \cite[Def. 4.4]{moz22}),
    and $(*)$ denotes the toggle operation on $\theta_{\widetilde{a}}$ while $(\dagger)$ indicates the toggle operation on $\theta_{\widetilde{b}}$.  
    In our cases, the toggle operation $(*)$ (resp.  $(\dagger)$) removes 
    $\theta_{\widetilde{a}}$ (resp. $\theta_{\widetilde{b}}$) from the corresponding term.  
    See \cite[Def. 5.6]{moz22} for the more general definition.
    
    The signs on the terms in $\widetilde{\alpha}$ and $\widetilde{\beta}$ are determined as in \Cref{rmk:local_signs}. That is, $\mathrm{wt}(M)$ is written
    in the positive order, followed by a substitution $\theta \mapsto -\theta$ for an appropriate subset of the odd variables (i.e. all those in the
    even-numbered fan segments). The $\widetilde{E}$-entry also potentially
    contains terms of both signs, but it is more complicated to specify.
\end{theorem}

Note that even though the expression for $H_{ab}$ in \Cref{thm:combo} involves the quantities $e_0$, $e_1$, $e_N$, and $e_{N+1}$, 
after reducing each of the nine matrix entries to lowest terms, such factors will always cancel.  
This is consistent with the fact that $H_{ab}$ is defined by the arc $(a,b)$ that is contained in the original triangulation $T$, 
where the triangles containing $\widetilde{a}$ and $\widetilde{b}$ do not appear.   

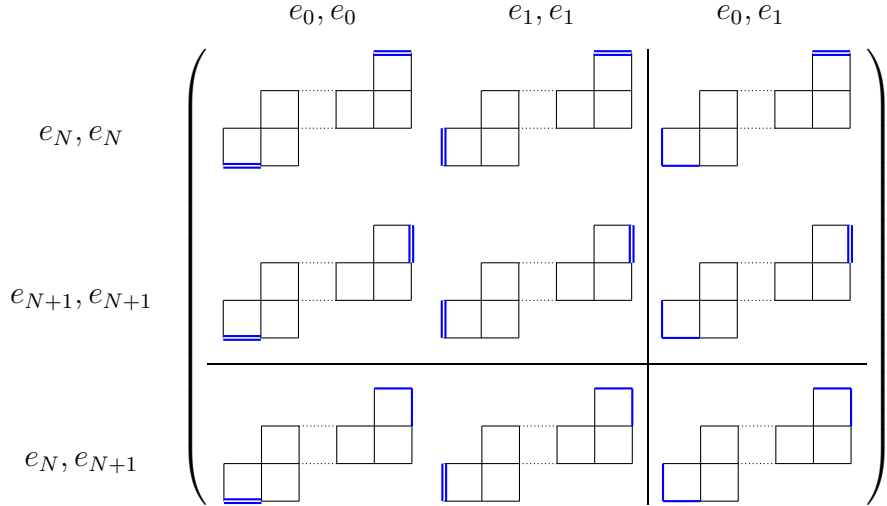
\begin{figure}[h]
    \begin{tabular}{cc}
    &$e_0,e_0$ \hspace{4.2em} $e_1,e_1$\hspace{4.2em} $e_0,e_1$\\
    &\\[-1.3em]
    & \multirow{3}{*}{$\left(
            \begin{array}{cc|c}   
                    \begin{tikzpicture}
	\begin{pgfonlayer}{nodelayer}
		\node [style=none] (0) at (-3, 0) {};
		\node [style=none] (1) at (-3, -1) {};
		\node [style=none] (2) at (-2, -1) {};
		\node [style=none] (3) at (-2, 0) {};
		\node [style=none] (4) at (-1, -1) {};
		\node [style=none] (5) at (-1, 0) {};
		\node [style=none] (6) at (0, 0) {};
		\node [style=none] (8) at (-1, 1) {};
		\node [style=none] (9) at (0, 1) {};
		\node [style=none] (10) at (1, 1) {};
		\node [style=none] (11) at (1, 0) {};
		\node [style=none] (12) at (2, 1) {};
		\node [style=none] (13) at (1, 2) {};
		\node [style=none] (14) at (2, 2) {};
		\node [style=none] (15) at (2, 0) {};
		\node [style=none] (16) at (-2, 1) {};
	\end{pgfonlayer}
	\begin{pgfonlayer}{edgelayer}
		\draw (0.center) to (3.center);
		\draw (3.center) to (5.center);
		\draw (0.center) to (1.center);
		\draw [style=blue-double] (1.center) to (2.center);
		\draw (2.center) to (3.center);
		\draw (2.center) to (4.center);
		\draw (4.center) to (5.center);
		\draw [style=extend-edge] (6.center) to (5.center);
		\draw (5.center) to (8.center);
		\draw [style=extend-edge] (8.center) to (9.center);
		\draw (9.center) to (6.center);
		\draw (11.center) to (15.center);
		\draw (15.center) to (12.center);
		\draw (12.center) to (10.center);
		\draw (11.center) to (10.center);
		\draw (10.center) to (13.center);
		\draw [style=blue-double] (13.center) to (14.center);
		\draw (14.center) to (12.center);
		\draw (16.center) to (8.center);
		\draw (16.center) to (3.center);
		\draw (9.center) to (10.center);
		\draw (11.center) to (6.center);
	\end{pgfonlayer}
\end{tikzpicture}&\begin{tikzpicture}
	\begin{pgfonlayer}{nodelayer}
		\node [style=none] (0) at (-3, 0) {};
		\node [style=none] (1) at (-3, -1) {};
		\node [style=none] (2) at (-2, -1) {};
		\node [style=none] (3) at (-2, 0) {};
		\node [style=none] (4) at (-1, -1) {};
		\node [style=none] (5) at (-1, 0) {};
		\node [style=none] (6) at (0, 0) {};
		\node [style=none] (8) at (-1, 1) {};
		\node [style=none] (9) at (0, 1) {};
		\node [style=none] (10) at (1, 1) {};
		\node [style=none] (11) at (1, 0) {};
		\node [style=none] (12) at (2, 1) {};
		\node [style=none] (13) at (1, 2) {};
		\node [style=none] (14) at (2, 2) {};
		\node [style=none] (15) at (2, 0) {};
		\node [style=none] (16) at (-2, 1) {};
	\end{pgfonlayer}
	\begin{pgfonlayer}{edgelayer}
		\draw (0.center) to (3.center);
		\draw (3.center) to (5.center);
		\draw [style=blue-double] (0.center) to (1.center);
		\draw (1.center) to (2.center);
		\draw (2.center) to (3.center);
		\draw (2.center) to (4.center);
		\draw (4.center) to (5.center);
		\draw [style=extend-edge] (6.center) to (5.center);
		\draw (5.center) to (8.center);
		\draw [style=extend-edge] (8.center) to (9.center);
		\draw (9.center) to (6.center);
		\draw (11.center) to (15.center);
		\draw (15.center) to (12.center);
		\draw (12.center) to (10.center);
		\draw (11.center) to (10.center);
		\draw (10.center) to (13.center);
		\draw [style=blue-double] (13.center) to (14.center);
		\draw (14.center) to (12.center);
		\draw (16.center) to (8.center);
		\draw (16.center) to (3.center);
		\draw (9.center) to (10.center);
		\draw (11.center) to (6.center);
	\end{pgfonlayer}
\end{tikzpicture}&\begin{tikzpicture}
	\begin{pgfonlayer}{nodelayer}
		\node [style=none] (0) at (-3, 0) {};
		\node [style=none] (1) at (-3, -1) {};
		\node [style=none] (2) at (-2, -1) {};
		\node [style=none] (3) at (-2, 0) {};
		\node [style=none] (4) at (-1, -1) {};
		\node [style=none] (5) at (-1, 0) {};
		\node [style=none] (6) at (0, 0) {};
		\node [style=none] (8) at (-1, 1) {};
		\node [style=none] (9) at (0, 1) {};
		\node [style=none] (10) at (1, 1) {};
		\node [style=none] (11) at (1, 0) {};
		\node [style=none] (12) at (2, 1) {};
		\node [style=none] (13) at (1, 2) {};
		\node [style=none] (14) at (2, 2) {};
		\node [style=none] (15) at (2, 0) {};
		\node [style=none] (16) at (-2, 1) {};
	\end{pgfonlayer}
	\begin{pgfonlayer}{edgelayer}
		\draw (0.center) to (3.center);
		\draw (3.center) to (5.center);
		\draw [style=Blue-Thick] (0.center) to (1.center);
		\draw [style=Blue-Thick] (1.center) to (2.center);
		\draw (2.center) to (3.center);
		\draw (2.center) to (4.center);
		\draw (4.center) to (5.center);
		\draw [style=extend-edge] (6.center) to (5.center);
		\draw (5.center) to (8.center);
		\draw [style=extend-edge] (8.center) to (9.center);
		\draw (9.center) to (6.center);
		\draw (11.center) to (15.center);
		\draw (15.center) to (12.center);
		\draw (12.center) to (10.center);
		\draw (11.center) to (10.center);
		\draw (10.center) to (13.center);
		\draw [style=blue-double] (13.center) to (14.center);
		\draw (14.center) to (12.center);
		\draw (16.center) to (8.center);
		\draw (16.center) to (3.center);
		\draw (9.center) to (10.center);
		\draw (11.center) to (6.center);
	\end{pgfonlayer}
\end{tikzpicture}\\
                    &&\\
                     \begin{tikzpicture}
	\begin{pgfonlayer}{nodelayer}
		\node [style=none] (0) at (-3, 0) {};
		\node [style=none] (1) at (-3, -1) {};
		\node [style=none] (2) at (-2, -1) {};
		\node [style=none] (3) at (-2, 0) {};
		\node [style=none] (4) at (-1, -1) {};
		\node [style=none] (5) at (-1, 0) {};
		\node [style=none] (6) at (0, 0) {};
		\node [style=none] (8) at (-1, 1) {};
		\node [style=none] (9) at (0, 1) {};
		\node [style=none] (10) at (1, 1) {};
		\node [style=none] (11) at (1, 0) {};
		\node [style=none] (12) at (2, 1) {};
		\node [style=none] (13) at (1, 2) {};
		\node [style=none] (14) at (2, 2) {};
		\node [style=none] (15) at (2, 0) {};
		\node [style=none] (16) at (-2, 1) {};
	\end{pgfonlayer}
	\begin{pgfonlayer}{edgelayer}
		\draw (0.center) to (3.center);
		\draw (3.center) to (5.center);
		\draw (0.center) to (1.center);
		\draw [style=blue-double] (1.center) to (2.center);
		\draw (2.center) to (3.center);
		\draw (2.center) to (4.center);
		\draw (4.center) to (5.center);
		\draw [style=extend-edge] (6.center) to (5.center);
		\draw (5.center) to (8.center);
		\draw [style=extend-edge] (8.center) to (9.center);
		\draw (9.center) to (6.center);
		\draw (11.center) to (15.center);
		\draw (15.center) to (12.center);
		\draw (12.center) to (10.center);
		\draw (11.center) to (10.center);
		\draw (10.center) to (13.center);
		\draw (13.center) to (14.center);
		\draw [style=blue-double] (14.center) to (12.center);
		\draw (16.center) to (8.center);
		\draw (16.center) to (3.center);
		\draw (9.center) to (10.center);
		\draw (11.center) to (6.center);
	\end{pgfonlayer}
\end{tikzpicture}&\begin{tikzpicture}
	\begin{pgfonlayer}{nodelayer}
		\node [style=none] (0) at (-3, 0) {};
		\node [style=none] (1) at (-3, -1) {};
		\node [style=none] (2) at (-2, -1) {};
		\node [style=none] (3) at (-2, 0) {};
		\node [style=none] (4) at (-1, -1) {};
		\node [style=none] (5) at (-1, 0) {};
		\node [style=none] (6) at (0, 0) {};
		\node [style=none] (8) at (-1, 1) {};
		\node [style=none] (9) at (0, 1) {};
		\node [style=none] (10) at (1, 1) {};
		\node [style=none] (11) at (1, 0) {};
		\node [style=none] (12) at (2, 1) {};
		\node [style=none] (13) at (1, 2) {};
		\node [style=none] (14) at (2, 2) {};
		\node [style=none] (15) at (2, 0) {};
		\node [style=none] (16) at (-2, 1) {};
	\end{pgfonlayer}
	\begin{pgfonlayer}{edgelayer}
		\draw (0.center) to (3.center);
		\draw (3.center) to (5.center);
		\draw [style=blue-double] (0.center) to (1.center);
		\draw (1.center) to (2.center);
		\draw (2.center) to (3.center);
		\draw (2.center) to (4.center);
		\draw (4.center) to (5.center);
		\draw [style=extend-edge] (6.center) to (5.center);
		\draw (5.center) to (8.center);
		\draw [style=extend-edge] (8.center) to (9.center);
		\draw (9.center) to (6.center);
		\draw (11.center) to (15.center);
		\draw (15.center) to (12.center);
		\draw (12.center) to (10.center);
		\draw (11.center) to (10.center);
		\draw (10.center) to (13.center);
		\draw (13.center) to (14.center);
		\draw [style=blue-double] (14.center) to (12.center);
		\draw (16.center) to (8.center);
		\draw (16.center) to (3.center);
		\draw (9.center) to (10.center);
		\draw (11.center) to (6.center);
	\end{pgfonlayer}
\end{tikzpicture}&\begin{tikzpicture}
	\begin{pgfonlayer}{nodelayer}
		\node [style=none] (0) at (-3, 0) {};
		\node [style=none] (1) at (-3, -1) {};
		\node [style=none] (2) at (-2, -1) {};
		\node [style=none] (3) at (-2, 0) {};
		\node [style=none] (4) at (-1, -1) {};
		\node [style=none] (5) at (-1, 0) {};
		\node [style=none] (6) at (0, 0) {};
		\node [style=none] (8) at (-1, 1) {};
		\node [style=none] (9) at (0, 1) {};
		\node [style=none] (10) at (1, 1) {};
		\node [style=none] (11) at (1, 0) {};
		\node [style=none] (12) at (2, 1) {};
		\node [style=none] (13) at (1, 2) {};
		\node [style=none] (14) at (2, 2) {};
		\node [style=none] (15) at (2, 0) {};
		\node [style=none] (16) at (-2, 1) {};
	\end{pgfonlayer}
	\begin{pgfonlayer}{edgelayer}
		\draw (0.center) to (3.center);
		\draw (3.center) to (5.center);
		\draw [style=Blue-Thick] (0.center) to (1.center);
		\draw [style=Blue-Thick] (1.center) to (2.center);
		\draw (2.center) to (3.center);
		\draw (2.center) to (4.center);
		\draw (4.center) to (5.center);
		\draw [style=extend-edge] (6.center) to (5.center);
		\draw (5.center) to (8.center);
		\draw [style=extend-edge] (8.center) to (9.center);
		\draw (9.center) to (6.center);
		\draw (11.center) to (15.center);
		\draw (15.center) to (12.center);
		\draw (12.center) to (10.center);
		\draw (11.center) to (10.center);
		\draw (10.center) to (13.center);
		\draw (13.center) to (14.center);
		\draw [style=blue-double] (14.center) to (12.center);
		\draw (16.center) to (8.center);
		\draw (16.center) to (3.center);
		\draw (9.center) to (10.center);
		\draw (11.center) to (6.center);
	\end{pgfonlayer}
\end{tikzpicture}\\
                    &&\\[-1em]
                    \hline
                    &&\\[-1em]
                    \begin{tikzpicture}
	\begin{pgfonlayer}{nodelayer}
		\node [style=none] (0) at (-3, 0) {};
		\node [style=none] (1) at (-3, -1) {};
		\node [style=none] (2) at (-2, -1) {};
		\node [style=none] (3) at (-2, 0) {};
		\node [style=none] (4) at (-1, -1) {};
		\node [style=none] (5) at (-1, 0) {};
		\node [style=none] (6) at (0, 0) {};
		\node [style=none] (8) at (-1, 1) {};
		\node [style=none] (9) at (0, 1) {};
		\node [style=none] (10) at (1, 1) {};
		\node [style=none] (11) at (1, 0) {};
		\node [style=none] (12) at (2, 1) {};
		\node [style=none] (13) at (1, 2) {};
		\node [style=none] (14) at (2, 2) {};
		\node [style=none] (15) at (2, 0) {};
		\node [style=none] (16) at (-2, 1) {};
	\end{pgfonlayer}
	\begin{pgfonlayer}{edgelayer}
		\draw (0.center) to (3.center);
		\draw (3.center) to (5.center);
		\draw (0.center) to (1.center);
		\draw [style=blue-double] (1.center) to (2.center);
		\draw (2.center) to (3.center);
		\draw (2.center) to (4.center);
		\draw (4.center) to (5.center);
		\draw [style=extend-edge] (6.center) to (5.center);
		\draw (5.center) to (8.center);
		\draw [style=extend-edge] (8.center) to (9.center);
		\draw (9.center) to (6.center);
		\draw (11.center) to (15.center);
		\draw (15.center) to (12.center);
		\draw (12.center) to (10.center);
		\draw (11.center) to (10.center);
		\draw (10.center) to (13.center);
		\draw [style=Blue-Thick] (13.center) to (14.center);
		\draw [style=Blue-Thick] (14.center) to (12.center);
		\draw (16.center) to (8.center);
		\draw (16.center) to (3.center);
		\draw (9.center) to (10.center);
		\draw (11.center) to (6.center);
	\end{pgfonlayer}
\end{tikzpicture}&\begin{tikzpicture}
	\begin{pgfonlayer}{nodelayer}
		\node [style=none] (0) at (-3, 0) {};
		\node [style=none] (1) at (-3, -1) {};
		\node [style=none] (2) at (-2, -1) {};
		\node [style=none] (3) at (-2, 0) {};
		\node [style=none] (4) at (-1, -1) {};
		\node [style=none] (5) at (-1, 0) {};
		\node [style=none] (6) at (0, 0) {};
		\node [style=none] (8) at (-1, 1) {};
		\node [style=none] (9) at (0, 1) {};
		\node [style=none] (10) at (1, 1) {};
		\node [style=none] (11) at (1, 0) {};
		\node [style=none] (12) at (2, 1) {};
		\node [style=none] (13) at (1, 2) {};
		\node [style=none] (14) at (2, 2) {};
		\node [style=none] (15) at (2, 0) {};
		\node [style=none] (16) at (-2, 1) {};
	\end{pgfonlayer}
	\begin{pgfonlayer}{edgelayer}
		\draw (0.center) to (3.center);
		\draw (3.center) to (5.center);
		\draw [style=blue-double] (0.center) to (1.center);
		\draw (1.center) to (2.center);
		\draw (2.center) to (3.center);
		\draw (2.center) to (4.center);
		\draw (4.center) to (5.center);
		\draw [style=extend-edge] (6.center) to (5.center);
		\draw (5.center) to (8.center);
		\draw [style=extend-edge] (8.center) to (9.center);
		\draw (9.center) to (6.center);
		\draw (11.center) to (15.center);
		\draw (15.center) to (12.center);
		\draw (12.center) to (10.center);
		\draw (11.center) to (10.center);
		\draw (10.center) to (13.center);
		\draw [style=Blue-Thick] (13.center) to (14.center);
		\draw [style=Blue-Thick] (14.center) to (12.center);
		\draw (16.center) to (8.center);
		\draw (16.center) to (3.center);
		\draw (9.center) to (10.center);
		\draw (11.center) to (6.center);
	\end{pgfonlayer}
\end{tikzpicture}&\begin{tikzpicture}
	\begin{pgfonlayer}{nodelayer}
		\node [style=none] (0) at (-3, 0) {};
		\node [style=none] (1) at (-3, -1) {};
		\node [style=none] (2) at (-2, -1) {};
		\node [style=none] (3) at (-2, 0) {};
		\node [style=none] (4) at (-1, -1) {};
		\node [style=none] (5) at (-1, 0) {};
		\node [style=none] (6) at (0, 0) {};
		\node [style=none] (8) at (-1, 1) {};
		\node [style=none] (9) at (0, 1) {};
		\node [style=none] (10) at (1, 1) {};
		\node [style=none] (11) at (1, 0) {};
		\node [style=none] (12) at (2, 1) {};
		\node [style=none] (13) at (1, 2) {};
		\node [style=none] (14) at (2, 2) {};
		\node [style=none] (15) at (2, 0) {};
		\node [style=none] (16) at (-2, 1) {};
	\end{pgfonlayer}
	\begin{pgfonlayer}{edgelayer}
		\draw (0.center) to (3.center);
		\draw (3.center) to (5.center);
		\draw [style=Blue-Thick] (0.center) to (1.center);
		\draw [style=Blue-Thick] (1.center) to (2.center);
		\draw (2.center) to (3.center);
		\draw (2.center) to (4.center);
		\draw (4.center) to (5.center);
		\draw [style=extend-edge] (6.center) to (5.center);
		\draw (5.center) to (8.center);
		\draw [style=extend-edge] (8.center) to (9.center);
		\draw (9.center) to (6.center);
		\draw (11.center) to (15.center);
		\draw (15.center) to (12.center);
		\draw (12.center) to (10.center);
		\draw (11.center) to (10.center);
		\draw (10.center) to (13.center);
		\draw [style=Blue-Thick] (13.center) to (14.center);
		\draw [style=Blue-Thick] (14.center) to (12.center);
		\draw (16.center) to (8.center);
		\draw (16.center) to (3.center);
		\draw (9.center) to (10.center);
		\draw (11.center) to (6.center);
	\end{pgfonlayer}
\end{tikzpicture}
            \end{array}
    \right)$}\\
    $e_N,e_N$&\\
    &\\[1.8em]
    $e_{N+1},e_{N+1}$&\\
    &\\[1.8em]
    $e_N,e_{N+1}$&
    \end{tabular}
    \caption{Graphical interpretations of entries of $H_{ab}$. Here the case of $\epsilon_a = 0$, $|T|-N=1\mod 2$ is illustrated.}
    \label{fig:dd-entries}
\end{figure}

\begin{remark}
    Comparing the entries of the top-left $2$-by-$2$ submatrix  
    with the entries in the matrix appearing in Proposition 5.5 of \cite{mw13},
    we see that our new result matches the expected formulas when we reduce to the classical case, up to 
    using the identifications $x_a = e_0$, $x_b = e_1$, $x_w = e_N$ and $x_z = e_{N+1}$. 
\end{remark}

We now prove \Cref{thm:combo}.

\begin{proof}
We begin with the $(1,2)$-entry of $H_{ab}$, namely  $(-1)^{\epsilon_a} \lambda_{c_0,c_{N+1}}$.  
By Theorem 6.2(a) of \cite{moz22}, $\lambda_{c_0,c_{N+1}}$ can be expressed as the generating function counting double dimer covers in the
snake graph $G$ associated with the arc $(c_0,c_{N+1})$:
\[ \lambda_{c_0,c_{N+1}} = \frac{\sum_{M \in D(G)} \wt(M)} {x_{i_2}\cdots x_{i_{d-1}}}, \] 
where $i_1, i_2, \dots i_{d-1}, i_d$ label the arcs crossed by the arc $(\widetilde{a},\widetilde{b})$ in order.  
In particular, arc $(c_0,c_{N+1})$ crosses the same list of arcs in order, except for $i_1 = (c_0,c_1)$ and $i_d = (c_N, c_{N+1})$.  
We have a bijection between $D(G)$ and $D_{11}^{NN}(\widetilde{G})$ by appending a tile on either side of $G$ (corresponding to arcs $i_1$ and $i_d$, respectively), 
and adjoining the doubled edges $e_1$ and $e_N$.
See Figure \ref{fig:5.1a}.  After dividing through by $e_1 e_N$, accounting for the weight of the doubled edges on the first and last tiles, 
it follows that $\lambda_{c_0,c_{N+1}} = \frac{\sum_{M \in D_{11}^{NN}(\widetilde{G})} \wt(M)} { (x_{i_2}\cdots x_{i_{d-1}})e_1 e_N}$ as desired.

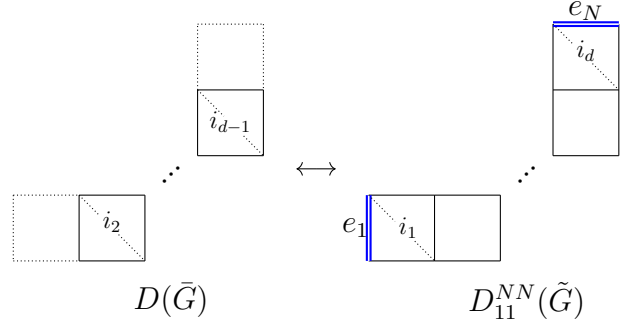
\begin{figure}
\begin{tikzpicture}[scale=0.7]
	\begin{pgfonlayer}{nodelayer}
		\node [style=none] (0) at (2.5, -1.5) {};
		\node [style=none] (1) at (2.5, 1) {};
		\node [style=none] (2) at (5, 1) {};
		\node [style=none] (3) at (5, -1.5) {};
		\node [style=none] (4) at (7.5, -1.5) {};
		\node [style=none] (5) at (7.5, 1) {};
		\node [style=none] (6) at (9.5, 2.5) {};
		\node [style=none] (7) at (12, 2.5) {};
		\node [style=none] (8) at (9.5, 5) {};
		\node [style=none] (9) at (12, 5) {};
		\node [style=none] (10) at (9.5, 7.5) {};
		\node [style=none] (11) at (12, 7.5) {};
		\node [style=none] (12) at (8.25, 1.5) {$\cdot$};
		\node [style=none] (14) at (8.75, 2) {$\cdot$};
		\node [style=none] (15) at (-11, -1.5) {};
		\node [style=none] (16) at (-8.5, -1.5) {};
		\node [style=none] (17) at (-8.5, 1) {};
		\node [style=none] (18) at (-11, 1) {};
		\node [style=none] (19) at (-6, 1) {};
		\node [style=none] (20) at (-6, -1.5) {};
		\node [style=none] (21) at (-4, 2.5) {};
		\node [style=none] (22) at (-1.5, 2.5) {};
		\node [style=none] (23) at (-1.5, 5) {};
		\node [style=none] (24) at (-4, 5) {};
		\node [style=none] (25) at (-4, 7.5) {};
		\node [style=none] (26) at (-1.5, 7.5) {};
		\node [style=none] (27) at (-0.25, 2) {};
		\node [style=none] (28) at (1.25, 2) {};
		\node [style=none] (29) at (0.5, 2) {};
		\node [style=none] (30) at (-5.25, 1.5) {$\cdot$};
		\node [style=none] (31) at (-4.75, 2) {$\cdot$};
		\node [style=none] (35) at (-5, 1.75) {$\cdot$};
		\node [style=none] (36) at (8.5, 1.75) {$\cdot$};
		\node [style=none] (37) at (-5, -3) {$D(\bar G)$};
		\node [style=none] (38) at (8.5, -3) {$D_{11}^{NN}(\tilde G)$};
		\node [style=0.8] (39) at (-7.25, 0) {$i_2$};
		\node [style=0.8] (40) at (-2.75, 3.75) {$\small{i_{d-1}} $};
		\node [style=0.8] (41) at (4, -0.25) {$i_1$};
		\node [style=0.8] (42) at (10.75, 6.5) {$i_d$};
		\node [style=none] (43) at (10.75, 8) {$e_N$};
		\node [style=none] (44) at (2, -0.25) {$e_1$};
		\node [style=none] (45) at (-7.75, 0.25) {};
		\node [style=none] (46) at (-7, -0.5) {};
		\node [style=none] (47) at (-3, 4) {};
		\node [style=none] (48) at (-2.25, 3.25) {};
		\node [style=none] (49) at (3.5, 0) {};
		\node [style=none] (50) at (4.25, -0.75) {};
		\node [style=none] (51) at (10.25, 6.75) {};
		\node [style=none] (52) at (11, 6) {};
	\end{pgfonlayer}
	\begin{pgfonlayer}{edgelayer}
		\draw (2.center) to (5.center);
		\draw (5.center) to (4.center);
		\draw (4.center) to (3.center);
		\draw (3.center) to (2.center);
		\draw (2.center) to (1.center);
		\draw (0.center) to (3.center);
		\draw (6.center) to (7.center);
		\draw (7.center) to (9.center);
		\draw (9.center) to (8.center);
		\draw (6.center) to (8.center);
		\draw (8.center) to (10.center);
		\draw (9.center) to (11.center);
		\draw [style=blue-double] (10.center) to (11.center);
		\draw [style=blue-double] (1.center) to (0.center);
		\draw [style=black arrow] (29.center) to (27.center);
		\draw [style=black arrow] (29.center) to (28.center);
		\draw (17.center) to (19.center);
		\draw (19.center) to (20.center);
		\draw (20.center) to (16.center);
		\draw (16.center) to (17.center);
		\draw (21.center) to (22.center);
		\draw (22.center) to (23.center);
		\draw (23.center) to (24.center);
		\draw (24.center) to (21.center);
		\draw [style=extend-edge] (18.center) to (17.center);
		\draw [style=extend-edge] (18.center) to (15.center);
		\draw [style=extend-edge] (15.center) to (16.center);
		\draw [style=extend-edge] (24.center) to (25.center);
		\draw [style=extend-edge] (25.center) to (26.center);
		\draw [style=extend-edge] (26.center) to (23.center);
		\draw [style=extend-edge] (24.center) to (47.center);
		\draw [style=extend-edge] (48.center) to (22.center);
		\draw [style=extend-edge] (17.center) to (45.center);
		\draw [style=extend-edge] (46.center) to (20.center);
		\draw [style=extend-edge] (10.center) to (51.center);
		\draw [style=extend-edge] (52.center) to (9.center);
		\draw [style=extend-edge] (1.center) to (49.center);
		\draw [style=extend-edge] (50.center) to (3.center);
	\end{pgfonlayer}
\end{tikzpicture}
\caption{Illustrating part of the proof of \Cref{thm:combo} for the $(1,2)$-entry.}
\label{fig:5.1a}
\end{figure}

We next consider the $(2,1)$-entry of $H_{a,b}$, namely $(-1)^{\epsilon_b}{\l{c_1}{c_N}\over \lambda_{c_0,c_1}\l {c_N}{c_{N+1}}}$.  
Assume that inside of the extended triangluation $\widetilde{T}$, the fan center $c_1$ has $k \geq 2$ internal arcs incident to it, 
including the arcs $i_1 = (c_0,c_1)$ and $i_k = (c_1,c_2)$, while the fan center $c_N$ has $\ell \geq 2$ internal arcs incident to it, including 
$i_{d-\ell+1} = (c_{N-1},c_N)$ and $i_d = (c_N, c_{N+1})$.  We let $\overline{G}$ denote the snake graph associated to the arc $(c_1,c_N)$, 
noting that $\overline{G}$ is a connected subgraph in the middle of $G$.  Then, as above, Theorem 6.2(a) of \cite{moz22} implies that 
$\lambda_{c_1,c_{N}}$ equals $\frac{\sum_{M \in D(\overline{G})} \wt(M)} { x_{i_{k+1}}\cdots x_{i_{d-\ell}}}$.

We have a bijection between $D(\overline{G})$ and $D_{00}^{N+1,N+1}(\widetilde{G})$ by appending tiles on both sides of $\overline{G}$ 
(corresponding to the zig-zag of tiles for arcs $i_1, i_2, \dots, i_k$ on the one hand, and the zig-zag of tiles 
for arcs $i_{d-\ell+1},\dots, i_{d-1}, i_d$ on the other), and adjoin the doubled edges $e_0$ and $e_N$ on tiles $i_1$ and $i_d$, respectively.  
This leads to a cascade of doubled edges from both ends of $\widetilde{G}$, giving a unqiue way to extend a given double dimer cover of $\overline{G}$.  
See \Cref{fig:5.1b}.  We also divide through by $e_0 e_{N+1}$, as well as by $x_{i_2} x_{i_3} \cdots x_{i_k}$ and $x_{i_{d-\ell+1}} \cdots x_{i_{d-2}} x_{i_{d-1}}$ 
the latter of which account for the two cascades of doubled edges.  
Noting the equalities $(c_0,c_1) = i_1$ and $(c_N, c_{N+1}) = i_d$, it follows that 
$(-1)^{\epsilon_b}{\l{c_1}{c_N}\over \lambda_{c_0,c_1}\l {c_N}{c_{N+1}}}
= (-1)^{\epsilon_b}\frac{\sum_{M \in D_{00}^{N+1,N+1}(\widetilde{G})} \wt(M)} { (x_{i_1}\cdots x_{i_{d}})e_0 e_{N+1}}$.

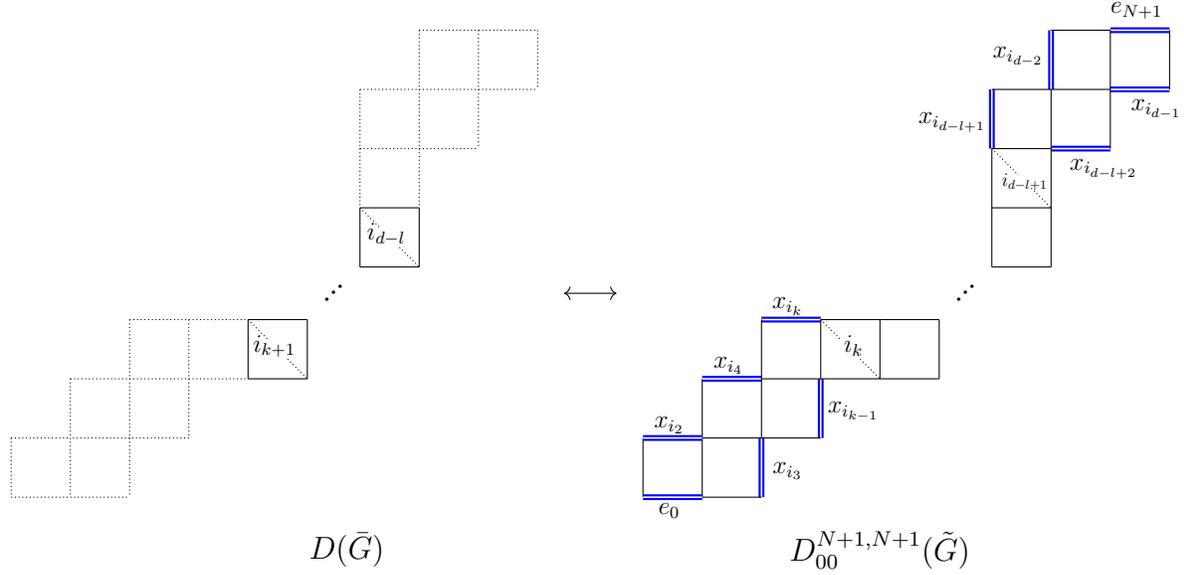
\begin{figure}
\begin{tikzpicture}[scale=0.7]
	\begin{pgfonlayer}{nodelayer}
		\node [style=none] (0) at (0.5, -3.25) {};
		\node [style=none] (1) at (2.75, -3.25) {};
		\node [style=none] (2) at (0.5, -1) {};
		\node [style=none] (3) at (2.75, -1) {};
		\node [style=none] (4) at (5, -3.25) {};
		\node [style=none] (5) at (5, -1) {};
		\node [style=none] (6) at (2.75, 1.25) {};
		\node [style=none] (7) at (5, 1.25) {};
		\node [style=none] (8) at (7.25, 1.25) {};
		\node [style=none] (9) at (7.25, -1) {};
		\node [style=none] (10) at (5, 3.5) {};
		\node [style=none] (11) at (7.25, 3.5) {};
		\node [style=none] (12) at (9.5, 1.25) {};
		\node [style=none] (13) at (9.5, 3.5) {};
		\node [style=none] (14) at (11.75, 3.5) {};
		\node [style=none] (15) at (11.75, 1.25) {};
		\node [style=none] (16) at (13.75, 5.5) {};
		\node [style=none] (17) at (16, 5.5) {};
		\node [style=none] (18) at (13.75, 7.75) {};
		\node [style=none] (19) at (16, 7.75) {};
		\node [style=none] (20) at (13.75, 10) {};
		\node [style=none] (21) at (16, 10) {};
		\node [style=none] (22) at (13.75, 12.25) {};
		\node [style=none] (23) at (16, 12.25) {};
		\node [style=none] (24) at (18.25, 10) {};
		\node [style=none] (25) at (18.25, 12.25) {};
		\node [style=none] (26) at (16, 14.5) {};
		\node [style=none] (27) at (18.25, 14.5) {};
		\node [style=none] (28) at (20.5, 14.5) {};
		\node [style=none] (29) at (20.5, 12.25) {};
		\node [style=none] (30) at (12.5, 4.25) {$\cdot$};
		\node [style=none] (31) at (12.75, 4.5) {$\cdot$};
		\node [style=none] (32) at (13, 4.75) {$\cdot$};
		\node [style=none] (33) at (-23.5, -3.25) {};
		\node [style=none] (34) at (-21.25, -3.25) {};
		\node [style=none] (35) at (-23.5, -1) {};
		\node [style=none] (36) at (-21.25, -1) {};
		\node [style=none] (37) at (-19, -3.25) {};
		\node [style=none] (38) at (-19, -1) {};
		\node [style=none] (39) at (-21.25, 1.25) {};
		\node [style=none] (40) at (-19, 1.25) {};
		\node [style=none] (41) at (-16.75, 1.25) {};
		\node [style=none] (42) at (-16.75, -1) {};
		\node [style=none] (43) at (-19, 3.5) {};
		\node [style=none] (44) at (-16.75, 3.5) {};
		\node [style=none] (45) at (-14.5, 1.25) {};
		\node [style=none] (46) at (-14.5, 3.5) {};
		\node [style=none] (47) at (-12.25, 3.5) {};
		\node [style=none] (48) at (-12.25, 1.25) {};
		\node [style=none] (49) at (-10.25, 5.5) {};
		\node [style=none] (50) at (-8, 5.5) {};
		\node [style=none] (51) at (-10.25, 7.75) {};
		\node [style=none] (52) at (-8, 7.75) {};
		\node [style=none] (53) at (-10.25, 10) {};
		\node [style=none] (54) at (-8, 10) {};
		\node [style=none] (55) at (-10.25, 12.25) {};
		\node [style=none] (56) at (-8, 12.25) {};
		\node [style=none] (57) at (-5.75, 10) {};
		\node [style=none] (58) at (-5.75, 12.25) {};
		\node [style=none] (59) at (-8, 14.5) {};
		\node [style=none] (60) at (-5.75, 14.5) {};
		\node [style=none] (61) at (-3.5, 14.5) {};
		\node [style=none] (62) at (-3.5, 12.25) {};
		\node [style=none] (63) at (-11.5, 4.25) {$\cdot$};
		\node [style=none] (64) at (-11.25, 4.5) {$\cdot$};
		\node [style=none] (65) at (-11, 4.75) {$\cdot$};
		\node [style=none] (66) at (-10.75, -5.25) {$D(\bar G)$};
		\node [style=none] (67) at (9.5, -5.25) {$D_{00}^{N+1,N+1}(\tilde G)$};
		\node [style=0.8] (68) at (-13.5, 2.5) {$i_{k+1}$};
		\node [style=0.8] (69) at (-9.25, 6.75) {$i_{d-l}$};
		\node [style=0.8] (70) at (8.5, 2.5) {$i_k$};
		\node [style=0.6] (71) at (15, 8.75) {${i_{d-l+1}}$};
		\node [style=none] (72) at (-2.5, 4.5) {};
		\node [style=none] (73) at (-0.5, 4.5) {};
		\node [style=none] (74) at (-1.5, 4.5) {};
		\node [style=none] (75) at (14.75, 9) {};
		\node [style=none] (76) at (15.25, 8.5) {};
		\node [style=0.8] (77) at (1.5, -3.75) {$e_0$};
		\node [style=0.8] (78) at (1.5, -0.5) {$x_{i_2}$};
		\node [style=0.8] (79) at (6, -2.25) {$x_{i_3}$};
		\node [style=0.8] (80) at (8.5, 0) {$x_{i_{k-1}}$};
		\node [style=0.8] (81) at (6, 4) {$x_{i_k}$};
		\node [style=0.8] (82) at (3.75, 1.75) {$x_{i_4}$};
		\node [style=0.8] (83) at (12.25, 11) {$x_{i_{d-l+1}}$};
		\node [style=0.8] (84) at (18, 9.25) {$x_{i_{d-l+2}}$};
		\node [style=0.8] (85) at (14.75, 13.5) {$x_{i_{d-2}}$};
		\node [style=0.8] (86) at (19.25, 15.25) {$e_{N+1}$};
		\node [style=0.8] (87) at (20, 11.5) {$x_{i_{d-1}}$};
		\node [style=none] (88) at (-9.5, 7) {};
		\node [style=none] (89) at (-8.75, 6.25) {};
		\node [style=none] (90) at (-13.75, 2.75) {};
		\node [style=none] (91) at (-13, 2) {};
		\node [style=none] (92) at (8, 2.75) {};
		\node [style=none] (93) at (8.75, 2) {};
	\end{pgfonlayer}
	\begin{pgfonlayer}{edgelayer}
		\draw (14.center) to (15.center);
		\draw (13.center) to (12.center);
		\draw (6.center) to (1.center);
		\draw (13.center) to (14.center);
		\draw (12.center) to (15.center);
		\draw (2.center) to (0.center);
		\draw [style=blue-double] (27.center) to (28.center);
		\draw [style=blue-double] (25.center) to (29.center);
		\draw [style=blue-double] (26.center) to (23.center);
		\draw [style=blue-double] (21.center) to (24.center);
		\draw [style=blue-double] (22.center) to (20.center);
		\draw (18.center) to (19.center);
		\draw (21.center) to (17.center);
		\draw (17.center) to (16.center);
		\draw (16.center) to (20.center);
		\draw (20.center) to (21.center);
		\draw (22.center) to (23.center);
		\draw (23.center) to (21.center);
		\draw (23.center) to (25.center);
		\draw (25.center) to (24.center);
		\draw (26.center) to (27.center);
		\draw (27.center) to (25.center);
		\draw (28.center) to (29.center);
		\draw [style=blue-double] (0.center) to (1.center);
		\draw [style=blue-double] (2.center) to (3.center);
		\draw [style=blue-double] (5.center) to (4.center);
		\draw [style=blue-double] (6.center) to (7.center);
		\draw [style=blue-double] (10.center) to (11.center);
		\draw [style=blue-double] (8.center) to (9.center);
		\draw (11.center) to (13.center);
		\draw (11.center) to (8.center);
		\draw (8.center) to (12.center);
		\draw (10.center) to (7.center);
		\draw (7.center) to (8.center);
		\draw (9.center) to (5.center);
		\draw (7.center) to (5.center);
		\draw (5.center) to (3.center);
		\draw (1.center) to (4.center);
		\draw (47.center) to (48.center);
		\draw (46.center) to (45.center);
		\draw [style=extend-edge] (39.center) to (34.center);
		\draw (46.center) to (47.center);
		\draw (45.center) to (48.center);
		\draw [style=extend-edge] (35.center) to (33.center);
		\draw [style=extend-edge] (60.center) to (61.center);
		\draw [style=extend-edge] (58.center) to (62.center);
		\draw [style=extend-edge] (59.center) to (56.center);
		\draw [style=extend-edge] (54.center) to (57.center);
		\draw [style=extend-edge] (55.center) to (53.center);
		\draw (51.center) to (52.center);
		\draw (50.center) to (49.center);
		\draw [style=extend-edge] (53.center) to (54.center);
		\draw [style=extend-edge] (55.center) to (56.center);
		\draw [style=extend-edge] (56.center) to (54.center);
		\draw [style=extend-edge] (56.center) to (58.center);
		\draw [style=extend-edge] (58.center) to (57.center);
		\draw [style=extend-edge] (59.center) to (60.center);
		\draw [style=extend-edge] (60.center) to (58.center);
		\draw [style=extend-edge] (61.center) to (62.center);
		\draw [style=extend-edge] (33.center) to (34.center);
		\draw [style=extend-edge] (35.center) to (36.center);
		\draw [style=extend-edge] (38.center) to (37.center);
		\draw [style=extend-edge] (39.center) to (40.center);
		\draw [style=extend-edge] (43.center) to (44.center);
		\draw [style=extend-edge] (41.center) to (42.center);
		\draw [style=extend-edge] (44.center) to (46.center);
		\draw [style=extend-edge] (44.center) to (41.center);
		\draw [style=extend-edge] (41.center) to (45.center);
		\draw [style=extend-edge] (43.center) to (40.center);
		\draw [style=extend-edge] (40.center) to (41.center);
		\draw [style=extend-edge] (42.center) to (38.center);
		\draw [style=extend-edge] (40.center) to (38.center);
		\draw [style=extend-edge] (38.center) to (36.center);
		\draw [style=extend-edge] (34.center) to (37.center);
		\draw (51.center) to (49.center);
		\draw (52.center) to (50.center);
		\draw [style=extend-edge] (54.center) to (52.center);
		\draw [style=extend-edge] (53.center) to (51.center);
		\draw [style=black arrow] (74.center) to (72.center);
		\draw [style=black arrow] (74.center) to (73.center);
		\draw [style=dashed] (20.center) to (75.center);
		\draw [style=dashed] (76.center) to (19.center);
		\draw [style=extend-edge] (51.center) to (88.center);
		\draw [style=extend-edge] (89.center) to (50.center);
		\draw [style=extend-edge] (46.center) to (90.center);
		\draw [style=extend-edge] (91.center) to (48.center);
		\draw [style=extend-edge] (11.center) to (92.center);
		\draw [style=extend-edge] (93.center) to (12.center);
	\end{pgfonlayer}
\end{tikzpicture}
\caption{Illustrating part of the proof of \Cref{thm:combo} for the $(2,1)$-entry.}
\label{fig:5.1b}
\end{figure}

The proofs for the validity of the $(1,1)$- and $(2,2)$-entries are analogous and involve combinations of the previous two cases.

To prove the result for the $(1,3)$-, $(2,3)$-, $(3,1)$-, $(3,2)$-, and $(3,3)$-entries 
takes further work, and combining together Theorem 6.2(b) and Lemma 5.8, both of \cite{moz22}, as we now show:

Consider the $(1,3)$-entry of $H_{ab}$, namely 
$\Td^{c_{N+1}}_{c_0,c_1}  = \sqrt{ \frac{\lambda_{c_0, c_{N+1}} \lambda_{c_1,c_{N+1}}}{\lambda_{c_0,c_1}}} \boxed{c_0, c_1, c_{N+1}}$.

We wish to apply Theorem 6.2(b) of \cite{moz22} here to simplify this expression, but before we can do so we need to redraw the extended 
triangulation $\widetilde{T}$ so that it matches the illustration in Figure 13(a) of \cite{moz22} so that $\boxed{c_0, c_1, c_{N+1}} = \varphi = \boxed{ijk}$.  
In particular, let $v$ denote the endpoint of arc $i_2 = (c_1, v)$ so that the first two triangles of $\widetilde{T}$ are
$(\widetilde{a},c_0,c_1)$ and $(c_0,c_1, v)$ respectively.  Then in the notation of Figure 13(a) of \cite{moz22}, we have 
$a = (c_0,c_1)$, $b = (c_0,v)$, $d = (c_1,c_{N+1})$, $e=(c_1,v)$, and $f = (c_0,c_{N+1})$. 
If the first diagonal $e$ is oriented incorrectly, we can reverse it and replace $\varphi \mapsto -\varphi$.
As discussed in Remark 2.6 of \cite{moz22}, this does not change the positive ordering.

Theorem 6.2(b) of \cite{moz22} then yields 
\[ 
    \sqrt{\lambda_{c_1,c_{N+1}}\lambda_{c_0, c_{N+1}}} \boxed{c_0, c_1, c_{N+1}} =
    \frac{1}{x_{i_2}\cdots x_{i_{d-1}}} \sqrt{ \frac{\lambda_{c_1,v}}{\lambda_{c_0,v}}} \sum_{M \in D_0(G)} \wt(M)^{(*2)} 
\]
where $G$ is the snake graph associated to the arc $(c_0, c_{N+1})$ (just as above) 
and $D_0(G) = D_{00}(G) \cup D_{01}(G)$ denotes the subset of double dimer covers that uses edge $(c_0,v)$ as a single or doubled edge on the first tile of $G$.
\footnote{In \cite{moz22}, Theorem 6.2(b) uses the notations $D_t(G)$ and $D_r(G)$. The new notations used here allow a more uniform treatment.} 

The notation $\wt(M)^{(*2)}$ also signifies that the weight of the double dimer cover $M \in D_0(G)$ is altered by toggling $\boxed{c_0,c_1,v}$, 
the $\mu$-invariant corresponding to the lower left triangle of the first tile of $G$ (second tile of $\widetilde{G}$). 

By Lemma 5.8 of \cite{moz22}, there is a bijection between $D_0(G)$ and in $D_{01}(G^+)$ where $G^+$ is the subgraph of $\widetilde{G}$ 
that contains tiles $i_1,i_2,\dots, i_{d-1}$ (i.e. it contains subgraph $G$ plus tile $i_1$), where the weights are related by
$\sum_{M \in D_0(G)} \wt(M)^{(*2)} = \sqrt{\frac{\lambda_{c_0,v}} {e_0 e_1 \lambda_{c_1,v}} } \sum_{M \in D_{01}(G^+)} \wt(M)^*$
where $(*2)$ toggles the weight by $\boxed{c_0,c_1,v}$ and 
$(*)$ toggles the weight by $\theta_{\widetilde{a}} = \boxed{\widetilde{a},c_0,c_1}$.

The quantity  $\sqrt{ \frac{ \lambda_{c_0,v}} {e_0 e_1 \lambda_{c_1,v}} }$ is based on the edge weights on the first tile, and recalling that arc $i_2 = (c_1,v)$.

Putting this altogether, and remembering that $\lambda_{c_0,c_1} = x_{i_1}$, we get
$$\sqrt{ \frac{\lambda_{c_1,c_{N+1}}\lambda_{c_0, c_{N+1}}}{\lambda_{c_0,c_1}}}
\boxed{c_0, c_1, c_{N+1}} = 
\frac{1}{\sqrt{x_{i_1}} (x_{i_2}\cdots x_{i_{d-1}})}
\sqrt{ \frac{\lambda_{c_1,v}}{\lambda_{c_0,v}}} \sqrt{ \frac{ \lambda_{c_0,v}} {e_0 e_1 \lambda_{c_1,v}} } \sum_{M \in D_{01}(G^+)} \wt(M)^*.$$
Furthermore, there is another straightforward bijection between $D_{01}(G^+)$ and $D_{01}^{NN}(\widetilde{G})$
by adjoining the last tile $i_d$, and utilizing the edge $e_N$ as a double edge.  See \Cref{fig:5.1c}.
Dividing through by this contribution, and noting that the weight of the single forced edges on the first tile is $\sqrt{e_0e_1}$, 
we thus conclude that the $(1,3)$-entry of $H_{a,b}$ is 
\[ \frac{1}{ \sqrt{e_0~e_1} \sqrt{x_{i_1}}(x_{i_2}\cdots x_{i_{d-1}})e_N}\sum_{M \in D_{01}^{NN}(\widetilde{G})} \wt(M)^* \] 
as desired.

\begin{figure}
\begin{tikzpicture}[scale=0.7]
	\begin{pgfonlayer}{nodelayer}
		\node [style=none] (0) at (2, -1.75) {};
		\node [style=none] (1) at (2, 1) {};
		\node [style=none] (2) at (4.75, 1) {};
		\node [style=none] (3) at (4.75, -1.75) {};
		\node [style=none] (4) at (7.5, -1.75) {};
		\node [style=none] (5) at (7.5, 1) {};
		\node [style=none] (6) at (9.5, 2.5) {};
		\node [style=none] (7) at (12.25, 2.5) {};
		\node [style=none] (8) at (9.5, 5.25) {};
		\node [style=none] (9) at (12.25, 5.25) {};
		\node [style=none] (10) at (9.5, 8) {};
		\node [style=none] (11) at (12.25, 8) {};
		\node [style=none] (12) at (8.25, 1.5) {$\cdot$};
		\node [style=none] (14) at (8.75, 2) {$\cdot$};
		\node [style=none] (15) at (-11.5, -1.75) {};
		\node [style=none] (16) at (-8.75, -1.75) {};
		\node [style=none] (17) at (-8.75, 1) {};
		\node [style=none] (18) at (-11.5, 1) {};
		\node [style=none] (19) at (-6, 1) {};
		\node [style=none] (20) at (-6, -1.75) {};
		\node [style=none] (21) at (-4, 2.5) {};
		\node [style=none] (22) at (-1.25, 2.5) {};
		\node [style=none] (23) at (-1.25, 5.25) {};
		\node [style=none] (24) at (-4, 5.25) {};
		\node [style=none] (25) at (-4, 8) {};
		\node [style=none] (26) at (-1.25, 8) {};
		\node [style=none] (27) at (-0.25, 2) {};
		\node [style=none] (28) at (1.25, 2) {};
		\node [style=none] (29) at (0.5, 2) {};
		\node [style=none] (30) at (-5.25, 1.5) {$\cdot$};
		\node [style=none] (31) at (-4.75, 2) {$\cdot$};
		\node [style=none] (35) at (-5, 1.75) {$\cdot$};
		\node [style=none] (36) at (8.5, 1.75) {$\cdot$};
		\node [style=none] (37) at (-5, -3) {$D_{01}(G^+)$};
		\node [style=none] (38) at (8.5, -3) {$D_{01}^{NN}(\tilde G)$};
		\node [style=none] (39) at (-7.25, -0.25) {$i_2$};
		\node [style=node small] (40) at (-2.75, 4) {$\small{i_{d-1}} $};
		\node [style=none] (41) at (3.5, -0.25) {$i_1$};
		\node [style=none] (42) at (11, 6.75) {$i_d$};
		\node [style=none] (43) at (10.75, 8.5) {$e_N$};
		\node [style=none] (44) at (1.5, -0.5) {$e_1$};
		\node [style=none] (45) at (-9.75, -2.25) {$e_0$};
		\node [style=none] (46) at (-12, -0.25) {$e_1$};
		\node [style=none] (47) at (3.5, -2.25) {$e_0$};
		\node [style=none] (48) at (-3, 4.25) {};
		\node [style=none] (49) at (-2.25, 3.5) {};
		\node [style=none] (50) at (10.5, 7) {};
		\node [style=none] (51) at (11.25, 6.25) {};
		\node [style=none] (52) at (3, 0) {};
		\node [style=none] (53) at (3.75, -0.75) {};
		\node [style=none] (54) at (-7.75, 0) {};
		\node [style=none] (55) at (-7, -0.75) {};
	\end{pgfonlayer}
	\begin{pgfonlayer}{edgelayer}
		\draw (2.center) to (5.center);
		\draw (5.center) to (4.center);
		\draw (4.center) to (3.center);
		\draw (3.center) to (2.center);
		\draw (2.center) to (1.center);
		\draw [style=Blue-Thick] (0.center) to (3.center);
		\draw (6.center) to (7.center);
		\draw (7.center) to (9.center);
		\draw (9.center) to (8.center);
		\draw (6.center) to (8.center);
		\draw (8.center) to (10.center);
		\draw (9.center) to (11.center);
		\draw [style=blue-double] (10.center) to (11.center);
		\draw [style=Blue-Thick] (1.center) to (0.center);
		\draw [style=black arrow] (29.center) to (27.center);
		\draw [style=black arrow] (29.center) to (28.center);
		\draw (17.center) to (19.center);
		\draw (19.center) to (20.center);
		\draw (20.center) to (16.center);
		\draw (16.center) to (17.center);
		\draw (21.center) to (22.center);
		\draw (22.center) to (23.center);
		\draw (23.center) to (24.center);
		\draw (24.center) to (21.center);
		\draw (18.center) to (17.center);
		\draw [style=Blue-Thick] (18.center) to (15.center);
		\draw [style=Blue-Thick] (15.center) to (16.center);
		\draw [style=extend-edge] (24.center) to (25.center);
		\draw [style=extend-edge] (25.center) to (26.center);
		\draw [style=extend-edge] (26.center) to (23.center);
		\draw [style=extend-edge] (24.center) to (48.center);
		\draw [style=extend-edge] (49.center) to (22.center);
		\draw [style=extend-edge] (10.center) to (50.center);
		\draw [style=extend-edge] (51.center) to (9.center);
		\draw [style=extend-edge] (1.center) to (52.center);
		\draw [style=extend-edge] (53.center) to (3.center);
		\draw [style=extend-edge] (17.center) to (54.center);
		\draw [style=extend-edge] (55.center) to (20.center);
	\end{pgfonlayer}
\end{tikzpicture}
\caption{Illustrating part of the proof of \Cref{thm:combo} for the $(1,3)$-entry.}
\label{fig:5.1c}
\end{figure}

We use an analogous argument to verify the formulas for the $(2,3)$-, $(3,1)$-, and $(3,2)$-entries,
noting that we must sometimes divide by $\lambda_{c_0,c_1} = x_{i_1}$ or $\lambda_{c_N,c_{N+1}} = x_{i_d}$ to get the formula.
Also, as mentioned in \Cref{rmk:local_signs}, the terms in $\widetilde{\alpha}$ and $\widetilde{\beta}$ will sometimes have signs.
This is because Theorem 6.2(b) from \cite{moz22} implicitly assumes one of two possible choices of positive ordering
of the odd variables, and we need to use the opposite choice when applying the theorem to the $(3,1)$- and $(3,2)$-entries.

For the case of the $(3,3)$-entry of $H_{a,b}$, we need to show that 
$$1 + (-1)^{\epsilon_a-1}{1\over\lambda_{c_0,c_1}}\td {c_1}{c_N}{c_{N+1}}\td {c_0} {c_N}{c_{N+1}}
= 1 + (-1)^{\epsilon_b-1}{1\over \l{c_N}{c_{N+1}}}\Td^{c_{N+1}}_{c_0,c_1} \Td^{c_{N}}_{c_0,c_1}$$
equals  
$$\frac{1}{  (x_{i_2}\cdots x_{i_{d-1}})\sqrt{x_{i_1} x_{i_d}} \sqrt{e_0 e_1 e_N e_{N+1}}}{\sum\limits_{M \in D_{0,1}^{N,N+1}(\widetilde{G})} (\wt(M)^*)^\dagger} .$$

Since we have already shown that the $(3,1)$- and $(3,2)$-entries each have the desired combinatorial interpretations, then by \Cref{eq:osp1}, it remains to show that
\[ \widetilde{E} - \partial = \frac{(-1)^{\epsilon_a - 1}}{\partial} \widetilde{\alpha} \widetilde{\beta}  \]
where $\partial := \sqrt{ e_0 e_1 x_{i_1} (x_{i_2}^2\cdots x_{i_{d-1}}^2) x_{i_d} e_N e_{N+1}}$ is the product of the (square roots of the) edge weights of
all the outer boundary sides of $\widetilde{G}$.

Note that among the double dimer covers in $D_{0,1}^{N,N+1}(\widetilde{G})$, there is the unique one that consists of a single cycle comprised of the entire boundary of $\widetilde{G}$.  
Based on the definition of edge weights on $\widetilde{G}$, this will contribute to $\widetilde{E}$ a weight of $((\partial \, \theta_{\tilde{a}} \theta_{\tilde{b}})^\ast)^\dagger = \partial$.
This means the left-hand side, $\widetilde{E} - \partial$, is simply the sum over $D_{01}^{N,N+1}$ minus this one special element. 

Therefore, we need to show that
\[ 
    (-1)^{\epsilon_a-1}\left( \sum_{M \in D_{00}^{N,N+1}} \wt(M)^\dagger \right) \left( \sum_{M \in D_{11}^{N,N+1}} \wt(M)^\dagger \right) 
    = \wt(M_0) \sum_{M \in D_{01}^{N,N+1} \setminus \{M_0\}} (\wt(M)^\ast)^\dagger 
\]
where $M_0$ the special double dimer cover mentioned above (with weight $\wt(M_0) = \partial$).
We will do this by seeing that some pairs of terms in the product on the left-hand side cancel, and that the remaining terms are in bijection
with the terms on the right-hand side, with weights differing by a factor of $\partial$.

Let $(M_1,M_2)$ be a pair of double dimer covers, where
$M_1$ is in the $\widetilde{\alpha}$ sum (over $D_{00}^{N,N+1}$) and $M_2$ is from the $\widetilde{\beta}$ sum (over $D_{11}^{N,N+1}$).
Let $i_t$ denote the first internal (non-boundary) edge of $\widetilde{G}$ which is used in either $M_1$ or $M_2$. 

First consider the case when $i_t$ is used as a double edge in either $M_1$ or $M_2$. 
This is illustrated in \Cref{fig:i_t-doubled}.
Note that it cannot be used as a double edge in both
because it is the \emph{first} occurrence of an internal edge, and so at least one of the $M_i$'s uses an adjacent boundary edge twice.
Let us assume (for simplicity of the following exposition) that $i_t$ belongs to $M_1$.
In this case, the next two boundary edges adjacent to $i_t$ (those immediately to the right or above) can be used by neither $M_1$ nor $M_2$. 
Thus we can swap the portions of $M_1$ and $M_2$ to the right of $i_t$, and obtain a new pair $(M_1',M_2')$ using the same edges. However, since
we swapped the portions at the \emph{end}, the odd variables corresponding to the cycles ending on the last tile now are multiplied in the opposite order.
Therefore $\wt(M_1)^\dagger \wt(M_2)^\dagger + \wt(M_1')^\dagger \wt(M_2')^\dagger = 0$.

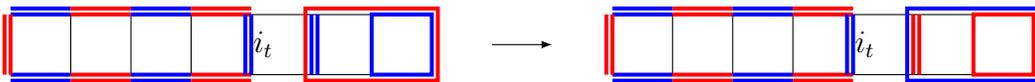
\begin {figure}[h!]
\centering
\begin {tikzpicture} [scale=0.8]
    \draw (0,0) grid (7,1);

    \foreach \x in {0,2} {
        \foreach \y in {-0.1, 0, 1, 1.1} {
            \draw[blue, line width=1.5]      (\x,\y)   -- (\x+1,\y);
            \draw[color=red, line width=1.5] (\x+1,\y) -- (\x+2,\y);
        }
    }

    \foreach \x in {-0.1, 0} {
        \draw[color=red, line width=1.5] (\x,0) -- (\x,1);
    }

    \foreach \x in {3.9, 4, 5, 5.1} {
        \draw[blue, line width=1.5] (\x,0) -- (\x,1);
    }

    \draw[blue, line width=1.5] (7,1) -- (6,1) -- (6,0) -- (7,0) -- cycle;
    \draw[color=red, line width=1.5] (7.1,1.1) -- (4.9,1.1) -- (4.9,-0.1) -- (7.1,-0.1) -- cycle;

    \draw (4.2, 0.5) node {$i_t$};

    \draw[-latex] (8,0.5) -- (9,0.5);

    \begin {scope}[shift={(10,0)}]
        \draw (0,0) grid (7,1);

        \foreach \x in {0,2} {
            \foreach \y in {-0.1, 0, 1, 1.1} {
                \draw[blue, line width=1.5]      (\x,\y)   -- (\x+1,\y);
                \draw[color=red, line width=1.5] (\x+1,\y) -- (\x+2,\y);
            }
        }

        \foreach \x in {-0.1, 0, 5, 5.1} {
            \draw[color=red, line width=1.5] (\x,0) -- (\x,1);
        }

        \foreach \x in {3.9, 4} {
            \draw[blue, line width=1.5] (\x,0) -- (\x,1);
        }

        \draw[color=red, line width=1.5] (7,1) -- (6,1) -- (6,0) -- (7,0) -- cycle;
        \draw[blue, line width=1.5] (7.1,1.1) -- (4.9,1.1) -- (4.9,-0.1) -- (7.1,-0.1) -- cycle;

        \draw (4.2, 0.5) node {$i_t$};
    \end {scope}
\end {tikzpicture}
\caption {The case when $i_t$ is a double edge. $M_1$ is in blue, and $M_2$ is in red.}
\label {fig:i_t-doubled}
\end {figure}

In all the remining terms, $i_t$ is not used as a doubled edge, so it is used only once in either $M_1$ or $M_2$. Note that if it is used as a single edge in
both $M_1$ and $M_2$ then we would have two cycles (one from $M_1$ and one from $M_2$) which contribute the same odd variable. But since $\theta^2 = 0$
for all odd variables, such terms would contribute a weight of zero. So we do not need to consider such configurations. So we only consider the case that $i_t$
is used once in exactly one of the $M_i$'s (and is used either twice or not at all in the other). Again, assume for the sake of exposition that $i_t$
is used only once in $M_1$.

At this point we further divide into two cases: either the cycle of $M_1$ beginning at $i_t$ continues until the last tile of $\widetilde{G}$, or not.

Consider first the former case, which is pictured in \Cref{fig:i_t-cycle-to-the-end}. 
Note that in the union $M_1 \cup M_2$, all boundary edges are used at least once. Indeed, all boundary edges before $i_t$
are used twice, and all boundary edges after $i_t$ are used in the cycle beginning at $i_t$. Therefore the product of the weights of $M_1$ and $M_2$
is divisible by $\partial$, and what remains after deleting one instance of each boundary edge is a double dimer cover with at least two cycles: 
one starting at the first tile and ending at $i_t$, and one ending at the last tile of $\widetilde{G}$ (coming from $M_2$).
This is precisely the type of configurations counted by $\widetilde{E} - \partial$.

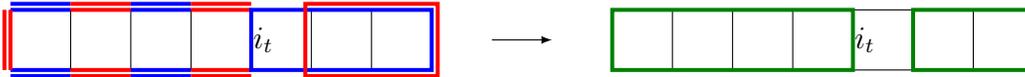
\begin {figure}[h!]
\centering
\begin {tikzpicture} [scale=0.8]
    \draw (0,0) grid (7,1);

    \foreach \x in {0,2} {
        \foreach \y in {-0.1, 0, 1, 1.1} {
            \draw[blue, line width=1.5]      (\x,\y)   -- (\x+1,\y);
            \draw[color=red, line width=1.5] (\x+1,\y) -- (\x+2,\y);
        }
    }

    \foreach \x in {-0.1, 0} {
        \draw[color=red, line width=1.5] (\x,0) -- (\x,1);
    }

    \draw[blue, line width=1.5] (7,1) -- (4,1) -- (4,0) -- (7,0) -- cycle;
    \draw[color=red, line width=1.5] (7.1,1.1) -- (4.9,1.1) -- (4.9,-0.1) -- (7.1,-0.1) -- cycle;

    \draw (4.2, 0.5) node {$i_t$};

    \draw[-latex] (8,0.5) -- (9,0.5);

    \begin {scope}[shift={(10,0)}]
        \draw (0,0) grid (7,1);

        \draw[green!50!black, line width=1.5] (0,0) -- (4,0) -- (4,1) -- (0,1) -- cycle;
        \draw[green!50!black, line width=1.5] (5,0) -- (7,0) -- (7,1) -- (5,1) -- cycle;

        \draw (4.2, 0.5) node {$i_t$};
    \end {scope}
\end {tikzpicture}
\caption {The case when $i_t$ is part of a cycle going to the last tile ($M_1$ in blue, $M_2$ in red). 
Removing one copy of each boundary edge gives an element of $\widetilde{E}$ (in green).}
\label {fig:i_t-cycle-to-the-end}
\end {figure}

In the latter case (when the cycle beginning at $i_t$ does \emph{not} extend all the way to the last tile of $\widetilde{G}$), then let $i_s$
be the internal edge which is either the top or right edge of the last tile of this cycle. We now consider the different cases by looking at the two boundary edges
immediately to the right/above $i_s$. There are three cases, depending on if these boundary edges are used once, twice, or not at all by $M_2$.

In the case that the boundary edges adjacent to $i_s$ are not used at all in either $M_1$ or $M_2$, then as described earlier (and pictured in \Cref{fig:i_t-doubled}), 
we may swap the parts of $M_1$ and $M_2$ occuring after $i_s$ to get another pair 
$(M_1',M_2')$ whose product of weights cancels with $(M_1,M_2)$. 

If these boundary edges are used twice by $M_2$ then replacing $i_s$ 
with a double edge and looking at the truncated snake graph from $i_s$ to the end, we are in the same situation we started with: the truncated $M_1$ has a double edge on one side
of the first tile (either the left or bottom), and the truncated $M_2$ has a double edge on the other, while both still end with cycles. Therefore we may repeat the
argument up to this point, looking for the next occurrence of an internal edge used by either $M_1$ or $M_2$, and finding either another term that cancels, or concluding that
this product $\frac{1}{\partial} \wt(M_1)^\dagger \wt(M_2)^\dagger$ represents a term from $\widetilde{E} - \partial$.

\begin {figure}[h!]
\centering
\begin {tikzpicture} [scale=0.8]
    \draw (0,0) grid (7,1);

    \foreach \y in {-0.1, 0, 1, 1.1} {
        \draw[blue, line width=1.5]      (0,\y)   -- (1,\y);
        \draw[color=red, line width=1.5] (1,\y) -- (2,\y);
    }

    \foreach \y in {-0.1, 0, 1, 1.1} {
        \draw[color=red, line width=1.5] (4,\y) -- (5,\y);
    }

    \foreach \x in {-0.1, 0} {
        \draw[color=red, line width=1.5] (\x,0) -- (\x,1);
    }

    \draw[blue, line width=1.5] (2,0) -- (4,0) -- (4,1) -- (2,1) -- cycle;
    \draw[color=red, line width=1.5] (6,0) -- (7,0) -- (7,1) -- (6,1) -- cycle;
    \draw[blue, line width=1.5] (7.1,1.1) -- (4.9,1.1) -- (4.9,-0.1) -- (7.1,-0.1) -- cycle;

    \draw (4.2, 0.5) node {$i_s$};
    \draw (2.2, 0.5) node {$i_t$};

    \draw[-latex] (8,0.5) -- (9,0.5);

    \begin {scope}[shift={(10,0)}]
        \draw[dashed] (0,0) grid (4,1);
        \draw         (4,0) grid (7,1);

        \foreach \x in {3.9, 4} {
            \draw[blue, line width=1.5] (\x,0) -- (\x,1);
        }

        \foreach \y in {-0.1, 0, 1, 1.1} {
            \draw[color=red, line width=1.5] (4,\y) -- (5,\y);
        }

        \draw[color=red, line width=1.5] (6,0) -- (7,0) -- (7,1) -- (6,1) -- cycle;
        \draw[blue, line width=1.5] (7.1,1.1) -- (4.9,1.1) -- (4.9,-0.1) -- (7.1,-0.1) -- cycle;

        \draw (4.25, 0.5) node {$i_s$};
    \end {scope}
\end {tikzpicture}
\caption {The case when $i_t$ is part of a cycle that does \emph{not} go to the last tile ($M_1$ in blue, $M_2$ in red). 
Removing the part of the picture before $i_s$ results in a similar situation, in a smaller graph.}
\label {fig:i_t-cycle-not-to-the-end}
\end {figure}
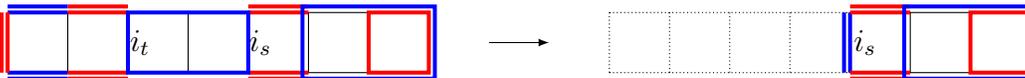

The final case that has not been considered is when $M_2$ uses the boundary sides adjacent to $i_s$ once (not doubled). As mentioned before, $M_2$ cannot have a cycle
beginning or ending adjacent to $i_s$ (else the weight would be zero). Therefore we need only consider the case that $M_2$ has a cycle beginning before $i_s$ and
ending after $i_s$. Let $i_{s'}$ be the internal edge on the end of this cycle of $M_2$. We continue the current argument with $i_{s'}$ instead of $i_s$. If $i_{s'}$
is not on the boundary of the last tile, we continue to look at the cases of the boundary sides adjacent to $i_{s'}$. Finally, if $i_{s'}$ is on the boundary of the
last tile (which must eventually happen, since both $M_1$ and $M_2$ are assumed to end with cycles), then we are back in the earlier case and this gives a term
from $\widetilde{E} - \partial$.
\begin {figure}[h!]
\centering
\begin {tikzpicture} [scale=0.8]
    \draw (0,0) grid (7,1);

    \foreach \y in {-0.1, 0, 1, 1.1} {
        \draw[blue, line width=1.5]      (0,\y)   -- (1,\y);
        \draw[color=red, line width=1.5] (1,\y) -- (2,\y);
    }

    \draw[color=red, line width=1.5] (3,-0.1) -- (5,-0.1) -- (5,1.1) -- (3,1.1) -- cycle;

    \foreach \x in {-0.1, 0} {
        \draw[color=red, line width=1.5] (\x,0) -- (\x,1);
    }

    \draw[blue, line width=1.5] (2,0) -- (4,0) -- (4,1) -- (2,1) -- cycle;
    \draw[color=red, line width=1.5] (6,0) -- (7,0) -- (7,1) -- (6,1) -- cycle;
    \draw[blue, line width=1.5] (7.1,1.1) -- (4.9,1.1) -- (4.9,-0.1) -- (7.1,-0.1) -- cycle;

    \draw (2.2, 0.5) node {$i_t$};
    \draw (4.2, 0.5) node {$i_s$};
    \draw (5.25, 0.5) node {$i_{s'}$};
\end {tikzpicture}
\caption {The case when the cycle starting at $i_t$ overlaps with a cycle from $M_2$ ($M_1$ in blue, $M_2$ in red). 
We can continue the argument with $i_{s'}$ instead of $i_s$.}
\label {fig:i_s-prime}
\end {figure}
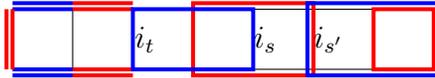

The inverse map, which shows that each term of $\tilde{E}-\partial$ can be written uniquely as the product of terms from $\alpha$ and $\beta$, 
can be constructed using an analysis similar to the above argument. 
\end{proof}

\section{Super Fibonacci Numbers Revisited}
\label{sec:Fib}

In \cite{moz22}, we used the decorated super Teichm\"{u}ller space of an annulus to find a sequence of $\lambda$-lengths 
satisfying a recurrence which generalizes the Fibonacci sequence. This is in the same spirit as Ovsienko's ``\emph{shadow sequences}'' \cite{ovsienko_22},
although our shadow of the Fibonacci sequence differs from his\footnote{Note that Ovsienko's shadow sequence for the Fibonacci numbers actually coincides with the $g_n$'s defined in \cite[Sec. 11]{moz22}.}.

In this section, we revisit these ``\emph{super Fibonacci numbers}'' from the point of view of the matrix formulas presented in the current paper.

\begin{figure}
\begin{tikzpicture}[scale=0.7]
    \node () at (-1,-1.5) {$\sigma$};
    \node () at (1,-1) {$\theta$};

    \draw (0,0) circle (1);
    \draw (0,0) circle (2.5);

    \draw [-->-, blue] (0,-1) -- (0,-2.5);

    \draw[-->-, color=red, domain=0:1, samples=100] plot ({(2.5-1.5*\x)*cos(2*pi*\x r - pi/2 r}, {(2.5-1.5*\x)*sin(2*pi*\x r - pi/2 r)});

    \draw[densely dashed, domain=0:1, samples=100] plot ({(2.5-1.5*\x)*cos(6*pi*\x r - pi/2 r}, {(2.5-1.5*\x)*sin(6*pi*\x r - pi/2 r)});
\end{tikzpicture}
\begin{tikzpicture}[scale=0.6,every node/.style={sloped,allow upside down}]
    \draw (0,-2.5) node {};

    \draw (-7.5,2)  -- (7.5,2);
    \draw (-7.5,-2) -- (7.5,-2);

    \foreach \x in {-6, -2, 2}
        \draw (\x+0.5, 2-0.5) node {$\theta$};

    \foreach \x in {-2, 2, 6}
        \draw (\x-0.5, -2+0.5) node {$\sigma$};

    \draw (-7.75, 0) node {$\cdots$};
    \draw (7.75,  0)  node {$\cdots$};

    \foreach \x in {-6, -2, 2, 6}
        \draw [style=blue] (\x,2) to node {\midarrow}  (\x,-2);
    \foreach \x in {-6, -2, 2}
        \draw [color=red] (\x,-2) -- node {\midarrow} (\x+4,2);

    \foreach \x in {-6, -2, 2, 6} {
        \draw [fill=black] (\x,2) circle (0.05);
        \draw [fill=black] (\x,-2) circle (0.05);
    }

    \draw ($(-6,-2) + (60:0.4)$)  --  ($(-2,2) + (-150:0.4)$);   
    \foreach \x in {-2, 2} {
        \draw ($(\x,2) + (-150:0.4)$) arc (-150:-105:0.4);           
        \draw ($(\x,2) + (-105:0.4)$) --  ($(\x,-2) + (105:0.4)$);   
        \draw ($(\x,-2) + (105:0.4)$) arc (105:60:0.4);              
        \draw ($(\x,-2) + (60:0.4)$)  --  ($(\x+4,2) + (-150:0.4)$);   
    }
\end{tikzpicture}
\caption{Left: Triangulation of an annulus. Right: its universal cover.}
\label{fig:annulus}
\end{figure}
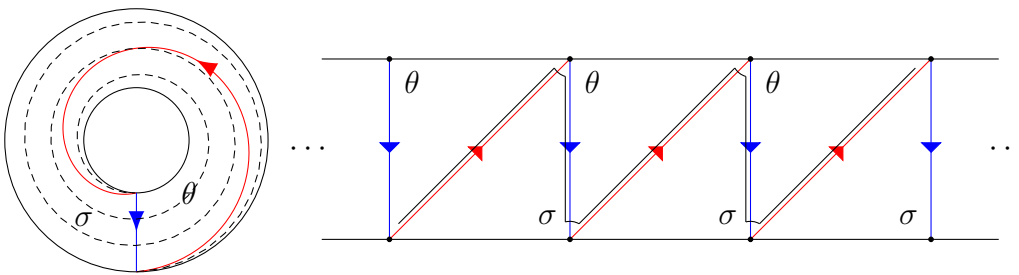

Consider an annulus with one marked point on each boundary component, and the oriented triangulation pictured in \Cref{fig:annulus},
where all $\lambda$-lengths are equal to $1$.
Let $z_n$ be $\lambda$-length of the arc connecting the two marked points which winds around the annulus $n-2$ times.
That is, $z_1=z_2=1$ are the diagonals of the triangulation (in blue and red in \Cref{fig:annulus}), $z_3$ circles once, $z_4$ circles twice, etc.
For example, $z_4$ is shown as a dashed line in the left picture of \Cref{fig:annulus}.
In \cite{moz22}, we showed that these satisfy the recurrence
\[ z_n = (3+2 \sigma \theta) z_{n-1} - z_{n-2} - \sigma \theta, \]
which generalizes a recurrence satisfied by the sequence of every-other Fibonacci number, i.e.
\[ f_n = 3 f_{n-2} - f_{n-4} \]
Furthermore, $z_n$ ``counts'' double dimer configurations on the snake graph $G_{2n-5}$, which consists of a horizontal row of $2n-5$ boxes.
By this we mean $z_n$ is the weight generating function of double dimer configurations, and since all even variables are set equal to $1$,
it is of the form $z_n = x_{2n-5} + y_{2n-5} \sigma \theta$, where $x_k$ and $y_k$ are integers. Specifically, $x_k = f_{k+1}$ counts those double dimer
configurations on the horizontal strip of $k$ boxes that are really single dimer covers (i.e. all edges are doubled), and $y_k$ counts all double dimer configurations
which contain one cycle surrounding an odd number of boxes\footnote{The number $y_k$ is really the weighted sum of all configurations which
contain cycles, but those which have either more than one cycle or a cycle of even length have weight zero. This is because the same two 
odd variables appear repeatedly. See \cite{moz22} for details.}.

As in \cite{moz22}, let $w_n = x_{2n-4} + y_{2n-4} \sigma\theta$ be the corresponding generating function for double dimer covers on
horizontal strips with an even number of boxes. These two sequences satisfy the following Fibonacci-like recurrences.

\begin {lemma} [\cite{moz22}, Lemma 11.6] \label{lem:fib_recurrence}
    The sequences $z_n$ and $w_n$ satisfy the following recurrences:
    \begin {itemize}
        \item[(a)] $\displaystyle z_n = z_{n-1} + (1 + \sigma \theta) w_{n-1}$ 
        \item[(b)] $\displaystyle w_n = w_{n-1} + (1 + \sigma \theta) z_n - \sigma \theta$
    \end {itemize}
\end {lemma}

We will use the matrix product method to calculate these numbers, utilizing a path as shown on the right side of \Cref{fig:annulus}.
Although \Cref{thm:generic} was only stated for polygons, we can lift the triangulation to the universal cover to obtain a polygon
with a zig-zag triangulation (in the default orientation), in which the appropriate odd elements are identified.

Since all $\lambda$-lengths are $1$,
let us abbreviate $E:= E(1)$ and $A_\theta := A(1|\theta)$. The figure shows the example of $z_4$. In general, the holonomy matrix product for $z_n$ is
\[ H_n = X^{n-2} E^{-1}, \quad \quad \text{where} \quad \quad X := E^{-1} A_\theta^{-1} \rho E A_\sigma \rho \]

\begin {lemma} \label{lem:fibonacci_holonomy}
    The holonomy $H_{n}$, realizing the arc $z_n$, is given by
    \begin {align*} 
        H_{n} &= \ospmatrix
                 {-w_{n-1}}                           {z_n}                           {(z_n-1)\sigma + w_{n-1}\theta}
                 {-z_{n-1}}                           {w_{n-1}}                       {(z_{n-1}-1)\theta + w_{n-1}\sigma}
                 {(z_{n-1}-1)\sigma - w_{n-1}\theta} {(z_n-1)\theta - w_{n-1}\sigma}  {1 - (\ell_{2n-4}-2)\sigma\theta} \\
              &= \ospmatrix
                 {-(x_{2n-6}+y_{2n-6}\sigma\theta)}                           {x_{2n-5}+y_{2n-5}\sigma\theta}     {(x_{2n-5}-1)\sigma + x_{2n-6}\theta}
                 {-(x_{2n-7}+y_{2n-7}\sigma\theta)}   {x_{2n-6}+y_{2n-6}\sigma\theta}     {(x_{2n-7}-1)\theta + x_{2n-6}\sigma}
                 {(x_{2n-7}-1)\sigma - x_{2n-6}\theta} {(x_{2n-5}-1)\theta - x_{2n-6}\sigma} {1 - (\ell_{2n-4}-2)\sigma\theta}
    \end {align*}
\end {lemma}
Here $\ell_k$ denotes the $k$th Lucas number defined by $\ell_1 =1$, $\ell_2 = 3$, and $\ell_k = \ell_{k-1} + \ell_{k-2}$ for $k\geq 3$.
\footnote{This quantity $\ell_{2n-4}-2$ also equals $\kappa(W_{n-2})$, the number of spanning trees on the wheel graph with $(n-1)$ vertices (see \cite{myers1971}).} 

\begin {proof}
    Consider the polygon in the universal cover of the cylinder surrounding the canonical path pictured in \Cref{fig:annulus}.
    Note that this polygon, which has $2n-2$ vertices, has a zig-zag triangulation with the default orientation. 
    Furthermore, the holonomy has type $01$ (in the sense of \Cref{def:H_ab}). This ensures that the signs match those in \Cref{thm:combo}.

    Label the vertices $0,1,\dots,2n-1$, as in \Cref{sec:zig-zag}
    (note that since this picture is on the universal cover, $i = i+2$).
    The snake graph $G_T$ assicated to this zig-zag triangulation is $G_{2n-5}$, a horizontal row of $2n-5$ boxes. Therefore the snake graph $G_{\widetilde{T}}$
    described in \Cref{sec:DD} is $G_{2n-3}$. Since, as we mentioned above, this holonomy is of type $01$, and since $G_{2n-3}$
    has an odd number of boxes, we will always have $e_0$ (resp. $e_1$) on the bottom (resp. left) side of the first tile, 
    and $e_N$ (resp. $e_{N+1}$) on the right (resp. top) of the last tile.
    \Cref{thm:combo} gives formulas for all nine entries of $H_n$ in terms of subsets of
    double dimer covers of $G_{\widetilde{T}} = G_{2n-3}$. Note that since we set all $\lambda$-lengths equal to 1, we can ignore the denominators in these expressions.

    The following picture is the specialization of \Cref{fig:dd-entries} to the situation $\widetilde{G} = G_{2n-3}$. 

    
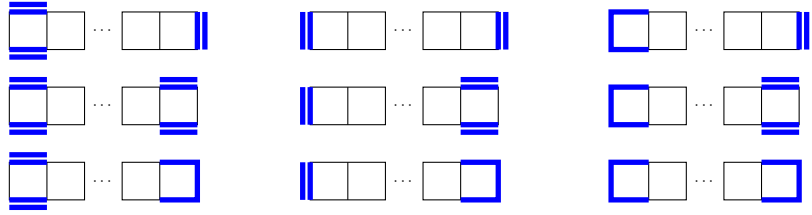
\begin {center}
    \begin {tikzpicture}[scale=0.5, every node/.style={scale=0.5}]
        \foreach \x/\y in {0/0, 8/0, 16/0, 0/-2, 8/-2, 16/-2, 0/-4, 8/-4, 16/-4} {
            \foreach \s in {0,3} {
                \draw (\x+\s,\y) -- (\x+\s+2,\y) -- (\x+\s+2,\y+1) -- (\x+\s,\y+1) -- cycle;
                \draw (\x+\s+1,\y) -- (\x+\s+1,\y+1);
            }
            \draw (\x+2.5,\y+0.5) node {$\cdots$};
        }

        \draw[blue, line width=2] (0,0) -- (1,0);
        \draw[blue, line width=2] (0,-0.2) -- (1,-0.2);
        \draw[blue, line width=2] (0,1) -- (1,1);
        \draw[blue, line width=2] (0,1.2) -- (1,1.2);
        \draw[blue, line width=2] (5,0) -- (5,1);
        \draw[blue, line width=2] (5.2,0) -- (5.2,1);

        \begin {scope}[shift={(8,0)}]
            \draw[blue, line width=2] (0,0) -- (0,1);
            \draw[blue, line width=2] (-0.2,0) -- (-0.2,1);
            \draw[blue, line width=2] (5,0) -- (5,1);
            \draw[blue, line width=2] (5.2,0) -- (5.2,1);
        \end {scope}

        \begin {scope}[shift={(16,0)}]
            \draw[blue, line width=2] (5,0) -- (5,1);
            \draw[blue, line width=2] (5.2,0) -- (5.2,1);
            \draw[blue, line width=2] (1,1) -- (0,1) -- (0,0) -- (1,0);
        \end {scope}

        \begin {scope}[shift={(0,-2)}]
            \draw[blue, line width=2] (0,0) -- (1,0);
            \draw[blue, line width=2] (0,-0.2) -- (1,-0.2);
            \draw[blue, line width=2] (0,1) -- (1,1);
            \draw[blue, line width=2] (0,1.2) -- (1,1.2);
            \draw[blue, line width=2] (4,0) -- (5,0);
            \draw[blue, line width=2] (4,-0.2) -- (5,-0.2);
            \draw[blue, line width=2] (4,1) -- (5,1);
            \draw[blue, line width=2] (4,1.2) -- (5,1.2);
        \end {scope}

        \begin {scope}[shift={(8,-2)}]
            \draw[blue, line width=2] (0,0) -- (0,1);
            \draw[blue, line width=2] (-0.2,0) -- (-0.2,1);
            \draw[blue, line width=2] (4,0) -- (5,0);
            \draw[blue, line width=2] (4,-0.2) -- (5,-0.2);
            \draw[blue, line width=2] (4,1) -- (5,1);
            \draw[blue, line width=2] (4,1.2) -- (5,1.2);
        \end {scope}

        \begin {scope}[shift={(16,-2)}]
            \draw[blue, line width=2] (4,0) -- (5,0);
            \draw[blue, line width=2] (4,-0.2) -- (5,-0.2);
            \draw[blue, line width=2] (4,1) -- (5,1);
            \draw[blue, line width=2] (4,1.2) -- (5,1.2);
            \draw[blue, line width=2] (1,1) -- (0,1) -- (0,0) -- (1,0);
        \end {scope}

        \begin {scope}[shift={(0,-4)}]
            \draw[blue, line width=2] (0,0) -- (1,0);
            \draw[blue, line width=2] (0,-0.2) -- (1,-0.2);
            \draw[blue, line width=2] (0,1) -- (1,1);
            \draw[blue, line width=2] (0,1.2) -- (1,1.2);
            \draw[blue, line width=2] (4,1) -- (5,1) -- (5,0) -- (4,0);
        \end {scope}

        \begin {scope}[shift={(8,-4)}]
            \draw[blue, line width=2] (0,0) -- (0,1);
            \draw[blue, line width=2] (-0.2,0) -- (-0.2,1);
            \draw[blue, line width=2] (4,1) -- (5,1) -- (5,0) -- (4,0);
        \end {scope}

        \begin {scope}[shift={(16,-4)}]
            \draw[blue, line width=2] (1,1) -- (0,1) -- (0,0) -- (1,0);
            \draw[blue, line width=2] (4,1) -- (5,1) -- (5,0) -- (4,0);
        \end {scope}
    \end {tikzpicture}
    \end {center}

    The $(1,2)$-entry (resp. the $(2,1)$-entry) has an obvious bijection with double dimer covers of $G_{2n-5}$ (resp. $G_{2n-7}$), and hence
    by a result from \cite{moz22} is equal to $z_n$ (resp. $z_{n-1}$). Similarly, the $(1,1)$ and $(2,2)$ entries have obvious bijections with
    double dimer covers of $G_{2n-6}$, and so are equal to $w_{n-1}$.

    The $(1,3)$-entry has a bijection with double dimer covers of $G_{2n-4}$ containing a cycle which surrounds (at least) the first square. The
    first $\mu$-invariant in $\widetilde{T}$ is $\theta$, so cycles which surround an odd number of squares will end with a $\sigma$, and those
    surrounding an even number of squares will end with $\theta$. According to \Cref{thm:combo}, we additionally need to 
    toggle the first $\theta$, meaning the weights will contain only the $\mu$-invariant at the \emph{end} of the cycle. If the first cycle
    surrounds $k$ tiles, the rest can be any double dimer configuration on the complement of the first $k+1$ tiles (i.e. on $G_{2n-5-k}$).
    When $k$ is odd, this is counted by $w_{n-\frac{k+1}{2}}$, and when $n$ is even by $z_{n - \frac{k}{2}}$.
    Thus, we have that the $(1,3)$-entry is given by
    \[ (w_{n-1} + w_{n-2} + \cdots + w_2) \sigma + (z_{n-1} + z_{n-2} + \cdots + z_2) \theta. \]
    Note that each $w_i$ or $z_i$ is of the form $x + y \sigma \theta$, so multiplying by either $\sigma$ or $\theta$ annihilates the $y$-term,
    and so we may replace each $z_i$ with $x_{2i-5}$ and each $w_i$ with $x_{2i-4}$, which are simply Fibonacci numbers.
    The following are well-known identities of Fibonacci numbers:
    \[ x_0 + x_2 + \cdots + x_{2n-6} = x_{2n-5} - 1 \quad \text{and} \quad x_{-1} + x_1 + x_3 + \cdots + x_{2n-7} = x_{2n-6} \]
    The arguments for the $(2,3)$, $(3,1)$, and $(3,2)$-entries are similar. The signs in the bottom row come from \Cref{rmk:local_signs}.
    In particular, this corresponds to a zig-zag triangulation where half of the odd variables are all set equal to $\sigma$ and the other half
    all set equal to $\theta$. Since we negate half of the odd variables in \Cref{rmk:local_signs}, this means simply negating either $\sigma$ or $\theta$.

    For the (3,3)-entry, we use standard identities on Fibonacci and Lucas numbers:  In particular, by \Cref{eq:osp1}, 
    we can express the (3,3)-entry as $1 + \alpha\beta$, where $\alpha$ is the (3,1)-entry and $\beta$ is the (3,2)-entry.  
    We thus obtain 
    \begin {align*} 
        H_n(3,3) &= 1 + \bigg((x_{2n-7}-1)\sigma - x_{2n-6}\theta \bigg)  \bigg((x_{2n-5} - 1)\theta - x_{2n-6}\sigma\bigg) \\
                 &= 1 - \bigg( x_{2n-6}^2 - x_{2n-5}x_{2n-7} + x_{2n-5} + x_{2n-7} - 1 \bigg)\sigma\theta 
    \end {align*}
    By Cassini's identity on Fibonacci numbers, we have the equality $x_{2n-5} x_{2n-7} - x_{2n-6}^2 = 1$ and another standard identity 
    relating Fibonacci and Lucas numbers is $\ell_{k} = x_{k-3} + x_{k-1}$. 
    Using this in the case of $k=2n-4$, we rewrite $H_n(3,3) = 1 + (\ell_{2n-4}-2)\theta\sigma = 1 + (\kappa(W_{n-3}) - 2)\theta\sigma$ as desired.
\end {proof}

\begin {remark} \label{rem:Conj11.19}
    In \cite{moz22}, we showed that $z_n$, defined as the generating function of double dimer covers of $G_{2n-5}$, is equal to the $\lambda$-length
    of certain arcs in an annulus, as described above. In the previous paper, we defined the $w_n$'s as double dimer generating functions for $G_{2n-4}$, 
    and used them for algebraic calculations, but did not give a geometric interpretation. In Remark 11.17 and Conjecture 11.19 in \cite{moz22}, we suggested
    that the $w_n$'s should have a particular geometric meaning, which we are now able to verify.
    Comparing $H_n$ from \Cref{lem:fibonacci_holonomy} with $H_{ab}$ from \Cref{thm:generic},
    we see that the $w_n$'s can also be interpreted as certain $\lambda$-lengths. In particular, $w_n$ is the $\lambda$-length of an arc which
    winds around the annulus $n-2$ times, and has both endpoints on the same boundary component.
\end {remark}

\section {Geometric Interpretation} \label{sec:appendix}

In this section, we will describe a more geometric interpretation of the results of this paper, in terms of 
the definitions of decorated super Teichm\"{u}ller spaces given in \cite{pz_19}. We take a viewpoint along the
lines of \cite{fg_06} (section 11), which is slightly different than the viewpoint of \cite{pz_19}. Some version of this viewpoint
was also present in \cite{bouschbacher_13}, applied to super shear coordinates (rather than $\lambda$-lengths).

This alternative viewpoint is as follows. Let $\mathcal{A}$ be the commutative super-algebra generated by
$\lambda$-lengths and $\mu$-invariants coming from some particular triangulation. The definitions of Penner and
Zeitlin in \cite{pz_19} are in terms of the Minkowski geometry of the free module $\mathcal{A}^{3|2}$, whereas our
alternate viewpoint will be to instead view the elements of $\osp(1|2)$ acting on the free module $\mathcal{A}^{2|1}$
(by ordinary matrix-vector multiplication).

To connect the two viewpoints, we will give a map $s \colon \mathcal{A}^{2|1} \to \mathcal{A}^{3|2}$,
which is equivariant with respect to the actions. This way, statements about the action may be stated in $\mathcal{A}^{2|1}$, where the action is simpler,
and the statements will transfer over to $\mathcal{A}^{3|2}$.

In particular, our main goal in this section is to use this viewpoint to give geometric interpretations for the matrices $E(\lambda)$ and $A(h|\theta)$,
and thus also the holonomy matrices $H_{a,b}$ of \Cref{def:H_ab} and \Cref{thm:generic}.

The map $s \colon \mathcal{A}^{2|1} \to \mathcal{A}^{3|2}$, mentioned above, is given explicitly as follows:
\[ s(x,y|\phi) := \left( y^2-x^2, \, -2xy, \, y^2+x^2 \, \middle| \, -2y\phi, \, 2x\phi \right) \frac{1}{\sqrt{2}}. \]

\begin {prop} \label{prop:s_is_equivariant}
    The map $s \colon \mathcal{A}^{2|1} \to \mathcal{A}^{3|2}$ is equivariant with respect to the $\osp(1|2)$-actions.
\end {prop}
\begin {proof}
    In \cite[section 1]{pz_19}, each element $v \in \mathcal{A}^{3|2}$ was identified with a $2|1$-by-$2|1$ matrix $Q(v)$, and the action
    of $\osp(1|2)$ on $\mathcal{A}^{3|2}$ was defined by
    \[ g \cdot Q(v) = (g^{-1})^{\mathrm{st}} Q(v) g^{-1}. \]
    Therefore one just needs to verify that 
    \[ Q(s(g \cdot v)) = (g^{-1})^{\mathrm{st}} \, Q(s(v)) \, g^{-1}.\]
    We omit this calculation, as it is straightforward.
\end {proof}

\begin {remark} \label{rmk:plus-or-minus}
    Note that all entries of $s(x,y|\phi)$ are homogeneous quadratic expressions in $x$, $y$, and $\phi$. It follows that $s(v) = s(-v)$ for all $v \in \mathcal{A}^{2|1}$.
    We may therefore think of the domain of the map $s$ as the quotient $\mathcal{A}^{2|1} / \pm 1$, where vectors are identified with their negatives.
\end {remark}

In \cite{pz_19}, the $\lambda$-lengths were defined in terms of a certain Minkowski inner product on $\mathcal{A}^{3|2}$.
Specifically, they define $\lambda(a,b) := \sqrt{\left<a,b\right>}$.
Under this map $s$, the $\lambda$-lengths correspond to a certain bilinear form, which we now describe.

\begin {definition}
    Define the map $\omega \colon \mathcal{A}^{2|1} \times \mathcal{A}^{2|1} \to \mathcal{A}$ as follows. If $v = (a,b|\phi)$ and $w=(x,y|\theta)$, then
    \[ \omega(v,w) := ay - bx + \phi \theta.\]
\end {definition}

\begin {remark}
    The map $\omega$ is the bilinear form defined by the matrix $J$ (from \Cref{sec:osp}), so by definition, $\osp(1|2)$ is the group which preserves $\omega$.
\end {remark}

\begin {prop} \label{prop:omega_is_lambda}
    For any $v,w \in \mathcal{A}^{2|1}$, we have $\left< s(v), s(w) \right> = \omega(v,w)^2$. 
    In particular, 
    \[ \lambda(s(v), s(w)) = \sqrt{\omega(v,w)^2} = \pm \, \omega(v,w). \]
\end {prop}

\begin {lemma} \label{lem:first_two_columns}
    Let $v_1,v_2 \in \mathcal{A}^{2|1}$ with $\omega(v_1,v_2) = 1$. Then there is a unique matrix $g \in \osp(1|2)$
    whose first two columns are $v_1$ and $v_2$.
\end {lemma}
\begin {proof}
    Let $v_1 = (a,b|\phi)$ and $v_2 = (x,y|\theta)$.
    We want to show there is some $v_3 = (\alpha,\beta,f)$ such that
    \[ \left( \begin{array}{cc|c} a & x & \alpha \\ b & y & \beta \\ \hline \phi & \theta & f \end{array} \right) \in \osp(1|2). \]
    If such a $v_3$ exists, it is unique by the relations defining $\osp(1|2)$. In particular, equations (\ref{eq:osp1}), (\ref{eq:osp5}), and (\ref{eq:osp6}) 
    say that $\alpha$, $\beta$, and $f$ are determined by $v_1$ and $v_2$ as follows:
    \[ f = 1 + \phi \theta, \quad \alpha = a \theta - x \phi, \quad \beta = b \theta - y \phi \]
    However, the entries of the matrix must also satisfy equations (\ref{eq:osp2}), (\ref{eq:osp3}), and (\ref{eq:osp4}),
    so we must check that defining $v_3$ by the formulas above is consistent with these other three equations.

    First is equation (\ref{eq:osp2}), which says that $ay-bx = 1-\phi\theta$, or equivalently $ay-bx + \phi\theta = 1$,
    which is precisely our assumption that $\omega(v_1,v_2) = 1$. So we see this assumption is necessary in order to satisfy equation (\ref{eq:osp2}).

    Next, equation (\ref{eq:osp3}) requires that $\phi = b \alpha - a \beta$. The right-hand side is 
    \[ b \alpha - a \beta = b(a \theta - x \phi) - a(b \theta - y \phi) = (ay-bx) \phi. \]
    But we already know from the discussion above that $ay-bx = 1 - \phi \theta$. Substituting this gives the desired result.

    The calculation verifying equation (\ref{eq:osp4}), i.e. $\theta = y \alpha - x \beta$, is similar.
\end {proof}

\begin {definition}
    Let $\mathcal{A}^{2|1}_+$ be the set of vectors with non-zero bodies. That is
    \[ \mathcal{A}^{2|1}_+ := \{ (x,y|\phi) \in \mathcal{A}^{2|1} ~|~ x \neq 0 \text{ or } y \neq 0 \} \]
\end {definition}

\begin {corollary} \label{cor:transitive}
    The action of $\osp(1|2)$ on $\mathcal{A}^{2|1}_+$ is transitive.
\end {corollary}
\begin {proof}
    Let $v = (x,y|\phi) \in \mathcal{A}^{2|1}_+$. If there exists some $w$ with $\omega(v,w) = 1$, then \Cref{lem:first_two_columns}
    tells us that there is some $g \in \osp(1|2)$ with $g \cdot e_1 = v$. If $x \neq 0$, we can choose $w = (0, x^{-1} | \phi)$,
    and if $y \neq 0$, we can choose $w = (-y^{-1}, 0 | \phi)$.
\end {proof}

\begin {remark}
    The image of $\mathcal{A}^{2|1}_+$ under the map $s$ was called the ``\emph{special light cone}'' $L_0^+$ in \cite{pz_19}, 
    and the decorated super Teichm\"{u}ller space of a polygon is the configuration space of tuples of points in this set, modulo
    the diagonal action of $\osp(1|2)$. From the point of view described in this section, we instead consider
    configurations of points in $\mathcal{A}^{2|1}_+$, up to diagonal action of $\osp(1|2)$. Also, by \Cref{rmk:plus-or-minus},
    we identify $v = -v$ in $\mathcal{A}^{2|1}_+$, since $s(v) = s(-v)$.
\end {remark}

The following is a kind of standard form result (which can be seen in the argument used in the proof of Lemma 3.1 from \cite{pz_19}).
In order to state it, we first define the vectors $u := s(e_1)$ and $w = s(e_2)$ in $\mathcal{A}^{3|2}$.

\begin {prop} \label{prop:standard_form_pair_appendix}
    Let $p,q \in L_0^+$ be two points in the special light cone with $\lambda = \lambda(p,q)$. 
    There is some $g \in \osp(1|2)$ such that 
    \[ g \cdot p = u \cdot \lambda^2  \quad \text{and} \quad  g \cdot q = w \]
    Moreover, $g$ is unique up to post-composition (i.e. left-multiplication) by $\rho$.
\end {prop}
\begin {proof}
    We will construct the inverse of the matrix $g$.
    Let $v_p,v_q \in \mathcal{A}^{2|1}$ be vectors such that $s(v_p) = p$ and $s(v_q) = q$. Then by \Cref{prop:omega_is_lambda},
    we know that $\omega(v_p,v_q) = \pm \lambda$. Since $\lambda$ has positive body, it is invertible, and we may replace
    $v_p$ with $v_p' = \pm \, v_p \cdot \frac{1}{\lambda}$, so that $\omega(v_p',v_q) = 1$. Thus by \Cref{lem:first_two_columns}, 
    there is a matrix $g \in \osp(1|2)$ whose first two columns are are $v_p'$ and $v_q$. If $e_1$, $e_2$, $\varepsilon$
    are the standard basis vectors of $\mathcal{A}^{2|1}$, then this matrix acts by $g \cdot e_1 = v_p'$ and $g \cdot e_2 = v_q$.
    \Cref{prop:s_is_equivariant} says the map $s$ is equivariant, so $g \cdot w = q$ and
    $g \cdot u = p \cdot \frac{1}{\lambda^2}$. Then clearly $g^{-1}$ is the matrix described in the proposition.

    Finally, we remark that the choices of $v_p$ and $v_q$ were not unique, since we may replace either (or both) by their negatives (\Cref{rmk:plus-or-minus}).
    If we replace $v_q \mapsto -v_q$, then when we re-scale $v_p$ to get $v_p'$,
    it will also be negated. The overall effect is that the first two columns of $g$ will be negated (but the third column will remain the same).
    This is the same as right-multiplication by $\rho$. Since the matrix from the statement is $g^{-1}$,
    it will be unique up to left-multiplication by $\rho$.
\end {proof}

\begin {prop} \label{prop:E_matrix_appendix}
    Let $p,q \in L_0^+$ be in the standard form guaranteed by \Cref{prop:standard_form_pair_appendix}. That is, if $\lambda = \lambda(p,q)$, assume
    that $q = w$ and $p = u \cdot \lambda^2$. Then $E(\lambda)$ and $E(-\lambda) = E(\lambda)^{-1}$ are the only
    matrices $g \in \osp(1|2)$ such that $g \cdot p = q$ and $g \cdot q = p$.
\end {prop}
\begin {proof}
    By \Cref{prop:standard_form_pair_appendix}, such a matrix is unique up to $\rho$. Note that $E(-\lambda) = \rho \, E(\lambda)$,
    so if we show that $E(\lambda)$ satisfies the conditions, then we are done. 

    If $e_1,e_2,\varepsilon$ are the standard basis vectors of $\mathcal{A}^{2|1}$, then looking at the columns of $E(\lambda)$ shows that
    \[ E(\lambda) \cdot e_1 = e_2 \cdot \lambda^{-1} \quad \quad \text{and} \quad \quad E(\lambda) \cdot e_2 = -e_1 \cdot \lambda \]
    Re-arranging the first equation gives $E(\lambda) \cdot (e_1 \cdot \lambda) = e_2$. Applying $s$ to these equations, remembering
    that $s$ is equivariant, and that $s(v \cdot \alpha) = s(v) \cdot \alpha^2$, we get
    \[ E(\lambda) \cdot (u \cdot \lambda^2) = w \quad \quad \text{and} \quad \quad E(\lambda) \cdot w = u \cdot \lambda^2 \qedhere\]
\end {proof}

To prove the corresponding result about the $A(h|\theta)$ matrices, we will use the following lemma.

\begin {lemma} \label{lem:third_point_coordinates}
    Let $p_1,p_2,p_3 \in L_0^+$, with $\lambda(p_i,p_j) = \lambda_{ij}$, and $\mu$-invariant $\theta$. 
    Suppose the edge $p_1,p_3$ is in the standard form guaranteed by \Cref{prop:standard_form_pair_appendix}. 
    That is, assume that $p_1 = u \cdot \lambda_{13}^2$ and $p_3 = w$. Then $p_2 = s(z)$, where
    \begin {itemize}
        \item If $p_1, p_2, p_3$ appear in clockwise order, then $z = \left(\lambda_{23}, \; \frac{\lambda_{12}}{\lambda_{13}} ~ \middle| ~ \pm \td{2}{1}{3} \right)$. 
        \item If $p_1, p_2, p_3$ appear in counter-clockwise order, then $z = \left(\lambda_{23}, \; - \frac{\lambda_{12}}{\lambda_{13}} ~ \middle| ~ \pm \td{2}{1}{3} \right)$. 
    \end {itemize}
    In both cases, the sign of $\td{2}{1}{3}$ cannot be determined from this information alone.
\end {lemma}
\begin {proof}
    We will show the calculation for $(a)$, and that of $(b)$ is similar. 
    Suppose that $p_1,p_2,p_3$ are in clockwise order. This means they are a ``negative triple'' (in the language of \cite{pz_19}). 
    Let $p_2 = s(z)$ for $z = (x,y|\alpha) \in \mathcal{A}^{2|1}$. Because this is a negative triple, the signs of $x$ and $y$ must be the same.
    By \Cref{prop:omega_is_lambda}, we can deduce that
    \[ \lambda_{12} = \lambda(u \cdot \lambda_{13}^2, p_2) = \pm \omega(e_1 \cdot \lambda_{13}, z) = \pm y \cdot \lambda_{13} \]
    \[ \lambda_{23} = \lambda(p_2, w) = \pm \omega(z, e_2) = \pm x \]
    Since $s(z) = s(-z)$, we may choose $z$ so that $x = \lambda_{23}$ and $y = \frac{\lambda_{12}}{\lambda_{13}}$.
    In order to compare $\alpha$ with the $\mu$-invariant $\theta$, we need to put the triple $p_1,p_2,p_3$ into the standard form 
    described in Lemma 3.3 and Lemma 3.5 of \cite{pz_19}.
    This can be done with the following matrix:
    \[ g = \ospmatrix {\sqrt{h^3_{12}}}{0}{0} {0}{\frac{1}{\sqrt{h^3_{12}}}}{0} {0}{0}{1} \]
    The effect on $z$ is then
    \[ g \cdot z = \left( 1, 1 ~\middle|~ \sqrt{h^2_{13}} \; \alpha \right) \cdot \frac{1}{\sqrt{h^2_{13}}} \]
    By the definition of $\mu$-invariant from \cite{pz_19}, we see that $\alpha = \frac{\pm \theta}{\sqrt{h^2_{13}}} = \pm \td{2}{1}{3}$.
    By Corollary 3.4 in \cite{pz_19}, knowing the three points $p_1$, $p_2$, $p_3$ only determines $\theta$ up to sign.

    The proof of $(b)$ is the the same, except $x$ and $y$ have opposite signs, and there are some signs in the diagonal matrix $g$.
\end {proof}

The following proposition is essentially a restatement of what is called the ``\emph{basic calculation}'' in section 4 of \cite{pz_19},
but phrased in a way that highlights the significance of the $A$-matrices we defined in \Cref{sec:Mpaths}.

\begin {prop} \label{prop:A_matrix_appendix}
    Let $p_i, p_\ell, p_k \in L_0^+$ in clockwise order, as in \Cref{fig:super_ptolemy}, with $\lambda$-lengths $a,b,e$ and $\mu$-invariant $\theta$, 
    and suppose the edge $p_i,p_k$ is in the standard form guaranteed by \Cref{prop:standard_form_pair_appendix}. 
    That is, $p_i = u \cdot e^2$ and $p_k = w$. Then there is a unique point $p_j \in L_0^+$ such that 
    \begin {itemize}
        \item[$(a)$] The triangle $p_i, p_j, p_k$ has the $\lambda$-lengths and $\mu$-invariant as in \Cref{fig:super_ptolemy}.
        \item[$(b)$] The point $p_j$ is defined by $A^k_{ij} \cdot p_j = u \cdot c^2$ (or $A^k_{ij}\rho \cdot p_j = u \cdot c^2$), depending on
                     the orientation of the edge $(i,k)$. In other words, $A^k_{ij}$ (or $A^k_{ij} \rho$) puts the edge $(j,k)$ into standard position.
    \end {itemize}
\end {prop}
\begin {proof}
    Part $(b)$ completely determines some point. We simply need to see that it satisfies the claim of part $(a)$. Let us consider the case pictured in
    \Cref{fig:super_ptolemy}, where the ege labeled ``$e$'' is directed $p_i \to p_k$. If $p_j = s(v)$, then $v$ must be (up to sign)
    the first column of $\rho {A^k_{ij}}^{-1}$, multiplied by $c$, which gives
    \[ v = \left( -c, \; \frac{d}{e} ~ \middle| ~ - \td{j}{i}{k} \right) \]
    Now, it is straightforward to check, using \Cref{prop:omega_is_lambda}, that
    \[ \lambda_{ij} = \omega(v,e_1) \cdot e = d \quad \text{and} \quad  \lambda_{jk} = \omega(v,e_2) = c, \]
    and by the same calculation done in \Cref{lem:third_point_coordinates}, we see that the $\mu$-invariant $\boxed{ijk}$ is $\sigma$.
\end {proof}

We conclude with a discussion of how the holonomy matrices $H_{a,b}$ from \Cref{def:H_ab} and \Cref{thm:generic} can be interpreted in this
geometric context. The main point is that we may represent a polygon as a configuration of points in
$\mathcal{A}^{2|1}$ or $\mathcal{A}^{3|2}$, in such a way so that if the first edge $(c_0,c_1)$ of the canonical path is in standard position,
then the matrix $H_{ab}$ is the transformation which puts the final edge $(c_N,c_{N+1})$ into standard position. Indeed, we can build this configuration
of points inductively. Start by choosing three points for the first triangle using \Cref{prop:standard_form_pair_appendix} and \Cref{lem:third_point_coordinates}.
Then \Cref{prop:A_matrix_appendix} tells us how to choose the third point of the next triangle (which shares one side with the first one).
Then we may use some product of $E$ or $A$ matrices to put the appropriate edge in standard position and continually use \Cref{prop:A_matrix_appendix}
to choose each subsequent point.

\section* {Acknowledgments}

The authors would like to acknowledge the support of the NSF grants DMS-1745638 and DMS-1854162.

\bibliographystyle {alpha}
\bibliography {main}

\end{document}